\documentclass[11pt,reqno]{amsart}

\usepackage{amsmath,amsfonts,amsthm,amssymb}
\usepackage{graphicx}
\usepackage{verbatim}
\usepackage{color}
\usepackage{amscd}
\usepackage{verbatim}
\usepackage{mathrsfs}
\usepackage{tikz-cd}
\usepackage[colorlinks,citecolor=red,pagebackref,hypertexnames=false, breaklinks]{hyperref}

\begin{document}
\newcommand{\M}{{\mathcal M}}
\newcommand{\loc}{{\mathrm{loc}}}
\newcommand{\core}{C_0^{\infty}(\Omega)}
\newcommand{\sob}{W^{1,p}(\Omega)}
\newcommand{\sobloc}{W^{1,p}_{\mathrm{loc}}(\Omega)}
\newcommand{\merhav}{{\mathcal D}^{1,p}}
\newcommand{\be}{\begin{equation}}
\newcommand{\ee}{\end{equation}}
\newcommand{\mysection}[1]{\section{#1}\setcounter{equation}{0}}
\newcommand{\laplace}{\Delta}
\newcommand{\pl}{\laplace_p}
\newcommand{\grad}{\nabla}
\newcommand{\pd}{\partial}
\newcommand{\bo}{\pd}
\newcommand{\csub}{\subset \subset}
\newcommand{\sm}{\setminus}
\newcommand{\ol}{\overline}
\newcommand{\ssm}{:}
\newcommand{\diver}{\mathrm{div}\,}
\newcommand{\bea}{\begin{eqnarray}}
\newcommand{\eea}{\end{eqnarray}}
\newcommand{\bean}{\begin{eqnarray*}}
\newcommand{\eean}{\end{eqnarray*}}
\newcommand{\thkl}{\rule[-.5mm]{.3mm}{3mm}}
\newcommand{\cw}{\stackrel{\rightharpoonup}{\rightharpoonup}}
\newcommand{\id}{\operatorname{id}}
\newcommand{\supp}{\operatorname{supp}}
\newcommand{\wlim}{\mbox{ w-lim }}
\newcommand{\mymu}{{x_N^{-p_*}}}
\newcommand{\R}{{\mathbb R}}
\newcommand{\N}{{\mathbb N}}
\newcommand{\Z}{{\mathbb Z}}
\newcommand{\Q}{{\mathbb Q}}
\newcommand{\abs}[1]{\lvert#1\rvert}
\newcommand{\norm}[1]{\left\lVert#1\right\rVert}
\newcommand{\im}{\mathrm{im}}
\newcommand{\kernl}{\mathrm{ker}}
\newtheorem{theorem}{Theorem}[section]
\newtheorem{corollary}[theorem]{Corollary}
\newtheorem{lemma}[theorem]{Lemma}
\newtheorem{cor}[theorem]{Corollary}
\newtheorem{lem}[theorem]{Lemma}
\newtheorem{notation}[theorem]{Notation}
\newtheorem{definition}[theorem]{Definition}
\newtheorem{remark}[theorem]{Remark}
\newtheorem{proposition}[theorem]{Proposition}
\newtheorem{assertion}[theorem]{Assertion}
\newtheorem{problem}[theorem]{Problem}
\newtheorem{conjecture}[theorem]{Conjecture}
\newtheorem{question}[theorem]{Question}
\newtheorem{example}[theorem]{Example}
\newtheorem{Thm}[theorem]{Theorem}
\newtheorem{Lem}[theorem]{Lemma}
\newtheorem{Pro}[theorem]{Proposition}
\newtheorem{Def}[theorem]{Definition}
\newtheorem{Exa}[theorem]{Example}
\newtheorem{Exs}[theorem]{Examples}
\newtheorem{Rems}[theorem]{Remarks}
\newtheorem{Rem}[theorem]{Remark}

\newtheorem{Cor}[theorem]{Corollary}
\newtheorem{Conj}[theorem]{Conjecture}
\newtheorem{Prob}[theorem]{Problem}
\newtheorem{Ques}[theorem]{Question}
\newtheorem*{corollary*}{Corollary}
\newtheorem*{theorem*}{Theorem}
\newcommand{\pf}{\noindent \mbox{{\bf Proof}: }}


\renewcommand{\theequation}{\thesection.\arabic{equation}}
\catcode`@=11 \@addtoreset{equation}{section} \catcode`@=12
\newcommand{\Real}{\mathbb{R}}
\newcommand{\real}{\mathbb{R}}
\newcommand{\Nat}{\mathbb{N}}
\newcommand{\ZZ}{\mathbb{Z}}
\newcommand{\CC}{\mathbb{C}}
\newcommand{\Pess}{\opname{Pess}}
\newcommand{\Proof}{\mbox{\noindent {\bf Proof} \hspace{2mm}}}
\newcommand{\mbinom}[2]{\left (\!\!{\renewcommand{\arraystretch}{0.5}
\mbox{$\begin{array}[c]{c}  #1\\ #2  \end{array}$}}\!\! \right )}
\newcommand{\brang}[1]{\langle #1 \rangle}
\newcommand{\vstrut}[1]{\rule{0mm}{#1mm}}
\newcommand{\rec}[1]{\frac{1}{#1}}
\newcommand{\set}[1]{\{#1\}}
\newcommand{\dist}[2]{$\mbox{\rm dist}\,(#1,#2)$}
\newcommand{\opname}[1]{\mbox{\rm #1}\,}
\newcommand{\mb}[1]{\;\mbox{ #1 }\;}
\newcommand{\undersym}[2]
 {{\renewcommand{\arraystretch}{0.5}  \mbox{$\begin{array}[t]{c}
 #1\\ #2  \end{array}$}}}
\newlength{\wex}  \newlength{\hex}
\newcommand{\understack}[3]{%
 \settowidth{\wex}{\mbox{$#3$}} \settoheight{\hex}{\mbox{$#1$}}
 \hspace{\wex}  \raisebox{-1.2\hex}{\makebox[-\wex][c]{$#2$}}
 \makebox[\wex][c]{$#1$}   }%
\newcommand{\smit}[1]{\mbox{\small \it #1}}
\newcommand{\lgit}[1]{\mbox{\large \it #1}}
\newcommand{\scts}[1]{\scriptstyle #1}
\newcommand{\scss}[1]{\scriptscriptstyle #1}
\newcommand{\txts}[1]{\textstyle #1}
\newcommand{\dsps}[1]{\displaystyle #1}
\newcommand{\dx}{\,\mathrm{d}x}
\newcommand{\dy}{\,\mathrm{d}y}
\newcommand{\dz}{\,\mathrm{d}z}
\newcommand{\dt}{\,\mathrm{d}t}
\newcommand{\dr}{\,\mathrm{d}r}
\newcommand{\du}{\,\mathrm{d}u}
\newcommand{\dv}{\,\mathrm{d}v}
\newcommand{\dV}{\,\mathrm{d}V}
\newcommand{\ds}{\,\mathrm{d}s}
\newcommand{\dS}{\,\mathrm{d}S}
\newcommand{\dk}{\,\mathrm{d}k}

\newcommand{\dphi}{\,\mathrm{d}\phi}
\newcommand{\dtau}{\,\mathrm{d}\tau}
\newcommand{\dxi}{\,\mathrm{d}\xi}
\newcommand{\deta}{\,\mathrm{d}\eta}
\newcommand{\dsigma}{\,\mathrm{d}\sigma}
\newcommand{\dtheta}{\,\mathrm{d}\theta}
\newcommand{\dnu}{\,\mathrm{d}\nu}
\newcommand{\Ker}{\mathrm{Ker}}
\newcommand{\Ima}{\mathrm{Im}}

\newcommand{\Rd}{\color{red}}
\newcommand{\Bk}{\color{black}}
\newcommand{\Mg}{\color{magenta}}
\newcommand{\Wh}{\color{white}}
\newcommand{\Bl}{\color{blue}}
\newcommand{\Yl}{\color{yellow}}


\renewcommand{\div}{\mathrm{div}}
\newcommand{\red}[1]{{\color{red} #1}}

\newcommand{\cqfd}{\begin{flushright}                  
			 $\Box$
                 \end{flushright}}


\title{$L^p$ cohomology and Hodge decomposition for ALE manifolds}

\author{Baptiste Devyver and Klaus Kr\"{o}ncke}
\address{Baptiste Devyver, Institut Fourier - Universit\'e de Grenoble Alpes, France}
\email{baptiste.devyver@univ-grenoble-alpes.fr}
\address{Klaus Kr\"oncke, KTH Stockholm, Sweden}
\email{kroncke@kth.se}

\maketitle
\tableofcontents

\begin{abstract}  

In this article, we prove that the dimensions of $\overline{H}_p^k(M)$, the $L^p$ reduced cohomology spaces on an ALE manifold $M$ of dimension $n\geq 3$, for any values $p\in (1,+\infty)$ and $k\in \{0,\cdots,n\}$, is equal to the dimension of some spaces of decaying harmonic forms that depends on $p$ and $k$.  In this class of manifolds, this provides an extension to $p\neq 2$ of the well-known result of Hodge.  In particular, we prove that for fixed $k\notin\left\{1,n-1\right\}$, the dimension is independent of $p\in (1,\infty)$, while for $k\in\left\{1,n-1\right\}$, the dimension jumps exactly once by a factor $N-1$ (where $N$ is the number of ends) if $p$ varies in $(1,\infty)$.
 We also prove $L^p$ Hodge decompositions for $k$-forms on such manifolds, for the optimal values of $k$ and $p$.

\end{abstract}

\section{Introduction}

In this paper, we look at $L^p$ reduced cohomology spaces and $L^p$ Hodge decomposition in the class of asymptotically, locally Euclidean manifolds (ALE manifolds, in short). In the case $p=2$, there is a very rich literature about $L^2$ cohomology and $L^2$ Betti numbers for non-compact manifolds. Much less is known in the case $p\neq 2$, however see for instance the survey \cite{PP} for pointers to the litterature, both in the case $p=2$ and $p\neq 2$. A general question is to relate the $L^p$ Betti numbers with the topology of the manifold, and its geometry at infinity. For $p=2$, an important fact for $L^2$ cohomology is the Hodge theorem, which asserts that for a general complete manifold, the dimension of the $k^{th}$ space of reduced $L^2$ cohomology is equal to the dimension of the space of $L^2$ differential forms of degree $k$, that are harmonic for the Hodge Laplacian. Our first goal in this paper is to show that for asymptotically locally Euclidean (in short, ALE) manifolds with a finite number of ends, and any $p\in (1,\infty)$, the dimension of the $k^{th}$ space of reduced $L^p$ cohomology can also be computed as the dimension of a space of harmonic forms, and thus can be compared to the dimension of the corresponding $L^2$ cohomology space. Quite surprisingly, this appears to be unknown, despite the rather restricted class of manifolds considered and their quite simple geometry at infinity. Let us introduce for $\alpha\in\R$ the following notation:

$$\ker_{\alpha}(\Delta_k)=\{\omega \in C^\infty(\Lambda^kM)\,;\,\Delta_k\omega=0,\,|\omega|=O_{r\to \infty}(r^{\alpha})\}.$$
Here, $r$ is a ``radial coordinate'' at infinity in $M$, see Section 2, and the operator $\Delta_k=d_{k-1}d^*_k+d_{k+1}^*d_k$ is the Hodge Laplacian, $d_k$ being the differential acting on $k$-forms and $d_k^*$ its formal adjoint. We also denote by $\mathcal{H}_k(M)$ the vector space of $L^2$ harmonic forms of degree $k$, and 

$$H^k_p(M)=\frac{\ker_{L^p}(d_k)}{\im_{L^p}(d_{k-1})}$$
the $k^{th}$ space of {\em reduced} $L^p$ cohomology of $M$, where by definition

$$\ker_{L^p}(d_k)=\{\alpha\in L^p(\Lambda^kM)\,;\,d_k\alpha=0\},$$
the equation $d_k\alpha=0$ being intended in the weak sense, and $\im_{L^p}(d_k)$ being the closure in $L^p$ of $dC_c^\infty(\Lambda^kM)$. We insist on the fact that we chose to define $\im_{L^p}(d_k)$ as a {\em closed} subspace, in order to avoid writing closures throughout the text and to keep notations light, however this is not the choice that is made by most authors in the literature. The Hodge theorem for the $L^2$ cohomology on complete manifolds mentionned before writes:

$$\overline{H}_2^k(M)\simeq \mathcal{H}_k(M).$$
With this settled, we show:

\begin{Thm}\label{thm:main1}

Let $M$ be a connected, oriented ALE manifold with dimension $n\geq 3$. Let $k\in \{0,\cdots,n\}$ and $p\in (1,+\infty)$. Denote by $N$ the number of ends of $M$, which is assumed to be finite. Then, if one of the following holds:

\begin{itemize}

\item[(a)] $k\notin\{1,n-1\}$;

\item[(b)] $k=1$ and $p\in (1,n)$, or $k=n-1$ and $p\in (\frac{n}{n-1},\infty)$;

\item[(c)] $k\in \{1,n-1\}$ and $M$ has only one end;

\end{itemize} 
then,

$$H^k_p(M)\simeq \overline{H}_2^k(M)\simeq \mathcal{H}_k(M).$$
On the other hand, if $k=1$ and $p\geq n$, or if $k=n-1$ and $p\leq \frac{n}{n-1}$, then

$$H^k_p(M)\simeq \ker_{-n}(\Delta_k),$$
which is a subspace of $\mathcal{H}_k(M)$ of codimension $N-1$.
\end{Thm}

\begin{Rem}
{\em 
In fact, we will show that $\mathcal{H}_k(M)=\ker_{1-n}(\Delta_k)$ for every $k$ and $\ker_{1-n}(\Delta_k)=\ker_{-n}(\Delta_k)$ for $k\notin\{1,n-1\}$. For $k\in \{1,n-1\}$, we will in turn show that $\ker_{-n}(\Delta_k)\subset\ker_{1-n}(\Delta_k)$ is a subspace of codimension $N-1$.
}
\end{Rem}

\begin{Rem}
{\em 

According to G. Carron in \cite{C3}, the spaces of reduced $L^2$ cohomology have a topological interpretation, for manifolds with ends that are quasi-isometric to flat ends; this class of manifolds includes in particular all ALE manifolds. Thus, as a corollary to Theorem \ref{thm:main1}, for any $p\in (1,+\infty)$, one can relate the reduced $L^p$ cohomology spaces of an ALE manifold with some topological invariants. In the special case of AE (asymptotically Euclidean) manifolds of dimension $n\geq3$ with a finite number of ends $N$, the topological interpretation of the reduced $L^2$ cohomology spaces is as follows: these spaces can be identified with the relative cohomology spaces of a manifold with boundary, obtained from $M$ by removing in each end the complement of a large Euclidean ball. If $M$ is AE, then the dimension of $\mathcal{H}_k(M)$ is also equal to 

$$
\begin{cases}
\mathrm{dim}\,H^k(\bar{M}), & k\notin \{1,n-1\}\\
\mathrm{dim}\,H^k(\bar{M})+N-1, & k\in\{1,n-1\}
\end{cases}
$$
where $\bar{M}$ is a compact manifold without boundary, obtained by one point -- compactifying each end of $M$.

}
\end{Rem}
Let us now compare briefly the result of Theorem \ref{thm:main1} for $p\neq 2$ with the results available in the literature.  For $p\neq 2$, the reduced $L^p$ cohomology spaces of ALE manifolds have been studied by Gold'shtein, Kuzminov and Shvedov in \cite{GKS}; more specifically, they show (see \cite[Theorem 1]{GKS}) that the long exact sequence in relative cohomology induces pieces of exact sequences in reduced $L^p$ cohomology. They also compute the reduced $L^p$ cohomology of cylinders (see \cite[Theorem 2]{GKS}), which can then be used to get some information about the spaces $H^k_p(M)$ (see for instance the proof of \cite[Theorem 5]{GKS}). Let us explain what their approach yields in the particular case where the manifold $M$ has Euclidean ends; fix a smooth, open relatively compact set $U$ such that $M\setminus U$ is the union of $N$ cones $E_i:=[a,\infty)\times S^{n-1}$, $1\leq i\leq N$, each one being endowed with the cone metric $dt^2+t^2g_{S^{n-1}}$. Note that $\partial U$ consists in the disjoint union of $N$ spheres $S^{n-1}$. Let $1\leq k\leq n-1$. According to \cite[Theorem 1]{GKS}, the following two pieces of long sequences are exact:

\begin{equation}\label{eq:rel1}
H^{k-1}(U)\to \oplus_{i=1}^N \overline{H}^k_p(E_i,\partial E_i)\to \overline{H}_p^k(M) \to H^k(U),
\end{equation}
and

\begin{equation}\label{eq:rel2}
H^k(U,\partial U)\to \overline{H}_p^k(M)\to \oplus_{i=1}^N\overline{H}_p^k(E_i) \to H^{k+1}(U,\partial U)
\end{equation}
(all the $L^p$ cohomology spaces here are reduced, and we recall that the relative cohomology spaces above are defined using similar definitions as the non-relative ones, but using forms with vanish at the boundary in a certain sense). We warn the reader that the exactness of the above two sequences stop on the right and on the left, and does not extend to a long exact sequence. According to \cite[Theorem 2]{GKS}, one has

\begin{equation}\label{eq:van1}
p\geq \frac{n}{k}\Rightarrow \overline{H}_p^k(E_i,\partial E_i)=\{0\}
\end{equation}
and

\begin{equation}\label{eq:van2}
p\leq \frac{n}{k}\Rightarrow \overline{H}_p^k(E_i)=\{0\}.
\end{equation}
Hence,

\begin{equation}\label{eq:van11}
p\geq \frac{n}{k}\Rightarrow \mathrm{dim}\,\overline{H}_p^k(M)\leq \mathrm{dim}\,H^k(U),
\end{equation}
and

\begin{equation}\label{eq:van22}
p\leq \frac{n}{k}\Rightarrow \mathrm{dim}\,\overline{H}_p^k(M) \leq \mathrm{dim}\,H^k(U,\partial U).
\end{equation}
Moreover, Carron's result in \cite{C3} implies that

$$\overline{H}_2^k(M)\simeq H^k(U,\partial U).$$
The exact sequence in relative cohomology of the pair $(U,\partial U)$ and the well-known computation of the cohomology of spheres imply that

$$\mathrm{dim }\,H^k(U,\partial U)=\begin{cases}
\mathrm{dim }\,H^k(U),& k\neq 1,\\
\mathrm{dim }\,H^k(U)+(N-1),& k=1.
\end{cases}$$
Therefore,

\begin{equation}\label{eq:vanL2}
\mathrm{dim }\,\overline{H}_2^k(M)=\mathrm{dim }\,H^k(U,\partial U)=\begin{cases}
\mathrm{dim }\,H^k(U),& k\neq 1,\\
\mathrm{dim }\,H^k(U)+(N-1),& k=1.
\end{cases}
\end{equation}
Assume first that $k\neq 1$, then by \eqref{eq:van11}, \eqref{eq:van22} and \eqref{eq:vanL2}, one has

\begin{equation}\label{eq:ineq1}
\mathrm{dim}\,\overline{H}_p^k(M)\leq \mathrm{dim }\,\overline{H}_2^k(M).
\end{equation}
And for $k=1$, one has

\begin{equation}\label{eq:ineq2}
\mathrm{dim}\,\overline{H}_p^1(M)\leq \begin{cases}
\mathrm{dim }\,\overline{H}_2^1(M),&p< n,\\
\mathrm{dim }\,\overline{H}_2^1(M)-(N-1),&p\geq n.
\end{cases}
\end{equation}
Using duality (see Proposition \ref{prop:duality}), one can in fact obtain for $k=n-1$ that

\begin{equation}\label{eq:ineq3}
\mathrm{dim}\,\overline{H}_p^{n-1}(M)\leq \begin{cases}
\mathrm{dim }\,\overline{H}_2^1(M),&p> \frac{n}{n-1},\\
\mathrm{dim }\,\overline{H}_2^1(M)-(N-1),&p\leq \frac{n}{n-1}.
\end{cases}
\end{equation}
It does not seem possible to relate more precisely reduced $L^2$ and $L^p$ cohomology using directly the results in \cite{GKS}. Our Theorem \ref{thm:main1} therefore appears as a substantial improvement.

\medskip

Another main result of this paper concerns $L^p$ Hodge decompositions on ALE manifolds. For this, we need two definitions:

\begin{Def}

{\em

For $p\in (1,+\infty)$ and $k\in \{0,\cdots,n\}$, we say that the {\em $L^p$ Hodge decomposition} holds for $k$-forms, provided $\im_{L^p}(d_{k-1})+ \im_{L^p}(d^*_{k+1})$ is closed, and 

\begin{equation}\label{eq:HodgeLp}\tag{$\mathscr{H}_p$}
L^p(\Lambda^k M)=\im_{L^p}(d_{k-1})\oplus \im_{L^p}(d^*_{k+1})\oplus \ker_{L^p}(\Delta_{k}).
\end{equation}
}

\end{Def}
We also consider a ``modified'' $L^p$ Hodge decomposition: 

\begin{Def}

{\em 

For $p\in (1,+\infty)$ and $k\in \{0,\cdots,n\}$, we say that a {\em modified $L^p$ Hodge decomposition} holds for $k$-forms, provided $\im_{L^p}(d_{k-1})+\im_{L^p}(d^*_{k+1})$ is closed and

\begin{equation}\label{eq:mod-HodgeLp}\tag{$\tilde{\mathscr{H}}_p$}
L^p(\Lambda^kM)=\im_{L^p}(d_{k-1})\oplus \im_{L^p}(d^*_{k+1})\oplus \ker_{-n}(\Delta_{k}).
\end{equation}

}

\end{Def}
The celebrated De Rham-Kodaira theorem asserts that \eqref{eq:HodgeLp} holds for $p=2$ on {\em any} complete Riemannian manifold. However, for $p\neq 2$, the question of the existence or non-existence of a $L^p$ Hodge decomposition, or --what is closely related-- of the $L^p$ boundedness of the Hodge projectors, is notoriously already difficult in the case of forms of degree $1$. In the case of the Euclidean space itself, the $L^p$ Hodge decomposition \eqref{eq:HodgeLp} has first been proved (even in a stronger form) by T. Iwaniec and G. Martin in \cite{IM}. It is well-known that the question of the $L^p$ boundedness of the Hodge projectors is itself related to the $L^p$ boundedness of Riesz transforms (on forms), that is of the operators $d\Delta_k^{-1/2}$ and $d^*\Delta_k^{-1/2}$ (see Section 6 for some details on this). This is in fact the approach of \cite{IM} to prove the $L^p$ Hodge decomposition on $\R^n$. The same approach has been used by X.D. Li in order to prove \eqref{eq:HodgeLp} in the case the Bochner-Weitzenb\"ock curvature tensors in degree $k-1$, $k$ and $k+1$ are non-negative (see \cite{XDLi1}), building on the celebrated work of Bakry about the Riesz transforms on Riemannian manifolds (\cite{Bakry}). X.D. Li also generalizes results due to Lohou\'e, and obtains stronger $L^p$ Hodge decompositions under assumptions that are essentially equivalent to the positivity of the bottom of the spectrum of the Hodge Laplacian (see \cite{XDLi2}), however these assumptions are not satisfied for ALE manifolds. Let us also warn the reader that the terminology used by X. D. Li (weak and strong $L^p$ Hodge decomposition) has a different meaning than that used in the present paper. Concerning $L^p$ cohomology in degree $1$, more is known: according to P. Auscher and T. Coulhon in \cite{AC}, on ALE manifolds, the $L^p$ boundedness of the Hodge projector on exact forms is essentially equivalent to the boundedness of the scalar Riesz transform $\nabla \Delta^{-1/2}$. The boundedness on $L^p$ spaces of the latter has only relatively recently been elucidated (see \cite{CCH}). For the boundedness on $L^p$ of the Riesz transforms on forms, partial results on asymptotically conical manifolds can be found in \cite{GS}; even for ALE manifolds, these results are somehow incomplete (but they can be made complete by using results that we will prove in the present paper). They however rely on a difficult, precise blow-up analysis of the Schwarz kernel of the resolvent of the Hodge Laplacian. 

To summarize, even for manifolds with Euclidean ends, the only results known to the authors about the $L^p$ Hodge decomposition problem use the boundedness of the Riesz transform either on functions, or on forms. In their 2003 paper \cite[p.5]{CD2}, T. Coulhon and X.T. Duong write in this respect: ``the problem of finding sufficient conditions [for the $L^p$ boundedness of the Hodge projector on exact $1$-forms] is not easier than the Riesz transform problem''.  However, in the present paper, we want to convey the idea that the former {\em can} indeed be easier to prove, at least in some situations: indeed, we prove {\em directly} the $L^p$ boundedness or unboundedness of the Hodge projectors of any degree on ALE manifolds, {\em via} the $L^p$ Hodge decomposition, without using any result on Riesz transforms. As a consequence of our detailed analysis of the decay properties of harmonic forms on ALE manifolds, we are in fact also able to complete the results of \cite{GS} and to completely characterize on any ALE manifold the exponents $p\in (1,+\infty)$ and $k\in\{0,\cdots,n\}$ for which the Riesz transforms on $k$-forms are bounded on $L^p$. See Corollary \ref{cor:Riesz}. However, we stress again that these results are not used to obtain the various $L^p$ Hodge decompositions, and in fact, in the case $p$ and $k$ are such that the Riesz transforms on $k$-forms are unbounded on $L^p$, we still obtain a modified Hodge decomposition, a result that seems inaccessible using only results about the Riesz transforms and the Hodge projectors. Our method ultimately relies on Fredholm properties of the Hodge Laplacian in weighted Sobolev spaces, which are well-known on ALE manifolds. Thus, we prove the $L^p$ Hodge decomposition on $k$-forms for the optimal values of $p$ and $k$ on ALE manifolds, and moreover we give a substitute (the modified $L^p$ Hodge decomposition) for the values of $p$ when this decomposition fails; our result is as follows:

\begin{Thm}\label{thm:Hodge}

Let $M$ be a connected, oriented ALE manifold with dimension $n\geq 3$, and let $p\in (1,+\infty)$, $k\in \{0,\cdots,n\}$. Then, the $L^p$ Hodge decomposition \eqref{eq:HodgeLp} holds if and only if one of the following holds:

\begin{itemize}

\item[(a)] $k\notin\{1,n-1\}$;

\item[(b)] $k\in \{1,n-1\}$ and $p\in (\frac{n}{n-1},n)$;

\item[(c)] $k\in \{1,n-1\}$ and $M$ has only one end.

\end{itemize} 
Moreover,

\begin{itemize}

\item[(d)] if $k\in \{1,n-1\}$, $M$ has $N\geq 2$ ends and $p\geq n$, then the modified $L^p$ Hodge decomposition \eqref{eq:mod-HodgeLp} holds.

\item[(e)] if $k\in \{1,n-1\}$, $M$ has $N\geq 2$ ends and $p\leq \frac{n}{n-1}$, then the space $im_{L^p}(d_{k-1})\oplus \im_{L^p}(d^*_{k+1})\oplus \ker_{L^p}(\Delta_{k})$ is closed and has codimension $N-1$ in $L^p(\Lambda^kM)$.

\end{itemize}

\end{Thm}

\begin{Rem}
{\em 

In case (e), a complement of dimension $N-1$ can be found explicitly, see Proposition \ref{pro:small_exp}.

}
\end{Rem}
This result has interesting consequences for the $L^p$ boudnedness of the Hodge projectors, and the Riesz transforms on forms (see Sections 4 and 6). In particular, it easily implies that on an ALE manifold of dimension $n\geq 3$, the Riesz transform on functions $\nabla \Delta^{-1/2}$ is bounded on $L^p$, if and only if $p\in (1,p^*)$ where $p^*=n$ if $M$ has at least two ends, $p^*=+\infty$ otherwise. See Corollary \ref{cor:Riesz-functions}. This provides arguably one of the shortest and most elementary proof of this well-known fact (see \cite{CCH}, \cite{GH}).

\medskip

As the careful reader would have noticed, our results do not apply to ALE manifolds of dimension $n=2$. We thus state as an interesting open problem:

\noindent {\bf Open problem:} extend Theorems \ref{thm:Hodge} and \ref{thm:main1} to ALE manifolds of dimension $n=2$.

\medskip

The plan of this article is as follows: in Section 2, we set the stage and present some general results concerning the decay of harmonic forms on ALE manifolds. In Section 3, we concentrate more specifically on the case of $1$-forms, which is special. In Section 4, we give some general results concerning $L^p$ Hodge decompositions. In Section 5, we prove the main results of this paper. In Section 6, we give applications to Hodge projectors. The appendix contains mostly material related to the theory of weighted Sobolev spaces on ALE manifolds, which is not presented elsewhere in the manuscript; it also contains the proof of some of the key technical results of the paper (some of which appearing for the first time in the literature). 

\begin{center}

\textbf{Acknowledgements } 

\end{center}

The authors wish to thank G. Carron for having pointed out an important mistake in an early version of this article. 

\medskip

B. Devyver was partly supported by the French ANR through the project RAGE ANR-18-CE40-0012, and as well as in the framework of the ``Investissements d'avenir'' program (ANR-15-IDEX-02) and the LabEx PERSYVAL (ANR-11-LABX-0025-01). K. Kr\"{o}ncke was partly supported by the DFG through the projects 338891943 and 441564857 in the framework of the priority program ``Geometry at Infinity'' (SPP2026).

\section{The setting, and decay of harmonic forms}\label{sec:decay_harmonic_forms}

{\em Notation:} if $\alpha\in \R$ and $f:\R^n \to \R$, we write $f=\mathcal{O}_\infty(r^\alpha)$ or $f\in \mathcal{O}_\infty(r^\alpha)$ to mean that for every $k\in \N$,

$$\partial^kf(x)=O(r(x)^{\alpha-k})\quad \mbox{as }r(x)\to\infty.$$
Here, $r(x)=d(0,x)$ is the usual radial coordinate.

\medskip

In all this paper, we consider $(M^n,g)$ a connected, oriented ALE manifold of dimension $n\geq 3$ and order $\tau>0$. By definition, this means that there exists a compact set $K\subset M$ and a finite number of ends $(E_i)_{i=1,\cdots,N}$, finite subgroups $(\Gamma_i)_{i=1,\cdots,N}$ of $SO(n,\R)$ acting freely on $S^{n-1}$, and diffeomorphims $\phi_i: E_i\to \left(\R^n\setminus \overline{B(0,1)}\right)/\Gamma_i$, $i=1,\cdots,N$ such that $M\setminus K=\sqcup_{i=1}^N E_i$ and $(\phi_i)^*g_{eucl}-g\in \mathcal{O}_\infty(r^{-\tau})$. Here, by abuse of notations, we call $r$ a smooth, positive function on $M$, which in each end $E_i$ is equal to $r_i\circ \phi_i$, where $r_i$ is the distance function to $0$ in $\R^n/\Gamma_i$, and $\mathcal{O}_\infty(r^{-\tau})$ has an obvious meaning despite $r$ not being the Euclidean radial coordinate. For a more detailed account concerning analysis on ALE manifolds, the reader may consult the celebrated paper \cite{Bar86}. The main result of \cite{BKN} shows that any manifold with a finite number of ends, Euclidean volume growth, and faster-than-quadratic curvature decay at infinity is ALE.

\medskip

We denote by
\begin{align*}
	d_k: C^{\infty}(\Lambda^{k}M)\to C^{\infty}(\Lambda^{k+1}M)
	\end{align*}
the exterior derivative on k-forms. Its formal adjoint is denoted by
\begin{align*}
d^*_{k+1}=		(d_k)^*:C^{\infty}(\Lambda^{k+1}M)\to C^{\infty}(\Lambda^{k}M)
	\end{align*}
and the Hodge Laplacian on k-forms is
\begin{align*}
	\Delta_k=d_{k-1}\circ d^*_{k}+ d^*_{k+1}\circ d_k:C^{\infty}(\Lambda^{k}M)\to C^{\infty}(\Lambda^{k}M).
\end{align*}
Sometimes, we will work with the vector bundle of differential forms of any degree:

$$\Lambda^*M=\oplus_{k=0}^n \Lambda^kM,$$
and the differential, co-differential and Hodge Laplacian acting on sections of $\Lambda^*M$ will be simply denoted $d$, $d^*$ and $\Delta$ respectively. We also define the Hodge-Dirac operator

$$\mathcal{D}=d+d^*$$
acting on sections of $\Lambda^*M$. Because of $d^2=0$ and $(d^*)^2=0$, it is obvious that $\mathcal{D}^2=\Delta$. Furthermore, we introduce the notations 
\begin{align*}
	\im_{L^p}(d_{k-1})=\overline{d(C^{\infty}_c(\Lambda^{k-1}M)}^{L^p},
	\qquad \im_{L^p}(d^*_{k+1})=\overline{d^*(C^{\infty}_c(\Lambda^{k+1}M)}^{L^p}
\end{align*}
and
\begin{align*}
	\ker_{L^p}(d_k)&=\left\{\omega\in L^p(\Lambda^kM)\mid d\omega=0\right\},\\
	 \ker_{L^p}(d^*_k)&=\left\{\omega\in L^p(\Lambda^kM)\mid d^*\omega=0\right\},\\
	\ker_{L^p}(\Delta_{k})&=\left\{\omega\in L^p(\Lambda^kM)\mid \Delta_{k}\omega=0\right\}.
\end{align*}
Here, the equations $d_k\omega=0$, $d_k^*\omega=0$ and $\Delta_k\omega=0$ are intended in the weak sense. These are closed subspaces of $L^p(\Lambda^kM)$. Note that the following inclusions take place:
\begin{align*}
	&\im_{L^p}(d_{k-1})\subset \ker_{L^p}(d_k),\qquad 
	\im_{L^p}(d^*_{k+1})\subset \ker_{L^p}(d^*_k),\\
	&\im_{L^p}(d_{k-1})\cap \im_{L^p}(d^*_{k+1})\subset
	\ker_{L^p}(d_k)\cap \ker_{L^p}(d^*_k)\subset
	\ker_{L^p}(\Delta_{k}).
\end{align*}
Finally, for $\alpha\in\R$, we introduce the notation
\begin{align*}
\ker_{\alpha}(\Delta_{k})=\left\{\omega\in C^{\infty}(\Lambda^kM)\mid \Delta_k\omega=0, \omega=\mathcal{O}_\infty(r^{\alpha}) \text{ as }r\to\infty\right\}.
\end{align*}
With these notations settled, one can quote a first result, which is well-known and is due to Yau (see \cite[Theorem 3]{Y}):

\begin{Pro}\label{pro:kernel_0}

Let $p\in (1,\infty)$ and $k\in \{0,n\}$. Then,

$$\ker_{L^p}(\Delta_k)=\{0\}.$$

\end{Pro}
Furthermore, it is easily seen that:

\begin{Pro}

Let $p\in (1,\infty)$, then

$$L^p(\Lambda^0M)=\im_{L^p}(d^*_1),\quad L^p(\Lambda^{n}M)=\im_{L^p}(d_{n-1}).$$

\end{Pro}
\begin{proof}

Let us prove the identity for the $0$-forms, the other case being completely analogous. Denote $q=p'$ the conjugate H\"older exponent of $p$. Since $\im_{L^p}(d_1^*)$ is closed, by $L^p-L^q$ duality it is enough to show that the annihilator in $L^q$ of $\im_{L^p}(d_1^*)$ is $\{0\}$. But if $\varphi\in L^q(\Lambda^0M)$ belongs to this annihilator, then by definition one has

$$(\varphi,\omega)=0,\quad \omega\in \im_{L^p}(d_1^*),$$
and in particular $d\varphi=0$ in the distribution sense. Since $\varphi\in L^1_{loc}$ and $M$ is connected, it follows that $\varphi$ is constant. But since $\varphi\in L^q$, $q\in (1,\infty)$ and $M$ has infinite volume, necessarily $\varphi\equiv 0$. This concludes the proof.

\end{proof}
Thus, in the case $k\in \{0,n\}$, we have a (trivial) $L^p$ Hodge decomposition for every $p\in (1,\infty)$, and the $L^p$ cohomology spaces $H^0_p(M)$, $H^{n}_p(M)$ vanish (recall $M$ is assumed to be connected). In the remaining part of the paper, we will thus only deal with forms of degree $k\in \{1,\cdots,n-1\}$. The following proposition is one of the main ingredients of this paper:
\begin{proposition}[see also {\cite[Proposition 4.3]{KP}}]\label{prop:decay_harmonic_forms}
	Let $p\in (1,\infty)$ and $k\in\{1,\cdots,n-1\}$. The following hold true:
	
\begin{enumerate}
\item[(a)] $$\ker_{L^p}(\Delta_{k})= \ker_{L^p}(d_k)\cap \ker_{L^p}(d^*_k).$$

\item[(b)] $$\ker_{L^p}(\Delta_{k})\subset \ker_{1-n}(\Delta_{k}).$$

\item[(c)] if $k\notin\left\{1,n-1\right\}$ or $p\leq \frac{n}{n-1}$, then 

$$\ker_{L^p}(\Delta_{k})= \ker_{-n}(\Delta_{k}).$$

\item[(d)] if $p>\frac{n}{n-1}$, then

$$\ker_{L^p}(\Delta_k)=\ker_{1-n}(\Delta_k).$$
\item[(e)] Furthermore, both $\ker_{1-n}(\Delta_1)$ and $\ker_{-n}(\Delta_1)$ are finite dimensional.
\end{enumerate}

\end{proposition}
The proof of (a)-(d) in Proposition \ref{prop:decay_harmonic_forms} relies on a standard iteration procedure in weighted Sobolev spaces. For readability of the paper, and since some of the arguments of the proof will be used again later in the paper, we have decided to include a sketch of the proof in appendix.
From the proof of Proposition \ref{prop:decay_harmonic_forms}, we also extract the following result which will be used later:

\begin{Cor}\label{cor:decay_harmonic}

Let $\omega$ be a harmonic form on $M$, and assume that there exists $\epsilon>0$ such that, as $r\to\infty$,

$$\omega\in \mathcal{O}_\infty(r^{1-n-\epsilon}).$$
Then, as $r\to\infty$,

$$\omega\in \mathcal{O}_\infty(r^{-n}).$$

\end{Cor}
We will also need two results about the asymptotic of harmonic forms at infinity. These two results are also proved in appendix. The first one is an expansion for bounded harmonic functions:

\begin{Lem}\label{lem:expansion_fn}

Let $u\in \ker_0(\Delta_0)$. Then, there exists $\epsilon>0$ such that in each end $E_i$, there exists a two constants $c_i\in\R$ and $A_i\in\R$ such that following expansion holds:

\begin{align}\label{eq:expansion_u}
	u=c_i+A_ir^{2-n}+\mathcal{O}_{\infty}( r^{2-n-\epsilon}),\quad r\to\infty.
\end{align}

\end{Lem}
The second lemma is an expansion for decaying harmonic one-forms:

\begin{Lem}\label{lem:expansion_forms}

Let $\omega\in \ker_{-\alpha}(\Delta_1)$ for some $\alpha>0$. Then, there exists $\epsilon>0$ such that in each end $E_i$, there exists a constant $B_i\in\R$ such that following expansion holds:

\begin{align}\label{eq:expansion_o}
\omega=B_ir^{1-n}dr+\mathcal{O}_\infty(r^{1-n-\epsilon}),\quad r\to\infty.
\end{align}

\end{Lem}

%

\section{Harmonic functions on ALE manifolds}

%
%
In this section, we investigate in more details the space of harmonic forms of degree $1$ and $n-1$. In order to do that, we will first have to look at the space of bounded harmonic functions. In the following, we will use the Green operator $\Delta^{-1}$ on $M$, defined by
$$\Delta^{-1}w(x)=\int_MG(x,y)w(y)\,dy,$$
where $G(x,y)$ is the positive Green function on $M$, which exists thanks to the assumptions $n\geq 3$ and $M$ being ALE (since the volume growth is faster than quadratic, $M$ is non-parabolic, as a consequence, e.g. of \cite[Theorem 7.6]{Grig}. See \cite[Theorem 5.1]{Grig} for the different characterizations of non-parabolicity). Moreover, the following assymptotics hold: for all $y\in M$,
\begin{equation}\label{eq:Green}
G(x,y)\in O(r^{2-n}(x)),\quad \mbox{as } r(x)\to \infty.
\end{equation}
Also, the above estimate is uniform for all $y$ staying in a fixed compact set. For the sake of completeness, we will explain these two points below. The proof of \eqref{eq:Green} relies on the concept of {\em minimal growth at infinity} for positive harmonic functions that we now introduce:

\begin{Def}

{\em 

Let $u$ be a postive harmonic function defined in an end $E_i$. We say that $u$ has {\em minimal growth at infinity in the end }$E_i$, if the following holds: for every positive harmonic function $v$ in the end $E_i$, there exists a constant $C>0$ such that

$$u\leq Cv.$$
If $u$ is a positive harmonic function outside a compact set and has minimal growth at infinity in every end, then we shall simply say that $u$ has {\em minimal growth at infinity}.

}

\end{Def}
It follows right away that if $u$ and $w$ are positive harmonic functions in an end $E_i$ which both have minimal growth, then there is a positive constant $C$ such that

$$C^{-1}u\leq v\leq Cu.$$
It is also standard and not hard to show that for every $o\in M$, the Green function $G(o,\cdot)$ with pole $o$ --provided it is finite-- has minimal growth at infinity: indeed, this is a consequence of the construction of the Green function through an exhaustion sequence of compact sets in $M$, and of the maximum principle. See for instance \cite{P}. Here is the two-sided estimate we need for the Green function: 

\begin{Lem}\label{lem:min_growth}

Let $K\Subset M$ be a compact set; then, there exists a constant $C>0$ such that for all $x\in K$ and $r(y)\to\infty$, 

$$C^{-1}r(y)^{2-n}\leq G(x,y)\leq C r(y)^{2-n}.$$

\end{Lem}

\begin{proof}

The hypothesis on the metric $g$ implies that

$$\Delta r^{2-n}=\mathcal{O}_\infty(r^{-\kappa}),$$
with $\kappa=n+\tau>3$.  Let $\chi\in C_c^\infty(\R^n)$ with $\chi\equiv 1$ in restriction to $B_{\R^n}(0,1)$, $\chi\equiv0$ in restriction to $\R^n\setminus B_{\R^n}(0,2)$. For $R>0$ large enough let $\xi_R$ be the radial function on $M$ which, in each end $E_i$, identifies with the function $x\mapsto 1-\chi\left(\frac{x}{R}\right)$. Finally, let $f=\Delta r^{2-n}$ and $f_R=f\xi_R$; it is plain to see that the support of $f_R$ is inside $\Omega_R^i=(\R^n\setminus B_{\R^n}(0,R))/\Gamma_i$ in each end $E_i$, and that there is a constant $C>0$ independant of $R$ such that

$$
|f_R|\leq Cr^{-\kappa}.
$$
Consider the function $\varphi_R:=r^{2-n}-\Delta^{-1}f_R$. Then, $\Delta\varphi_R=0$ in $\R^n\setminus B_{\R^n/\Gamma}(0,2R)$ in each end. We claim that it is enough to prove that for some choice of $R$ large enough, there exists a constant $c>0$ such that, as $x\to\infty$,

\begin{equation}\label{eq:phiR}
c^{-1}r(x)^{2-n}\leq \varphi_R(x)\leq cr(x)^{2-n}.
\end{equation}
Indeed, \eqref{eq:phiR} first implies that 

$$\lim_{x\to\infty}\varphi_R(x)=0,$$
and according to \cite[Proposition 6.1]{DFP} (with $u_0=\varphi_R$, $u_1=1$), $\varphi_R$ has minimal growth at infinity, hence there is a constant $C>0$ such that, as $y\to\infty$,

$$C^{-1}\varphi(y)\leq G(o,y)\leq C\varphi(y),$$
where $o\in K$ is a fixed point.  Given the asymptotics on $\varphi_R$, one gets

$$G(o,y)\simeq r(y)^{2-n},\quad\mbox{ as }y\to\infty.$$
The local Harnack inequality (see \cite[Theorem 5, p. 353]{Ev}) for positive solutions of uniformly elliptic operators with smooth coefficients (such as the Laplacian on $M$) then yields the existence of a uniform constant $\tilde{C}>0$ such that for all $x\in K$,

$$ \tilde{C}^{-1} r(y)^{2-n}\leq G(x,y)\leq  \tilde{C} r(y)^{2-n},\quad\mbox{ as }y\to\infty.$$
Hence, in order to prove Lemma \ref{lem:min_growth}, it is enough to prove \eqref{eq:phiR}, which is what we do now. We first estimate, thanks to the positivity of the Green operator:

$$|\Delta^{-1}f_R(x)|\leq \Delta^{-1}|f_R|(x)\lesssim (\Delta^{-1}(\mathbf{1}_{\Omega_R^i}r^{-\kappa}))(x).$$
We claim that on $M$ the following Sobolev inequality holds:

\begin{equation}\label{eq:Sob}
||u||_{\frac{2n}{n-2}}\lesssim||\nabla u||_2,\quad \forall u\in C_c^\infty(M).
\end{equation}
Indeed, such an inequality when $u$ has compact support in an end of $M$ comes from the fact that \eqref{eq:Sob} holds on $\R^n$, hence also on $\R^n$ endowed with a metric which is bi-Lipschitz to the Euclidean one, and the fact that $u$ identifies with a smooth, $\Gamma$-periodic function on $\R^n$ with compact support (this uses the fact that $\Gamma$ is finite). Hence, \eqref{eq:Sob} holds in the ends of $M$, and the validity of \eqref{eq:Sob} on $M$ itself then follows from \cite[Proposition 2.4]{C2}. The Sobolev inequality \eqref{eq:Sob} implies estimates $p_t(x,y)\leq Ct^{-n/2}e^{-\frac{d^2(x,y)}{ct}}$ for some positive constants $C,c$, and by integration this yields the following uniform bound for the Green function on $M$:

$$G(x,y)\lesssim d(x,y)^{2-n}.$$
See for instance \cite[Exercice 15.8]{Grig2} for details. Hence, one has the following estimate:

$$|\Delta^{-1}f_R(x)|\lesssim (k*\mathbf{1}_{\Omega_R^i}r^{-\kappa})(x),\quad x\in E_i$$
where the kernel $k$ is by definition $k(x,y)=|x-y|^{2-n}$ and $*$ denotes the convolution in the Euclidean end $E_i$ for the Riemannian measure. Since (up to a constant) the Riemannian measure on $M$ is bounded from above by the Euclidean measure in each end and $d$ is comparable to the Euclidean distance, it follows that $|\Delta^{-1}f_R(x)|$ is bounded by a constant multiple of the following integral in $\R^n$:

$$I(x):=\int_{|y|\geq R} \frac{dy}{|x-y|^{n-2}|y|^\kappa},$$
where by a slight abuse of notations, the point $x\in E_i$ has been identified to $\Phi_i(x)\in\R^n$. In order to estimate $I(x)$, we let $\delta=\frac{|\bar{x}|}{2}$, and we split the domain of integration into three parts (if $R<4\delta$) or two parts (if $R\geq 4\delta$): in the case $R<4\delta$, we write

$$I\leq I_1+I_2+I_3,$$
where

$$I_1(x)=\int_{|\bar{x}-y|\leq \delta}\frac{dy}{|\bar{x}-y|^{n-2}|y|^\kappa},$$

$$I_2(x)=\int_{R\leq |y|\leq 4\delta\,;\,|\bar{x}-y|\geq \delta}\frac{dy}{|\bar{x}-y|^{n-2}|y|^\kappa},$$
and 

$$I_3(x)=\int_{|y|\geq 4\delta,\,|\bar{x}-y|\geq \delta}\frac{dy}{|\bar{x}-y|^{n-2}|y|^\kappa}.$$
In the case $R\geq 4\delta$, we do not need $I_2(x)$ so by convention in this case we set $I_2(x)=0$. We first estimate $I_1(x)$: notice that by the definition of $\delta$, $|\bar{x}-y|\leq \delta$ implies $|y|\geq \delta$, so

$$I_1(x)\lesssim \delta^{-\kappa}\int_0^\delta\frac{t^{n-1}dt}{t^{n-2}}\simeq \delta^{2-\kappa}.$$
Next, if $R<4\delta$, we estimate $I_2(x)$ by

$$I_2(x)\lesssim \delta^{2-n}\int_{R}^{4\delta}t^{n-1-\kappa}\,dt\lesssim \delta^{2-n}R^{n-2-\kappa},$$
(since $n-2-\kappa<0$). Finally, we estimate $I_3(x)$: note that if $|y|\geq 4\delta=2|\bar{x}|$ then $|\bar{x}-y|\simeq |y|$, so 

$$I_3(x)\lesssim \int_{|y|\geq 4\delta} \frac{dy}{|y|^{n-2+\kappa}}\simeq \delta^{2-\kappa}.$$
Hence, recalling that $\delta\simeq r(x)$, one obtains the existence of a constant $C>0$ such that for all $x\in M$,

$$I(x)\leq C(R^{n-2-\kappa}+r(x)^{n-\kappa})r(x)^{2-n}.$$
Therefore, recalling that $\kappa>n$, one can choose $R>0$ so that $CR^{n-2-\kappa}<\frac{1}{4}$, one gets that for all $x\in M$ such that $r(x)\geq (4C)^{\frac{1}{\kappa-n}}$,

$$|\Delta^{-1}f_R(x)|\leq I(x)<\frac{1}{2}r(x)^{2-n}.$$
Let $r_0:=(4C)^{\frac{1}{\kappa-n}}$. Then, for the above choice of $R$ we get the following estimate on $\varphi_R$:

$$\frac{1}{2}r(x)^{2-n}\leq \varphi_R(x)\leq \frac{3}{2}r(x)^{2-n},\quad r(x)\geq r_0,$$
which proves \eqref{eq:phiR}, and the proof is complete.

\end{proof}
%
We now describe the space of bounded harmonic functions on $M$. It is well-known that the dimension of this space is closely related to the number of ends of $M$, even if $M$ is not ALE, as is apparent for instance from \cite[Theorem 2.1]{LT1} and \cite[Theorems 4.2 \& 4.5]{LT2}. Here we investigate such a relationship in detail, in the setting of ALE manifolds. Recall that $\ker_{0}(\Delta_0)$ denotes the space of bounded harmonic functions. Then we have:
\begin{lem}\label{lem:dimension_bounded_harmonics}
For each tuple $c=(c_1,\ldots c_{N})\in \R^{N}$, there exists a unique bounded harmonic function $u\in C^{\infty}(M)$ such that in each end $E_i$, $i=1,\ldots, N$, we have $u\to c_i$, as $r\to\infty$. In particular,
\begin{align*}
	\dim(\ker_{0}(\Delta_0))=N.
\end{align*}

\end{lem}
This result should be compared with \cite[Theorem 4.5]{LT2}.
\begin{proof}
	We first consider uniqueness: assume that $u,v\in \ker_{0}(\Delta_0)$ are both converging to the same constant $c_i$ at each end $E_i$, then $u-v$ is a harmonic function converging to $0$ at each end. By the maximum principle, this implies that $u-v\equiv0$, so $u\equiv v$.\\
	For the existence, let $c=(c_1,\ldots c_{N})\in \R^{N}$ and consider a bounded function $v\in C^{\infty}(M)$ such that $v\equiv c_i$ at each end $E_i$. Then, $\Delta v$ is compactly supported and by \eqref{eq:Green}, we have at each end $\Delta^{-1}(\Delta v)=O(r^{2-n})$ as $r\to \infty$, so that $u:=v-\Delta^{-1}(\Delta v)$ is the desired bounded harmonic function.
\end{proof}
Because $M$ is connected, the kernel of $d_0$ is given by the constant functions and the isomorphism theorem from linear algebra directly yields
\begin{cor}\label{cor:dimension_d_harmonic functions}
	The subspace
	\begin{align*}
		d_0(\ker_{0}(\Delta_0))\subset C^{\infty}(\Lambda^1M)
		\end{align*}
	is $N-1$-dimensional.
\end{cor}
%
%
\begin{Lem}\label{lem:Ai}
Let $u$ be a bounded harmonic function; by \eqref{eq:expansion_u} in every end $E_i$, one has for some $\varepsilon>0$,

$$u=c_i+A_ir^{2-n}+\mathcal{O}_\infty(r^{2-n-\varepsilon}).$$
Then the following two properties hold concerning the real numbers $A_i$:
\begin{enumerate}

\item[(i)] $$\sum_{i=1}^N\frac{A_i}{\mathrm{Card}(\Gamma_i)}=0.$$

\item[(ii)] If $u$ is nonconstant, then there is $i\in\{1,\cdots,N\}$ such that $A_i\neq 0$.

\end{enumerate}

\end{Lem}

\begin{proof}

The proof of (i) follows from arguments similar to \cite[Proposition 4.5]{D}. Denote $\omega=du$, and let $f\in C^\infty(M)$ be such that for every $i=1,\cdots,N$ and every $x$ in the end $E_i$, $f(x)=c_i+A_ir^{2-n}$. Then,

$$\omega=df+\eta,$$
with $\eta\in \mathcal{O}_\infty(r^{1-n-\epsilon})$ for some $\epsilon>0$. Let $R\geq1$ and let $D(R)$ be the smooth open set defined by
$$D(R)=K\sqcup_{i=1}^N \phi_i^{-1}(B^{\R^n/\Gamma_i}(0,R)\setminus \overline{B^{\R^n/\Gamma_i}(0,1)}).$$
By the Green formula,

$$\int_{D(R)}\Delta f\dv=\int_{\partial D(R)}\frac{\partial f}{\partial \nu}\dS.$$
Given the asymptotic of the metric on $M$, and the definition of $f$, one has, as $R\to \infty$,

$$\int_{\partial D(R)}\frac{\partial f}{\partial \nu}\dS=-(n-2)\left(\sum_{i=1}^N\frac{A_i}{\mathrm{Card}(\Gamma_i)}\right)\omega_n+\varepsilon(R),$$
where $\lim_{R\to \infty}\varepsilon(R)=0$ ands $\omega_n$ denotes the $(n-1)$-dimensional volume of the unit sphere in $\R^n$ (we have used here that the volume of the unit sphere $S^{n-1}/\Gamma_i$ is equal to $\frac{\mathrm{Vol}(S^{n-1})}{\mathrm{Card}(\Gamma_i)}$). On the other hand, since $d^*\omega=\Delta u=0$, one has $\Delta f=-d^*\eta$. Let $X=\eta^\flat$ the vector field canonically associated to the $1$-form $\eta$, then $-d^*\eta=\div(X)$, so Gauss' theorem implies that, as $R\to\infty$,

\begin{eqnarray*}
\int_{D(R)}\Delta f\dv &=& \int_{D(R)}\div(X)\dv\\
&=&\int_{\partial D(R)}\langle X,\nu\rangle\dS\\
&=& O(R^{-\epsilon}),
\end{eqnarray*}
since $X\in O(r^{1-n-\epsilon})$. Comparing the two estimates, one obtains that

$$\lim_{R\to\infty} \int_{\partial D(R)}\frac{\partial f}{\partial \nu}\dS=0,$$
which implies (i).

\medskip

For the proof of (ii), we choose $i\in \{1,\cdots,N\}$ such that $c_i=\inf_{j=1,\cdots,N}c_j$. Up to changing $u$ into $u-c_i$, we can --and will-- assume that $c_i=0$. Then, 

$$\liminf_{x\to \infty} u(x)=0,$$
which implies, by the maximum principle and the assumption that $u$ is non-constant, that $u>0$ on $M$. On the end $E_i$, one thus has a positive harmonic function $u$, whose limit is zero at infinity in $E_i$. According to \cite[Proposition 6.1]{DFP}, it follows that $u$ has minimal growth at infinity in the end $E_i$. But the metric in $E_i$ being ALE, the minimal growth at infinity for positive harmonic functions on $E_i$ is $\simeq r^{2-n}$ (Lemma \ref{lem:min_growth}). Therefore, one obtains 

$$u(x)\gtrsim r(x)^{2-n},\quad x\in E_i,\,x\to\infty.$$
Hence, $A_i\neq 0$, which concludes the proof of (ii).
\end{proof}
\begin{cor}\label{cor:second_term_harmonic_expansion}
	For each tuple $B=(B_1,\ldots,B_N)$ with  $\sum_{i=1}^N\frac{B_i}{\mathrm{Card}(\Gamma_i)}=0$, there exists a function $u\in \ker_0(\Delta_0)$,
	unique up to the addition of a constant, such that for every $i=1,\cdots,N$, $A_i=B_i$, where $A_i$ are the coefficients in the expansion \eqref{eq:expansion_u} of $u$ in the end $E_i$.
\end{cor}
\begin{proof}
	Denote $\gamma_i=\mathrm{Card}(\Gamma_i)^{-1}$, and consider the linear map 
	\begin{align*}\Phi: \ker_0(\Delta_0)&\to \left\{(x_1,\ldots x_N)\in \R^N\mid \gamma_1x_1+\ldots \gamma_Nx_N=0\right\}\\
		u&\mapsto (A_1,\ldots,A_N),
		\end{align*}
	where the $A_i$ are as in \eqref{eq:expansion_u}. By Lemma \ref{lem:Ai}, $\Phi$ is well-defined and $\ker(\Phi)$ is exactly given by the constant functions, hence is 1-dimensional. Therefore, since the domain of $\Phi$ is N-dimensional by Lemma \ref{lem:dimension_bounded_harmonics}, $\im(\Phi)$ is (N-1)-dimensional and hence, $\Phi$ is surjective.
\end{proof}

We will also need later the following result concerning the asymptotic behaviour of harmonic $1$-forms:

\begin{Lem}\label{lem:asym_forms}
Let $\omega\in \ker_{1-n}(\Delta_1)$. Then, there are real numbers $(B_j)_{j=1,\cdots,N}$ such that, in each end $E_i$,
$$\omega=B_jr^{1-n}dr+\mathcal{O}_\infty(r^{1-n-\epsilon})$$
for some $\epsilon>0$.
Furthermore, $\sum_{j=1}^N \frac{B_j}{\mathrm{Card}(\Gamma_i)}=0$.
\end{Lem}

\begin{proof}

The asymptotic expansion follows directly from Lemma \ref{lem:expansion_forms}.
Let $f\in C^\infty(M)$ be such that $f=-\frac{B_j}{n-2}r^{2-n}$ outside a compact set in the end $E_j$, then
$$\omega=df+\mathcal{O}_\infty(r^{1-n-\epsilon})$$
for some $\epsilon>0$.
Since $d^*\omega=0$ by Proposition \ref{prop:decay_harmonic_forms}, the proof of part (i) in Lemma \ref{lem:Ai} now implies that $\sum_{j=1}^N \frac{B_j}{\mathrm{Card}(\Gamma_i)}=0$.
\end{proof}
\begin{cor}\label{cor:decomp_kernel_1}
We have
\begin{align*}\ker_{1-n}(\Delta_1)=d(\ker_0(\Delta_0))\oplus \ker_{-n}(\Delta_1)
	\end{align*}
and this sum is $L^2$-orthogonal.
\end{cor}
\begin{proof}
	 Let $\omega\in \ker_{1-n}(\Delta_1)$. By Lemma \ref{lem:asym_forms}, we may write 
	$$\omega=B_jr^{1-n}dr+\mathcal{O}_\infty(r^{1-n-\epsilon})$$
	for some $\epsilon>0$ and a tuple $(B_j)_{j=1,\cdots,N}$ with $\sum_{j=1}^N \frac{B_j}{\mathrm{Card}(\Gamma_i)}=0$. By Corollary \ref{cor:second_term_harmonic_expansion}, there exists a harmonic function $u\in\ker_0(\Delta_0)$,
	 such that the terms $A_i$ in the expansion \eqref{eq:expansion_u} of $u$ are at each end equal to $B_i$. Because $u$ is unique up to a constant, $du\in \ker_{1-n}(\Delta_1)$ is uniquely determined. We obtain the expansion
	 \begin{align*}
	 	du=B_jr^{1-n}dr+\mathcal{O}_\infty(r^{1-n-\epsilon}).
	 \end{align*}
	 Therefore, $\eta:=\omega-du=\mathcal{O}_\infty(r^{1-n-\epsilon})$ and $\eta$ is uniquely determined. Because $\eta$ is a harmonic form since $u$ and $\omega$ are, we conclude from Corollary \ref{cor:decay_harmonic} that $\eta =\mathcal{O}_\infty(r^{-n})$. This finishes the first part of the proof.
	 
	To show $L^2$-orthogonality of this decomposition, let $u\in \ker_0(\Delta_0)$ and $\omega\in \ker_{-n}(\Delta_1)$. Furthermore, let $D(R)$ as in the proof of Lemma \ref{lem:Ai} be given by $$D(R)=K\sqcup_{i=1}^N \phi_i^{-1}(B^{\R^n/\Gamma_i}(0,R)\setminus \overline{B^{\R^n/\Gamma_i}(0,1)}).$$ Because $d_1^*\omega=0$ (Proposition \ref{prop:decay_harmonic_forms}), we get
	 \begin{align*}
	 	\int_{D(R)} \langle du,\omega\rangle \dv=\int_{\partial D(R)}u(*\omega)=O(R^{-1}),
	 \end{align*}
	 since $\omega=O(r^{-n})$ and $u$ is bounded. Letting $R\to\infty$ yields
	 \begin{align*}
	 	(du,\omega)_{L^2}=0
	 \end{align*}
 which is exactly what we needed to show.
\end{proof}
Proposition \ref{prop:decay_harmonic_forms}, Corollary \ref{cor:dimension_d_harmonic functions}, Corollary \ref{cor:decomp_kernel_1} and Hodge duality imply:

\begin{Cor}\label{cor:missing_ker}

Assume that $M$ is ALE with only one end; let $k\in \{1,n-1\}$ and $p\in (1,\infty)$. Then,

$$\ker_{L^p}(\Delta_k)=\ker_{-n}(\Delta_k).$$

\end{Cor}
We conclude this section with the following result:

\begin{Pro}\label{pro:im_p>n}

Let $p\in [n,+\infty)$. Then,
$$d(\ker_0(\Delta_0))\subset \im_{L^p}(d_0)\cap \ker_{1-n}(\Delta_1).$$
\end{Pro}

\begin{proof}
Clearly, the expansion \eqref{eq:expansion_u} for bounded harmonic functions provides that $d(\ker_0(\Delta_0))\subset \ker_{1-n}(\Delta_1)$, so in order to finish the proof of the proposition, it is enough to prove that
$$d(\ker_0(\Delta_0))\subset \im_{L^p}(d_0).$$
Since $p\geq n$, $M$ is $p$-parabolic, hence there is a sequence $(\chi_n)_{n\in\N}\subset C_c^\infty(M)$ of cut-off functions such that for every $n\geq0$, $0\leq \chi_n\leq 1$, $\chi_n\rightarrow 1$ pointwise and $\int_M|\nabla \chi_n|^p\rightarrow 0$ (see \cite[Section 3.2]{CHS}). Let $\omega=d u\in d(\ker_0(\Delta_0))$, and denote $\omega_n=d(\chi_nu)\in dC_c^\infty(M)$. We are going to prove that $(\omega_n)_{n\in\N}$ converges to $\omega$ in $L^p$. One has

$$||\omega-\omega_n||_p\leq ||(1-\chi_n)du||_p+||u(d\chi_n)||_p.$$
Given that $u\in L^\infty$ and $||\nabla \chi_n||_p\rightarrow 0$, the second term on the right hand side converges to zero. The first one also converges to zero, as follows from the Dominated Convergence theorem, since $(1-\chi_n)du$ converges pointwise to zero and $du\in L^p$. Finally, we have $||\omega-\omega_n||_p\rightarrow 0$, which implies that $\omega\in\im_{L^p}(d_0)$. This is what was needed to achieve the proof.

%
%
%
%

\end{proof}

\section{Hodge projectors}

In this section, we introduce the Hodge projectors and make the connection with the $L^p$-Hodge decomposition. In part of this section, we will work with general complete Riemannian manifolds, which are not ALE unless explicitly stated. We first recall the $L^2$-Hodge decomposition for $k$-forms, which holds true for any complete Riemannian manifold:

\begin{equation}\label{eq:HodgeL2}\tag{$\mathscr{H}_2$}
L^2(\Lambda^kM)=\im_{L^2}(d_{k-1})\oplus \im_{L^2}(d^*_{k+1})\oplus \ker_{L^2}(\Delta_{k}),
\end{equation}
and the sum is orthogonal. We will denote by $\Pi_{k,d}$, $\Pi_{k,d^*}$ and $\Pi_{k,0}$ the three orthogonal projectors from $L^2(\Lambda^kM)$ onto $\im_{L^2}(d_{k-1}),$ $\im_{L^2}(d^*_{k+1}),$ $\ker_{L^2}(\Delta_{k})$ respectively. In addition to the linear algebra equality \eqref{eq:HodgeL2}, the three Hodge projectors defined above are bounded operators on $L^2(\Lambda^*M)$. If now $p\in (1,\infty)$, we say that the {\em $L^p$-Hodge decomposition} holds for $k$-forms, provided $\im_{L^p}(d_{k-1})+\im_{L^p}(d^*_{k+1})$ is {\em closed} and

\begin{equation}\label{eq:HodgeLp}\tag{$\mathscr{H}_p$}
L^p(\Lambda^kM)=\im_{L^p}(d_{k-1})\oplus \im_{L^p}(d^*_{k+1})\oplus \ker_{L^p}(\Delta_{k}).
\end{equation}
Note that the above equation \eqref{eq:HodgeLp} is {\em a priori} only meant as a linear algebra equality, and $\im_{L^p}(d_{k-1})+\im_{L^p}(d^*_{k+1})$ is {\em closed} if and only if \eqref{eq:HodgeLp} is a {\em topological} equality, i.e. the linear projections onto each factor are continuous. Note that this is always the case for $p=2$, as has been recalled above. We chose to stress the difference between the linear algebra equality and the topological equality by singling out the key assumption that $\im_{L^p}(d_{k-1})+\im_{L^p}(d^*_{k+1})$ is closed. If the $L^p$ Hodge decomposition holds in the above sense, we will denote by $\Pi^p_{k,d}$, $\Pi^p_{k,d^*}$ and $\Pi^p_{k,0}$ the three projectors from $L^p(\Lambda^kM)$ onto the three subspaces $\im_{L^p}(d_{k-1})$, $\im_{L^p}(d^*_{k+1})$, and $\ker_{L^p}(\Delta_{k})$ respectively. To stress again, the assumption that $\im_{L^p}(d_{k-1})+\im_{L^p}(d^*_{k+1})$ is closed together with the linear algebra equality \eqref{eq:HodgeLp}, are equivalent to the three $L^p$-Hodge projectors being bounded on $L^p$. We warn here the reader that in general, despite $\im_{L^p}(d_{k-1})$ and $\im_{L^p}(d_{k+1})$ being closed subspaces, their sum $\im_{L^p}(d_{k-1})+\im_{L^p}(d^*_{k+1})$ is not necessarily closed.  As will be apparent, sometimes it will be also useful to consider the following {\em ``modified'' $L^p$- Hodge decomposition}, which is said to hold provided $\im_{L^p}(d_{k-1})+\im_{L^p}(d^*_{k+1})$ is closed and the following linear algebra equality holds:

\begin{equation}\label{eq:mod-HodgeLp}\tag{$\tilde{\mathscr{H}}_p$}
L^p(\Lambda^kM)=\im_{L^p}(d_{k-1})\oplus \im_{L^p}(d^*_{k+1})\oplus \ker_{-n}(\Delta_{k}).
\end{equation}
In case \eqref{eq:mod-HodgeLp} holds, we will call $\tilde{\Pi}^p_{k,d}$, $\tilde{\Pi}^p_{k,d^*}$ and $\tilde{\Pi}^p_{k,0}$ the three ``modified'' Hodge projectors, that is the projectors asociated with the decomposition \eqref{eq:mod-HodgeLp}.

It will be interesting to compare the Hodge projectors on $L^2$ and on $L^p$. However, in order to be able to do this, one implicitly needs that

$$\ker_{L^p}(\Delta_k)=\ker_{L^2}(\Delta_k),$$
and we know from Proposition \ref{prop:decay_harmonic_forms} that this does not always hold. We now introduce an assumption that will allow us to compare the kernels $\ker_{L^p}(\Delta_k)$ for different values of $p$. Recall first that there is a Bochner formula:

$$\Delta_k=\nabla^*\nabla+\mathscr{R}_k,$$
where $\nabla$ is the connection induced on $\Lambda^*M$ by the Levi-Civita connection, and $\mathscr{R}_k\in \mathrm{End}(\Lambda^k M)$ is self-adjoint and is related to the curvature operator. See \cite{GM}. We say that $M$ satisfies Assumption ($H_k$), provided $||\mathscr{R}_k||_{L^\infty(M)}<\infty$, and for every $1\leq p\leq q\leq +\infty$, $e^{-\Delta_0} : L^p\to L^q$ is bounded. Note that the latter holds in particular if the Sobolev inequality holds true on $M$ (see \cite{S}), a property which, as we have already recalled, holds true for ALE manifolds of dimension $n\geq 3$, see equation \eqref{eq:Sob}. If ($H_k$) holds for any $k=0,\cdots,n$, we say that $M$ satisfies Assumption (H). As a consequence of the fact that the Riemann curvature tensor of an ALE manifold tends to $0$ at infinity, it follows that ALE manifolds of dimension $n\geq 3$ satisfy Assumption (H). The relevance of assumption (H) stands from the following result, whose proof (in the case $k=1$) can be found in \cite[p.87]{CCH}:

\begin{Lem}\label{lem:ker}

Let $M$ be a complete Riemannian manifold satisfying ($H_k$) for some $k\in \{0,\cdots,n\}$. Then, for every $1\leq q\leq p\leq +\infty$, one has

$$\ker_{L^q}(\Delta_k)\subset \ker_{L^p}(\Delta_k).$$ 

\end{Lem}
We have the following consequence of the $L^p$ Hodge decomposition:

\begin{proposition}\label{pro:Hodge}

Let $p\in (1,\infty)$ and $k\in\{0,\cdots,n\}$ such that the $L^p$ Hodge-De Rham decomposition \eqref{eq:HodgeLp} holds. Assume moreover that 

$$\ker_{L^p}(\Delta_k)=\ker_{L^2}(\Delta_k).$$
Then, each one of the three $L^2$ Hodge projectors $\Pi_{k,d}$, $\Pi_{k,d^*}$ and $\Pi_{k,0}$, in restriction to $dC_c^\infty(\Lambda^{k-1}T^*M)\oplus d^*C_c^\infty(\Lambda^{k+1}T^*M)\oplus \ker_{L^2}(\Delta_k)$, extends uniquely to a bounded operator on $L^p(\Lambda^k M)$.

\end{proposition}

\begin{proof}

Denote

$$E=dC_c^\infty(\Lambda^{k-1}M)\oplus d^*C_c^\infty(\Lambda^{k+1}M)\oplus \ker(\Delta_k),$$
where $\ker(\Delta_k):=\ker_{L^2}(\Delta_k)=\ker_{L^p}(\Delta_k)$ by assumption. Then, by definition of the $L^2$ Hodge projectors, the restriction to $E$ of the three projectors $\Pi^p_{k,d}$, $\Pi^p_{k,d^*}$ and $\Pi^p_{k,0}$ coincide with the restriction to $E$ of their $L^2$ counterparts; indeed, this is an easy consequence of the facts that $dC_c^\infty(\Lambda^{k-1}M)\subset \im_{L^2}(d_{k-1})\cap \im_{L^p}(d_{k-1})$, $d^*C_c^\infty(\Lambda^{k+1}M)\subset \im_{L^2}(d^*_{k+1})\cap \im_{L^p}(d^*_{k+1})$ and the assumption on the $L^2$ and $L^p$ kernel. Since $E$ is dense in $L^p(\Lambda^k M)$ by \eqref{eq:HodgeLp}, we conclude that $\Pi^p_{k,d}$, $\Pi^p_{k,d^*}$ and $\Pi^p_{k,0}$ are $L^p$ bounded extensions of their $L^2$ counterparts defined on $E$.

\end{proof}
There is a subtlety in the extent to which by Proposition \ref{pro:Hodge} the $L^2$ projectors have a unique extension to $L^p$, which we now explain. By Proposition \ref{pro:Hodge}, starting from one of the three $L^2$ Hodge projectors $\rho$, we have first restricted $\rho$ to a projector $\bar{\rho}$ on

$$E=dC_c^\infty(\Lambda^{k-1}M)\oplus d^*C_c^\infty(\Lambda^{k+1}M)\oplus \ker(\Delta_k),$$
then we have extended $\bar{\rho}$ by density to a projector $\hat{\rho}$ on $L^p(\Lambda^k M)$. However, it is desirable but not clear {\em a priori} that the extended projector $\hat{\rho}$ agrees with the original projector $\rho$ on $L^2\cap L^p$, that is: 

$$\hat{\rho}|_{L^2\cap L^p}=\rho|_{L^2\cap L^p},$$
Of course, $L^2(\Lambda^kM)\cap L^p(\Lambda^kM)$ is dense both in $L^2(\Lambda^kM)$ and in $L^p(\Lambda^kM)$ separately, but this is not enough in order that the above equality holds. Consider the following diagram in which each {\em black} arrow is a dense inclusion:

\begin{equation}
\begin{tikzcd}
L^2(\Lambda^kM) & & L^p(\Lambda^kM) \\
& L^2(\Lambda^kM) \cap  L^p(\Lambda^kM) \ar{ul} \ar{ur} & \\
& E \ar[uul,bend left=20] \ar[uur,bend right=20] \ar[u, red,"?" red] & 
\end{tikzcd}
\end{equation}
It turns out that the above uniqueness issue is related to the question whether the middle red arrow in the above diagram is a dense inclusion. In appendix (see Lemma \ref{lem:dense} and Corollary \ref{cor:extension} therein), we give conditions so that this is indeed the case, and we prove that indeed,

$$\hat{\rho}|_{L^2\cap L^p}=\rho|_{L^2\cap L^p}.$$
In particular, we have the following criterion for the uniqueness of the $L^p$ extension of $L^2$ Hodge projectors:

\begin{Cor}\label{cor:extension_hodgeLp}

Let $M$ be a complete manifold satisfying (H); let $p\in (1,\infty)$ and $q=p'$, and assume that

$$\ker_{L^p}(\Delta_k)=\ker_{L^q}(\Delta_k)=\ker_{L^2}(\Delta_k).$$
Suppose that \eqref{eq:HodgeLp}, the $L^p$ Hodge decomposition for forms of degree $k$ holds on $M$. Then, each one of the three $L^2$ Hodge projectors $\Pi_{k,d}$, $\Pi_{k,d^*}$ and $\Pi_{k,0}$ coincide with its $L^p$ counterpart given by Proposition \ref{pro:Hodge}, in restriction to $L^2(\Lambda^kM)\cap L^p(\Lambda^kM)$. In short, we will say that the $L^2$ projectors extend uniquely to $L^p$ bounded projectors.

\end{Cor}
Another situation that we will encounter in the case of ALE manifolds is that of a modified $L^p$ Hodge decomposition \eqref{eq:mod-HodgeLp}. This case will happen only for $k\in \{1,n-1\}$, $p\geq n$ and $N\geq 2$ ($N$ being the number of ends of $M$). The case $k=n-1$ will follow from the case $k=1$ by Hodge duality, so in the remaining part of this section, we assume that $M$ is an ALE manifold, that $p\in [n,\infty)$, that $k=1$ and that $N\geq 2$. Recall from Proposition \ref{prop:decay_harmonic_forms} and Corollary \ref{cor:decomp_kernel_1} that for $k=1$, we have the orthogonal decomposition:

$$\ker_{L^2}(\Delta_1)=d(\ker_0(\Delta_0))\oplus \ker_{-n}(\Delta_1).$$
Denote by $\Pi_{d\ker_0}$ and $\Pi_{-n}$ the $L^2$ orthogonal projectors onto $d(\ker_0(\Delta_0))$ and $\ker_{-n}(\Delta_1)$ respectively. Denote $q=p'$ the conjugate exponent, then according to Proposition \ref{prop:decay_harmonic_forms}, one has

$$\ker_{L^p}(\Delta_1)=\ker_{L^2}(\Delta_1)$$
and

$$\ker_{L^q}(\Delta_1)=\ker_{-n}(\Delta_1)\varsubsetneqq \ker_{L^2}(\Delta_1).$$
We now have the following modified version of Corollary \ref{cor:extension_hodgeLp}:

\begin{Cor}\label{cor:extension_HodgeLp-mod}

Let $M$ be ALE, $p\in [n,\infty)$; suppose that the modified Hodge decomposition \eqref{eq:mod-HodgeLp}
 for forms of degree $1$ holds on $M$. Then, each one of the three $L^2$ projectors $\Pi_{1,d}+\Pi_{d\ker_0}$, $\Pi_{1,d^*}$ and $\Pi_{-n}$ coincides with its $L^p$ counterpart $\tilde{\Pi}_{1,d}^p$, $\tilde{\Pi}_{1,d^*}^p$, $\tilde{\Pi}_{1,0}^p$ in restriction to $L^2(\Lambda^kM)\cap L^p(\Lambda^kM)$. Hence, these three $L^2$ projectors extend uniquely to $L^p$ bounded projectors.

\end{Cor}

\begin{proof}

Let 

$$\mathscr{G}=dC_c^\infty(M)\oplus d^*C_c^\infty(\Lambda^{2}M)\oplus \ker_{L^2}(\Delta_1).$$
According to the remarks above, one can also write

$$\mathscr{G}=\left(dC_c^\infty(M)\oplus d(\ker_0(\Delta_0))\right)\oplus d^*C_c^\infty(\Lambda^{2}M)\oplus \ker_{-n}(\Delta_1).$$
On other hand, the modified $L^p$ Hodge decomposition, which holds by assumption, writes as the following topological equality:

$$L^p(\Lambda^1M)=\im_{L^p}(d_0)\oplus \im_{L^p}(d_2^*)\oplus \ker_{-n}(\Delta_1).$$
We first notice that in restriction to $\mathscr{G}$, one has

$$\Pi_{1,d}+\Pi_{d\ker_0}=\tilde{\Pi}_{1,d}^p,\,\, \Pi_{1,d^*}=\tilde{\Pi}_{1,d^*}^p,\,\,\Pi_{-n}=\tilde{\Pi}_{1,0}^p.$$
Indeed, the only difference with Proposition \ref{pro:Hodge} is that, according to Proposition \ref{pro:im_p>n} we have

$$d(\ker_0(\Delta_0))\subset \im_{L^p}(d_0),$$
so the modified $L^p$ Hodge decomposition of $\omega\in d(\ker_0(\Delta_0))$ writes

$$\omega=\omega\oplus 0\oplus 0.$$
This entails that

$$\Pi_{1,d}+\Pi_{d\ker_0}=\tilde{\Pi}_{1,d}^p,$$
and the other equalities between Hodge projectors are easy to prove and left to the reader. The result now follows from Lemma \ref{lem:dense}.

\end{proof}
Finally, we conclude this section with an interpolation result that will be useful later in order to prove the closedness of $\im_{L^p}(d_{k-1})+\im_{L^p}(d_{k+1}^*)$. One will see in the next section that on an ALE manifold, for all $k\in \{1,\cdots,n-1\}$ one can find a finite dimensional space $E_{k}\subset L^n(\Lambda^kM)$ such that the following {\em weak} $L^n$ Hodge decomposition holds:

\begin{equation}\label{eq:weak_Hs}\tag{$\mathrm{w}\mathscr{H}_n$}
L^n(\Lambda^kM)=\overline{\im_{L^n}(d_{k-1})\oplus \im_{L^n}(d_{k+1}^*)}^{L^n}\oplus E_{k}.
\end{equation}
Moreover, $E_k=\ker_{L^2}(\Delta_k)$ if $k\notin\{1,n-1\}$ or if $M$ has only one end, while $E_k=\ker_{-n}(\Delta_k)$ if $k\in \{1,n-1\}$ and $M$ has at least two ends. Here, the term ``weak'' refers to the fact that one needs a closure on the exact/co-exact factor. Compare with the $L^p$-Hodge decomposition \eqref{eq:HodgeLp}  and the modified $L^p$-Hodge decomposition \eqref{eq:mod-HodgeLp}, which we may sometimes call {\em strong} decomposition to emphasize the contrast with \eqref{eq:weak_Hs}.
With this settled, we show:
\begin{Pro}\label{prop:interpolation}

Let $M$ be an ALE manifold, $1< p< n< q<+\infty$, $k\in \{1,\cdots,n-1\}$. Assume that the weak $L^n$-Hodge decomposition \eqref{eq:weak_Hs} holds, and moreover assume one of the following:

\begin{enumerate}

\item[(a)] the $L^p$- and $L^q$-Hodge decomposition hold;

\item[(b)] or $k\in\{1,n-1\}$, $2\leq p<n$, the $L^p$ Hodge decomposition holds, and the $L^q$ modified Hodge decomposition holds. 
 
\end{enumerate} 
Then, the direct sum

$$\im_{L^n}(d_{k-1})\oplus \im_{L^n}(d_{k+1}^*)$$
is closed in $L^n(\Lambda^kM)$, and the following (strong) $L^n$-Hodge decomposition holds:

$$L^n(\Lambda^kM)=\im_{L^n}(d_{k-1})\oplus \im_{L^n}(d_{k+1}^*)\oplus E_{k}.$$

\end{Pro}

\begin{proof}

We treat only the more complicated case (b) for $k=1$. The case (b) for $k=n-1$ follows by Hodge duality, and the other case (a) is similar and is left to the reader. By assumption, we have the $L^p$ Hodge decomposition, which can be written, using Proposition \ref{prop:decay_harmonic_forms} and Corollary \ref{cor:decomp_kernel_1},

\begin{eqnarray*}
L^p(\Lambda^1M)&=& \im_{L^p}(d_0)\oplus \im_{L^p}(d_2^*)\oplus \ker_{L^p}(\Delta_1)\\
&=&\left(\im_{L^p}(d_0)\oplus d(\ker_0(\Delta_0))\right) \oplus \im_{L^p}(d_2^*)\oplus \ker_{-n}(\Delta_1).
\end{eqnarray*}
Also, by assumption, one has the modified $L^q$ Hodge decomposition:

$$L^q(\Lambda^1M)=\im_{L^q}(d_0)\oplus \im_{L^p}(d_2^*)\oplus \ker_{-n}(\Delta_1),$$
and the weak $L^n$ Hodge decomposition:

$$L^n(\Lambda^kM)=\overline{\im_{L^n}(d_{k-1})\oplus \im_{L^n}(d_{k+1}^*)}^{L^n}\oplus \ker_{-n}(\Delta_1).$$
Recall also that according to Proposition \ref{pro:im_p>n},

$$d(\ker_0(\Delta_0))\subset \im_{L^n}(d_0).$$
Consider the $L^2$ Hodge projector

$$\Pi:=\Pi_{1,d}+\Pi_{d\ker_0}$$
onto

$$\im_{L^2}(d_0)\oplus d(\ker_0(\Delta_0)).$$
According to Corollary \ref{cor:extension_HodgeLp-mod}, $\Pi|_{L^2\cap L^q}$ extends uniquely to a bounded $L^q$ projector, which we will denote $\Pi^q$. Furthermore, according to Corollary \ref{cor:extension_hodgeLp}, the $L^2$ projector $\Pi_{1,d}$, in restriction to $L^2\cap L^p$, extends uniquely to a bounded $L^p$. Moreover, by writing

$$\Pi_{d\ker_0}=\sum_{i=1}^{N-1}\langle \omega_i,\cdot\rangle \omega_i$$
for an $L^2$ orthonormal basis $(\omega_1,\cdots,\omega_{N-1})$ of $d(\ker_0(\Delta_0))$, one checks easily that $\Pi_{d\ker_0}$ is bounded from $L^p$ to $L^p$ provided

$$\omega_i\subset L^p\cap L^{p'},\quad \forall i\in\{1,\cdots,N-1\},$$
and the latter holds since $d(\ker_0(\Delta_0))\subset \ker_{1-n}(\Delta)$ (Lemma \ref{lem:expansion_fn}) and $2\leq p<n$. So, $\Pi|_{L^2\cap L^p}$, extends uniquely to a bounded projector $\Pi^p$ on $L^p$. If $\omega\in L^p(\Lambda^1M)\cap L^q(\Lambda^1M)$, there exists a sequence $(\omega_n)_{n\in \N}$ with $\omega_n\in L^2\cap L^p\cap L^q$ such that

$$L^p-\lim_{n\to\infty}\omega_n=\omega,\quad L^q-\lim_{n\to\infty}\omega_n=\omega.$$
Indeed, just take an non-decreasing exhaustion $(\Omega_n)_{n\in \N}$ of $M$, and define $\omega_n=\omega\cdot \mathbf{1}_{\Omega_n}$. Since

$$\Pi^p|_{L^2\cap L^p\cap L^q}=\Pi^q|_{L^2\cap L^p\cap L^q}=\Pi|_{L^2\cap L^p\cap L^q},$$
one has for every $n\in \N$,

$$\Pi^p(\omega_n)=\Pi^q(\omega_n),$$
and passing to the limit as $n\to\infty$, one gets

$$\Pi^p(\omega)=\Pi^q(\omega).$$
Thus,

$$\Pi^p|_{L^p\cap L^q}=\Pi^q|_{L^p\cap L^q},$$
and extends uniquely to a bounded operator on $L^p$ and $L^q$. The Riesz-Thorin interpolation theorem now implies that $\Pi$ extends uniquely to a bounded projector on $L^n$.  

A completely similar interpolation argument shows that $\Pi_{1,d^*}$ extends uniquely to a bounded projector on $L^n$. We have the following lemma, whose proof is postponed to the end of the proof of the proposition:

\begin{Lem}\label{lem:proj_id}

The following identities hold:

$$\im_{L^n}(\Pi)=\im_{L^n}(d_0),\quad \im_{L^n}(\Pi_{1,d^*})=\im_{L^n}(d_2^*),$$
and

$$\im_{L^n}(d_2^*)\subset \ker_{L^n}(\Pi),\quad \im_{L^n}(d_0)\subset \ker_{L^n}(\Pi_{1,d^*}).$$

\end{Lem}
Now, let $\omega\in \overline{\im_{L^n}(d_0)\oplus \im_{L^n}(d_2^*)}^{L^n}$. One finds a sequence $(\omega_j)_{j\in \N}\subset \im_{L^n}(d_0)\oplus \im_{L^n}(d_2^*) $ and forms $\eta_j\in \im_{L^n}(d_0)$, $\nu_j\in \im_{L^n}(d_2^*)$ such that for all $j\in \N$, $\omega_j=\eta_j+\nu_j$ and

$$||\omega_j-\omega||_n\to  0,\quad \mbox{as }j\to\infty.$$
By Lemma \ref{lem:proj_id}, since the projectors $\Pi$ and $\Pi_{1,d^*}$ are bounded on $L^n$,

$$\eta_j=\Pi(\omega_j)\underset{L^n}{\longrightarrow}  \Pi(\omega)=:\eta\in \im_{L^n}(d_0) \quad \mbox{as }j\to\infty,$$
and

 $$\nu_j=\Pi_{1,d^*}(\omega_j)\underset{L^n}{\longrightarrow} \Pi_{1,d^*}(\omega)=:\nu\in \im_{L^n}(d_2^*) \quad \mbox{as }j\to\infty.$$
Passing to the limit in the identity $\omega_j=\eta_j+\nu_j$ as $j\to\infty$, one gets

$$\omega=\eta+\nu\in \im_{L^n}(d_0)\oplus \im_{L^n}(d_2^*).$$
Thus, 

$$\overline{\im_{L^n}(d_0)\oplus \im_{L^n}(d_2^*)}^{L^n}=\im_{L^n}(d_0)\oplus \im_{L^n}(d_2^*),$$
which shows that $\im_{L^n}(d_0)\oplus \im_{L^n}(d_2^*)$ is closed. The strong $L^n$ Hodge decomposition now follows directly from the weak one.

\end{proof}

\begin{proof}[Proof of Lemma \ref{lem:proj_id}:]

We first show that $\im_{L^n}(\Pi)=\im_{L^n}(d_0)$. Since the projector $\Pi$ is bounded on $L^n$, $\im_{L^n}(\Pi)=\ker_{L^n}(I-\Pi)$ is a closed subspace of $L^n$. Furthermore, since $\Pi$ coincides with $\Pi_{1,d}+\Pi_{d\ker_0}$ on $L^2\cap L^n$,

$$dC_c^\infty(M)\oplus d(\ker_0(\Delta_0))\subset \im_{L^n}(\Pi),$$
thus by taking the closure in $L^n$ and using the fact that according to Proposition \ref{pro:im_p>n},

$$d(\ker_0(\Delta_0))\subset \im_{L^n}(d_0),$$
one concludes that

$$\im_{L^n}(d_0)\subset \im_{L^n}(\Pi).$$
We now prove the converse inclusion. Denote 

$$E:=dC_c^\infty(M)\oplus d(\ker_0(\Delta_0)),\quad F=d^*C_c^\infty(\Lambda^2M)\oplus \ker_{-n}(\Delta_1).$$
By definition of $\Pi$, 

$$\Pi(E\oplus F)=E\subset \im_{L^n}(d_0).$$
Since by definition $\im_{L^n}(d_0)$ is closed in $L^n$, by taking the closure in $L^n$ of the above inclusion one gets

$$\overline{\Pi(E\oplus F)}^{L^n}\subset \im_{L^n}(d_0).$$
However, since $\Pi$ is bounded,

$$\Pi\left(\overline{E\oplus F}^{L^n}\right)\subset\overline{\Pi(E\oplus F)}^{L^n},$$
hence

$$\Pi\left(\overline{E\oplus F}^{L^n}\right)\subset \im_{L^n}(d_0).$$
But the weak $L^n$ Hodge decomposition implies that $L^n(\Lambda^1M)=\overline{E\oplus F}^{L^n}$, so the above inclusion yields

$$\im_{L^n}(\Pi)\subset \im_{L^n}(d_0).$$
Therefore, we get the equality $\im_{L^n}(\Pi)= \im_{L^n}(d_0)$ which ends the first part of the proof.

The proof that $\im_{L^n}(\Pi_{1,d^*})=\im_{L^n}(d_2^*)$ is completely similar and is skipped. Concerning the last two inclusions: we have

$$d_0C_c^\infty(M)\subset \ker_{L^n}(\Pi_{1,d^*}),\quad d^*_2C_c^\infty(\Lambda^2M)\subset \ker_{L^n}(\Pi),$$
and since the two projectors $\Pi$, $\Pi_{1,d^*}$ are bounded on $L^n$, their kernels are closed, so by taking the closure of the above inclusions, one gets

$$\im_{L^n}(d_2^*)\subset \ker_{L^n}(\Pi),\quad \im_{L^n}(d_0)\subset \ker_{L^n}(\Pi_{1,d^*}),$$
and this completes the proof of Lemma \ref{lem:proj_id}.

\end{proof}

\section{$L^p$ cohomology and Hodge decomposition}

Recall that the reduced $L^p$-cohomology vector spaces are defined as
\begin{align*}
	H^k_p(M):=\frac{\ker_{L^p}(d_{k-1})}{\im_{L^p}(d_k)},
\end{align*}
where $k\in\{0,\cdots,n\}$ is an integer. In the following, we wish to compute these $L^p$ cohomology spaces, and more precisely we wish to relate $H^k_p(M)$ to $L^p$-harmonic $k$-forms. We will first establish {\em weak} $L^p$-Hodge decompositions, where the term ``weak'' refers to the fact that the factor $\im_{L^p}(d_{k-1})\oplus \im_{L^p}(d^*_{k+1})$ is replaced by $\overline{\im_{L^p}(d_{k-1})\oplus \im_{L^p}(d^*_{k+1})}^{L^p}$. Then, we will show how to turn these into {\em strong} $L^p$-Hodge decomposition, and study the consequences for the $L^p$-cohomology.

First we recall some notations: for a subspace $V$ of a Banach space $X$, its {\em annihilator} is defined by
\begin{align*}
	\mathrm{Ann}(V)=\left\{x^*\in X^*\mid x^*(v)=0\text{ for all } v\in V\right\}\subset X^*.
\end{align*}
If $p\in (1,\infty)$, $X$ is an $L^p$ space, $q=p'$ is the conjugate exponent and $V\subset X$ is a subspace, we use for the sake of clarity the notation
\begin{align*}
		\mathrm{Ann}_{L^q}(V):=\mathrm{Ann}(V)\subset L^q
\end{align*}
for the annihilator, in order to emphasise that it is a subspace of $L^q$.
If $p,q\in (1,\infty)$ with $\frac{1}{p}+\frac{1}{q}=1$, we clearly have
\[
\mathrm{Ann}_{L^q}(\im_{L^p}(d_{k-1}))=\ker_{L^q}(d^*_k),\qquad \mathrm{Ann}_{L^q}(\im_{L^p}(d^*_{k+1}))=\ker_{L^q}(d_k).
\]
In this section, we will prove both weak and strong $L^p$ Hodge decompositions; the proofs will rely in particular on $L^q-L^p$ duality. In order to prove weak $L^p$ Hodge decompositions, one will often use the following well-known lemma:

\begin{lem}\label{lem:sum_finite_dim}

Let $\mathscr{B}$ be a Banach space, $E$ and $F$ two closed subspaces of $\mathscr{B}$ such that $E\cap F=\{0\}$ and $F$ is finite dimensional. Then, the sum $E\oplus F$ is closed.

\end{lem}
We point out that in general, the result of the lemma is {\em false} if $F$ is not finite dimensional. Since Lemma \ref{lem:sum_finite_dim} is instrumental in the present article, for the sake of completeness we give a sketch of its proof.

\begin{proof}[Proof of Lemma \ref{lem:sum_finite_dim}]

It is enough to prove that the projection

$$\pi : E\oplus F\to F$$
which is defined by $\pi(z)=y$ for all $z=x+y\in E\oplus F$, is continuous. Indeed, if it is the case, then $\pi$ and $id-\pi$ both extend to bounded projectors on $\mathscr{B}=\overline{E\oplus F}$, which satisfy $\im(\pi)=F$, $\im(id-\pi)=E$ since $E$ and $F$ are closed; thus, if $\{z_n\}_{n\in\mathbb{N}}$ is a sequence of $E\oplus F$ which converges to $z$, then $\pi(z_n)$ and $(id-\pi)(z_n)$ converge respectively to $e\in E$ and $f\in F$, and passing to the limit in the equality $z_n=\pi(z_n)+(id-\pi)(z_n)$, we have $z=e+f\in E\oplus F$, which concludes the proof.

To prove that $\pi$ is continuous on $E\oplus F$, we estimate its operator norm:

$$|||\pi|||=\sup_{x+y\in E\oplus F\setminus \{0\}}\frac{||y||}{||x+y||},$$
and by homogeneity and the symmetry $y\mapsto -y$ (which preserves $F$), it is also equal to

$$|||\pi|||=\sup_{x+y\in E\oplus F,\,||y||=1}\frac{1}{||x-y||}=\left(\inf_{x\in E,\,y\in F,\,||y||=1}||x-y||\right)^{-1}.$$
The above infimum is precisely equal to the distance from $E$ to the unit sphere in $F$, which is compact since $F$ is finite dimensional. But the distance between a compact set and a closed set that are disjoint being $>0$, we conclude that $|||\pi|||<+\infty$, and $\pi$ is bounded.







\end{proof}
Let us now continue with a useful lemma:

\begin{lem}\label{lem : de Rham intersection}
	Let $p,q\in (1,\infty)$ such that $\frac{1}{p}+\frac{1}{q}=1$. Then,
	\begin{align}\label{eq : de Rham intersection 1}
		\overline{\im_{L^p}(d_{k-1})+\im_{L^p}(d^*_{k+1})}^{L^p}\cap \ker_{L^p}(\Delta_{k})\cap \ker_{L^q}(\Delta_{k})=\left\{0\right\}
	\end{align}
	Moreover, if $\ker_{L^p}(\Delta_{k})\subset \ker_{L^q}(\Delta_{k})$, we have
	\begin{align}\label{eq : de Rham intersection 2}
		\im_{L^p}(d_{k-1})\cap \im_{L^p}(d^*_{k+1})
	=\left\{0\right\}.
	\end{align}
\end{lem}
\begin{proof}
	Let $\omega \in \ker_{L^p}(\Delta_{k})\cap \ker_{L^q}(\Delta_{k})$,  and suppose that there exist sequences  $\alpha_i\in C^{\infty}_{c}(\Lambda^{k-1}M)$, $\beta_i\in C^{\infty}_{c}(\Lambda^{k+1}M)$, $i\in\N$  such that we have $d\alpha_i+d^*\beta_i\to\omega$ in $L^p$. By integration by parts and Proposition \ref{prop:decay_harmonic_forms},
	\begin{align*}
	 (d\alpha_i,\omega)_{L^2}=(\alpha_i,d^*\omega)_{L^2}=0,\qquad (d^*\beta_i,\omega)_{L^2}=(\beta_i,d\omega)_{L^2}=0.
	 \end{align*}
Therefore,
	$$
	\norm{\omega}_{L^2}^2=(\omega-d\alpha_i-d^*\beta_i,\omega)_{L^2}\leq\norm{\omega-(d\alpha_i+d^*\beta_i)}_{L^p}\norm{\omega}_{L^q}\to0.
	$$
	Thus, $\omega=0$, which proves \eqref{eq : de Rham intersection 1}.  Concerning \eqref{eq : de Rham intersection 2}, we notice that as a consequence of the fact that $d^2=0$, $(d^*)^2=0$, we have the inclusions $\im_{L^p}(d_{k-1})\subset \ker_{L^p}(d_k)$ and $\im_{L^p}(d^*_{k+1})\subset \ker_{L^p}(d^*_k)$. Indeed, if $\eta\in \im_{L^p}(d_{k-1})$, $\eta=L^p-\lim_{j\to\infty} (d\psi_j)$, $\psi_j\in C_c^\infty(\Lambda^{k-1}M)$, and $\varphi\in C_c^\infty(\Lambda^{k-1})$, then

$$\langle \eta,d^*\varphi\rangle=\lim_{j\to\infty}\langle d\psi_j,d^*\varphi\rangle=\lim_{j\to\infty}\langle (d^2)\psi_j,\varphi\rangle=0,$$
so $d\eta=0$ in the weak sense, and thus $\eta \in \ker_{L^p}(d_k)$. Similarly, one shows that 	$\im_{L^p}(d^*_{k+1})\subset \ker_{L^p}(d^*_k)$. Hence we conclude that
$$\im_{L^p}(d_{k-1})\cap \im_{L^p}(d^*_{k+1})\subset \ker_{L^p}(\Delta_k),$$
therefore
$$\im_{L^p}(d_{k-1})\cap \im_{L^p}(d^*_{k+1})\subset (\im_{L^p}(d_{k-1})+\im_{L^p}(d^*_{k+1}))\cap \ker_{L^p}(\Delta_{k})\cap \ker_{L^q}(\Delta_{k}),$$
and the right hand side is $\{0\}$ according to \eqref{eq : de Rham intersection 1}.
\end{proof}

{The following theorem provides a weak Hodge-De Rham decomposition, for some values of $p$ and $k$:}
\begin{Lem}\label{lem : de Rham sum}
Let $p\in (1,\infty)$ and $1\leq k\leq n-1$ be such that one of the following holds:

\begin{enumerate}
\item[(i)] $2\leq k\leq n-2$.

\item[(ii)] $k\in\left\{1,n-1\right\}$ and $p\in (\frac{n}{n-1},n)$.

\item[(iii)] $k\in\left\{1,n-1\right\}$, $p\in (1,\infty)$. and $M$ has only one end.

\end{enumerate}
Then we have
	\begin{align}\label{eq : de Rham sum}
		L^p(\Lambda^kM)=\overline{\im_{L^p}(d_{k-1})\oplus \im_{L^p}(d^*_{k+1})}^{L^p}\oplus \ker_{L^p}(\Delta_{k}).
	\end{align}
\end{Lem}

\begin{proof}
	By Proposition \ref{prop:decay_harmonic_forms}, Corollary \ref{cor:decomp_kernel_1} and Corollary \ref{cor:dimension_d_harmonic functions}, if one lets $q=p'$ the conjugate exponent, we have
	$\ker_{L^p}(\Delta_{k})=\ker_{L^q}(\Delta_{k})$ by the assumptions on $p$, $k$ and $M$. 
	By Lemma \ref{lem : de Rham intersection}, the sum on the right hand side of \eqref{eq : de Rham sum} is indeed direct. Moreover, according to Lemma \ref{lem:sum_finite_dim}, since $\ker_{L^p}(\Delta_k)$ is finite dimensional the direct sum of closed subspaces $\overline{\im_{L^p}(d_{k-1})\oplus \im_{L^p}(d^*_{k+1})}^{L^p}\oplus \ker_{L^p}(\Delta_{k})$ is closed. Hence, in order to finish the proof, it suffices to show that the annihilator (in $L^q$) of the direct sum vanishes. 
	We have
	\begin{align*}
	&\mathrm{Ann}_{L^q}(\overline{\im_{L^p}(d_{k-1})\oplus \im_{L^p}(d^*_{k+1})}^{L^p}\oplus \ker_{L^p}(\Delta_{k}))\\
		&\mathrm{Ann}_{L^q}(\im_{L^p}(d_{k-1})\oplus \im_{L^p}(d^*_{k+1})\oplus \ker_{L^p}(\Delta_{k}))\\
		&\qquad=\ker_{L^q}(d^*_k)\cap \ker_{L^q}(d_k)\cap\mathrm{Ann}_{L^q}( \ker_{L^p}(\Delta_{k}))\\
		&\qquad= \ker_{L^q}(\Delta_{k})\cap\mathrm{Ann}_{L^q}(\ker_{L^p}(\Delta_{k}))
	\end{align*}
	However, the right hand side is zero as $\ker_{L^p}(\Delta_{k})=\ker_{L^q}(\Delta_{k})$. This finishes the proof of the lemma.
\end{proof}
Lemma \ref{lem : de Rham sum} indicates that the hardest task for proving the (strong) $L^p$ Hodge decomposition is to prove that the direct sum $\im_{L^p}(d_{k-1})\oplus \im_{L^p}(d^*_{k+1})$ is closed.

We now focus on one of the remaining cases: $k\in \{1,n-1\}$, $p\in [n,\infty)$ (and $M$ has $N\geq 2$ ends). Note that, as a consequence of Proposition \ref{prop:decay_harmonic_forms}, we have 
\begin{equation}\label{eq : decay kernel}
	\begin{split}
		\ker_{L^p}(\Delta_{1})&=\ker_{1-n}(\Delta_{1}),\text{ if }p\in \left(\frac{n}{n-1},\infty\right),\\
		\ker_{L^p}(\Delta_{1})&=\ker_{-n}(\Delta_{1}),\text{ if }p\in \left(1,\frac{n}{n-1}\right].
	\end{split}
\end{equation}
The same is true for $\ker_{L^p}(\Delta_{n-1})$ by Hodge duality. Similarly to the previous lemma, one has:

\begin{lem}\label{lem de Rham sum large p}
	Let $p\in [n,\infty)$. Then we have, for $k\in \{1,n-1\}$:
	\begin{align}\label{eq : de Rham sum large p}
		L^p(\Lambda^kM)=\overline{\im_{L^p}(d_{k-1})+ \im_{L^p}(d^*_{k+1})}^{L^p}\oplus \ker_{-n}(\Delta_{k}).
	\end{align}
\end{lem}
\begin{proof}
	By Hodge duality, it suffices to look at the case $k=1$. Let $q\in (1,\frac{n}{n-1}]$ be such that $1=\frac{1}{p}+\frac{1}{q}$. Then, 
	$\ker_{{-n}}(\Delta_{1})=\ker_{L^q}(\Delta_{1})= \ker_{L^p}(\Delta_{1})\cap \ker_{L^q}(\Delta_{1})$ and from Lemma \ref{lem : de Rham intersection}, we get that
	\[
	\overline{\im_{L^p}(d_{0})+ \im_{L^p}(d^*_{2})}^{L^p}\cap \ker_{-n}(\Delta_{1})=\left\{0\right\},
	\]
	so that the sum on the right hand side of \eqref{eq : de Rham sum large p} is indeed direct; it is also closed, according to Lemma \ref{lem:sum_finite_dim} since $\ker_{-n}(\Delta_1)$ is finite dimensional (cf point (e) in Proposition \ref{prop:decay_harmonic_forms}). According to Proposition \ref{prop:decay_harmonic_forms} (a), the annihilator of the sum is given by
	\begin{align*}
	&\mathrm{Ann}_{L^q}(\overline{\im_{L^p}(d_{0})+ \im_{L^p}(d^*_{2})}^{L^p}\oplus \ker_{-n}(\Delta_{1}))\\
		&\mathrm{Ann}_{L^q}((\im_{L^p}(d_{0})+ \im_{L^p}(d^*_{2}))\oplus \ker_{-n}(\Delta_{1}))\\
		&\qquad=\ker_{L^q}(d_1)\cap\ker_{L^q}(d^*_1)\cap \mathrm{Ann}_{L^q}(\ker_{-n}(\Delta_{1}))\\
		&\qquad=\ker_{L^q}(\Delta_{1})\cap \mathrm{Ann}_{L^q}(\ker_{-n}(\Delta_{1}))\\
		&\qquad=\ker_{{-n}}(\Delta_{1})\cap \mathrm{Ann}_{L^q}(\ker_{-n}(\Delta_{1}))=\left\{0\right\}.
	\end{align*}
	The latter implies that \eqref{eq : de Rham sum large p} holds
\end{proof}
We wish to prove now that for $k\in \{1,n-1\}$, the sum $\im_{L^p}(d_{k-1})+ \im_{L^p}(d^*_{k+1})$ is direct, so that \eqref{eq : de Rham sum large p} is a weak modified $L^p$ Hodge decomposition. It is done in points (iii) and (iv) of the next:

\begin{Pro}\label{pro:Lp1}

Let $p\in[ n,+\infty)$. Then, the following hold:

\begin{enumerate}

\item[(i)] $\ker_{1-n}(\Delta_1)\cap \im_{L^p}(d_2^*)\subset \ker_{-n}(\Delta_1)$.

\item[(ii)] $\ker_{L^p}(\Delta_1)\cap \im_{L^p}(d_2^*)=\{0\}$.

\item[(iii)] $\im_{L^p}(d_2^*)\cap \im_{L^p}(d_0)=\{0\}.$

\item[(iv)] $\im_{L^p}(d_{n-2})\cap \im_{L^p}(d_n^*)=\{0\}$.
\end{enumerate}

\end{Pro}

\begin{proof}
Let $\omega\in \ker_{1-n}(\Delta_1)\cap \im_{L^p}(d_2^*)$ (in particular $\omega$ is smooth). 
%
By Hodge duality, we have $*\omega\in \im_{L^p}(d_{n-2})$, and according to Proposition \ref{prop:decay_harmonic_forms}, $\omega$ is closed and co-closed, hence $*\omega$ too. According to the proof of \cite[Lemma 1.1]{C6}, there exists $\eta\in C^\infty(\Lambda^{n-2}M)$ such that $*\omega=d\eta$. For the sake of completeness, let us detail this point. Recall the Poincar\'e duality map for oriented manifolds:

$$
\begin{array}{rrcl}
F \,: & H^k(M)\times H_c^{n-k}(M) & \rightarrow & \R\\
& ([\alpha],[\beta]) & \mapsto & \int_M \alpha\wedge \beta
\end{array}
$$
Here, by definition

$$H^j(M)=\frac{\{\alpha\in C^\infty(\Lambda^jM)\,;\,d\alpha=0\}}{dC^\infty(\Lambda^{j-1}M)},$$
is the usual cohomology space of $M$, and

$$H_{c}^{j}(M)=\frac{\{\alpha\in C_c^\infty(\Lambda^jM)\,;\,d\alpha=0\}}{dC_c^\infty(\Lambda^{j-1}M)},$$
is the cohomology space with compact support. Poincar\'e duality for non-compact manifolds (see \cite[p.248]{H} and \cite[Section 1.1.2]{C4}) asserts that if $M$ is oriented then $F$ is well-defined and is a duality pairing between $H^k(M)$ and $H_{c}^{n-k}(M)$. By hypothesis, there is a sequence $(\eta_j)_{j\in\N}$ of forms in $C_c^\infty(\Lambda^{2}M)$ such that $d^*\eta_j$ converges in $L^p$ to $\omega$. Let $\beta\in C_c^\infty(\Lambda^1M)$ such that $d\beta=0$, then

$$
\begin{array}{rcl}
F(*\omega,\beta)&=&\int_M*\omega\wedge \beta\\
&=& \langle \omega,\beta\rangle_{L^2}\\
&=& \lim_{j\to\infty} \langle d^*\eta_j,\beta\rangle_{L^2}\\
&=& \lim_{j\to\infty} \langle \eta_j,d\beta\rangle_{L^2}\\
&=&0,
\end{array}
$$
where we have used the fact that since $p\geq 2$, $L^p$ convergence on a compact set implies $L^2$ convergence. Thus, one concludes that $*\omega$ is zero in $H^{n-1}(M)$, that is there exists $\eta\in C^\infty(\Lambda^{n-2}M)$ such that $*\omega=d\eta$. 

\medskip

Let $\Sigma^{n-1}\subset M$ be a {\em closed}, smooth hypersurface, then by Stokes' theorem
$$\int_{\Sigma^{n-1}}*\omega=\int_{\Sigma^{n-1}}d\eta=\int_{\partial \Sigma^{n-1}}\eta=0.$$
By Lemma \ref{lem:asym_forms}, we can at each end $E_i$ expand $\omega$ as
$$\omega=B_jr^{1-n}dr+\mathcal{O}_\infty(r^{1-n-\epsilon}).$$
Thus, if we fix an end $E_i$, and we take the hypersurface

$$\Sigma^{n-1}=\phi_i^{-1}(S^{n-1}(0,R)/\Gamma_i),$$
 for some large enough $R>0$ so that this is well-defined inside $E_i$, then 
%
we get, using the asymptotic of the metric $g$ and the form $\omega$,
$$0=\int_{\Sigma^{n-1}}*\omega=B_i\cdot\mathrm{Vol}({\bf S}^{n-1}/\Gamma_i)+O(R^{-\epsilon}).$$
Consequently, by letting $R\to \infty$, one concludes that $B_i=0$ for every $i\in\{1,\cdots,N\}$, and therefore $\omega\in \ker_{1-n-\epsilon}(\Delta_1)$. According to Corollary \ref{cor:decay_harmonic}, we can conclude that $\omega\in \ker_{-n}(\Delta_1)$. Thus, part (i) is proved.

\medskip

In order to prove part (ii), we first notice that (i) implies that
$$\ker_{1-n}(\Delta_1)\cap \im_{L^p}(d_2^*)\subset \im_{L^p}(d_2^*)\cap \ker_{L^p}(\Delta_1)\cap \ker_{L^q}(\Delta_1),$$
with $q$ the conjugate exponent of $p$. But according to Lemma \ref{lem : de Rham intersection}, the intersection on the right hand side is $\{0\}$, and part (ii) of the lemma follows.

Concerning point (iii), 
we clearly have
$$\im_{L^p}(d_2^*)\cap \im_{L^p}(d_0)\subset \ker_{L^p}(\Delta_1)\cap \im_{L^p}(d_2^*),$$
and the right hand side is zero by part (ii). 

Finally, for point (iv), we rely on (iii) and Hodge duality: indeed, using that $d^*=\pm * d*$ and $*^2=\pm \mathrm{id}$ (where the signs depend on the degree of the form), one easily sees that

$$*\im_{L^p}(d_0)=\im_{L^p}(d_n^*),\quad *\im_{L^p}(d_2^*)=\im_{L^p}(d_{n-2}),$$
and since the Hodge star is an isometry, we see that (iii) is equivalent to (iv).
\end{proof}

Note that, according to Corollary \ref{cor:missing_ker}, if $M$ has only one end and $k\in\{1,n-1\}$ then

$$\ker_{-n}(\Delta_k)=\ker_{L^p}(\Delta_k)$$
for any $p\in (1,\infty)$. According to Lemma \ref{lem : de Rham intersection}, this implies
\begin{align*}
		\im_{L^p}(d_{0})\cap \im_{L^p}(d^*_{2})
	=\left\{0\right\}
	\end{align*}	
for $p>\frac{n}{n-1}$.	
 And moreover, by Corollaries \ref{cor:dimension_d_harmonic functions} and \ref{cor:decomp_kernel_1}, if $M$ has at least two ends and $p\geq n$, then

$$\ker_{-n}(\Delta_k) \varsubsetneqq \ker_{L^p}(\Delta_k),$$
so the weak modified $L^p$ Hodge decomposition is not a weak $L^p$ Hodge decomposition. Therefore, we can summarize the results of Lemmas \ref{lem de Rham sum large p} and \ref{lem : de Rham sum} and Proposition \ref{pro:Lp1} as follows:

\begin{Cor}\label{cor:weak_hodge}

Let $M$ be an ALE manifold, then, the weak $L^p$-Hodge decomposition
	\begin{align*}
		L^p(\Lambda^kM)=\overline{\im_{L^p}(d_{k-1})\oplus \im_{L^p}(d^*_{k+1})}^{L^p}\oplus \ker_{L^p}(\Delta_{k}).
	\end{align*}
 for forms of degree $k$ on $M$ holds if either

\begin{enumerate}

\item[(a)] $p\in (1,\infty)$, $k\in \{2,\cdots,n-2\}$,

\item[(b)] or $p\in (\frac{n}{n-1},n)$, $k\in \{1,n-1\}$,

\item[(c)] or $p\in [n,\infty)$, $k\in \{1,n-1\}$ and $M$ has only one end.

\end{enumerate}
Moreover, if $p\in [n,\infty)$, $k\in \{1,n-1\}$, and $M$ has at least two ends, then the weak modified $L^p$-Hodge decomposition
	\begin{align*}
		L^p(\Lambda^kM)=\overline{\im_{L^p}(d_{k-1})\oplus \im_{L^p}(d^*_{k+1})}^{L^p}\oplus \ker_{-n}(\Delta_{k}).
	\end{align*}
 holds, while the weak $L^p$-Hodge decomposition does not.

\end{Cor}
Note that the case $k\in \{1,n-1\}$, $p\in (1\frac{n}{n-1}]$ is not covered by the above corollary. The next result, which is key, will allow us to get corresponding {\em strong} Hodge decompositions from the weak ones:

\begin{Thm}\label{thm:closed sum}

For every $p\in (1,\infty)$ and $k\in \{1,\cdots,n-1\}$, the space $\im_{L^p}(d_{k-1})+\im_{L^p}(d_{k+1})$ is closed in $L^p(\Lambda^kM)$.

\end{Thm}

\begin{proof}

The proof is split into two parts. First, we prove the result for $p\neq n$. In this case, the result follows from:

\begin{proposition}\label{prop:closed_subspace}
We have, for $p\in (1,\infty)$ and $p\neq n$ that
\begin{align*}
\mathrm{Ann}_{L^p}(\ker_{L^q}(\Delta_k))=\im_{L^p}(d_{k-1})+ \im_{L^p}(d^*_{k+1}).
\end{align*}
\end{proposition}

The proof of Proposition \ref{prop:closed_subspace} (which is new) uses the theory of weighted Sobolev spaces; in order not to disturb the flow of the presentation, we chose to postpone it to Appendix B, where the theory of weighted Sobolev spaces on ALE manifolds will also be briefly recalled for the convenience of the reader. Since an annihilator is always closed, the result follows in the case $p\neq n$.

Let us now prove the result for $p=n$. Note that Proposition \ref{prop:closed_subspace} implies strong $L^p$ Hodge decompositions and modified Hodge decompositions in any case of Corollary \ref{cor:weak_hodge} if $p\neq n$. Using Proposition \ref{prop:interpolation} with $2\leq p<n<q<\infty$, one concludes that 

$$\im_{L^n}(d_{k-1})+\im_{L^n}(d_{k+1}^*)$$
is also closed in $L^n$, for every $k\in \{1,\cdots,n-1\}$. This concludes the proof.

\end{proof}

\begin{Cor}\label{cor:strong_hodge}

Let $M$ be an ALE manifold, then, the strong $L^p$ Hodge decompositions for an ALE manifold \eqref{eq:HodgeLp} holds if either

\begin{enumerate}

\item[(a)] $p\in (1,\infty)$, $k\in \{2,\cdots,n-2\}$,

\item[(b)] or $p\in (\frac{n}{n-1},n)$, $k\in \{1,n-1\}$,

\item[(c)] or $p\in (1,\infty)$, $k\in \{1,n-1\}$ and $M$ has only one end.

\end{enumerate}
Moreover, if $k\in\{1,n-1\}$, $p\in [n,\infty)$, and $M$ has at least two ends, then the strong modified $L^p$ Hodge decomposition \eqref{eq:mod-HodgeLp} holds.

\end{Cor}
From this result, one can already compute the $L^p$ cohomology spaces in most of the cases:

\begin{theorem}\label{mainthm:part_1}
	Let $M$ be ALE, and let $p\in (1,\infty)$ and $k\in\{1,\cdots,n-1\}$ such that one of the following holds:
\begin{enumerate}

\item[(a)] $p\in (1,\infty)$, $k\in \{2,\cdots,n-2\}$,

\item[(b)] or $p\in (\frac{n}{n-1},n)$, $k\in \{1,n-1\}$,

\item[(c)] or $p\in [n,\infty)$, $k\in \{1,n-1\}$ and $M$ has only one end.

\end{enumerate}	
Then,
	\[
	\overline{H}_p^k(M)\cong \mathcal{H}_k(M),
	\]
(recall that this latter space is by definition $\ker_{L^2}(\Delta_k)$).
\end{theorem}
\begin{proof}By the assumptions on $p$ and $k$,
	\begin{align*}
	\mathcal{H}_k(M)=\ker_{L^2}(\Delta_k)=\ker_{L^p}(\Delta_k).	
	\end{align*}
    Thus by definition of the $L^p$-cohomology, it suffices to show
    \begin{align}\label{eq : closed forms decomp}
    	\ker_{L^p}(d_{k})=\im_{L^p}(d_{k-1})\oplus \ker_{L^p}(\Delta_{k}).
    \end{align}
    We clearly have $\im_{L^p}(d_{k-1})\oplus \ker_{L^p}(\Delta_{k})\subset \ker_{L^p}(d_{k})$. By Theorem \ref{thm:closed sum} and \eqref{eq : de Rham sum}, we have 
    \begin{align*}
		L^p(\Lambda^kM)=\im_{L^p}(d_{k-1})\oplus \im_{L^p}(d^*_{k+1})\oplus \ker_{L^p}(\Delta_{k})   , 
\end{align*}    
    which, intersected with $\ker_{L^p}(d_{k})$ yields \eqref{eq : closed forms decomp}, provided that $
    \ker_{L^p}(d_k)\cap\im_{L^p}(d_{k+1}^*)=\{0\}$.
	To show the latter, observe that 
	\[
	\ker_{L^p}(d_k)\cap\im_{L^p}(d_{k+1}^*)\subset \ker_{L^p}(d_k)\cap\ker_{L^p}(d_{k}^*)\subset \ker_{L^p}(\Delta_{k}).
	\]
	But again by the assumptions on $p$ and $k$, $\ker_{L^p}(\Delta_{k})=\ker_{L^q}(\Delta_{k})$, where $q$ is the conjugate H\"{o}lder exponent.
	Thus,
	\[
\ker_{L^p}(d_k)\cap\im_{L^p}(d_{k+1}^*)=\ker_{L^p}(d_k)\cap\im_{L^p}(d_{k+1}^*)\cap \ker_{L^q}(\Delta_{k})\cap \ker_{L^p}(\Delta_k),
\]	
	which is $\{0\}$ by  \eqref{eq : de Rham intersection 1}. This finishes the proof of the theorem.
	\end{proof}

The rest of the section is devoted to the computation of $H^k_p(M)$ and the proof of the $L^p$ Hodge decompositions in the remaining cases. Concerning the cohomology spaces, by Hodge duality one sees that it is enough to focus on the case $k=1$. Indeed, one has the following proposition, which is well-known (see the main result in \cite{GolT} as well as Lemma 4 therein):

\begin{Pro}\label{prop:duality}

Let $p\in (1,\infty)$, $q=p'$ the conjugate exponent, and $k\in \{0,\cdots,n\}$. Then, 

$$\overline{H}^k_p(M)\simeq \overline{H}^{n-k}_q(M).$$ 

\end{Pro}

\begin{proof}

Using the fact that $d^*=\pm *d*$ and $*^2=\pm \mathrm{id}$ (where the signs depend on the degree of the forms), one finds that

$$*\ker_{L^p}(d_k)=\ker_{L^p}(d^*_{n-k}),$$
and

$$*\im_{L^p}(d_{k-1})=\im_{L^p}(d^*_{n-k+1}).$$
Therefore, since the Hodge star is an isomorphism,

$$\overline{H}^k_p(M)\simeq \frac{\ker_{L^p}(d^*_{n-k})}{\im_{L^p}(d^*_{n-k+1})}.$$
Using $L^p-L^q$ duality, one has

$$\frac{\ker_{L^p}(d^*_{n-k})}{\im_{L^p}(d^*_{n-k+1})}=\frac{\mathrm{Ann}_{L^q}\left(\im_{L^p}(d^*_{n-k+1})\right)}{\mathrm{Ann}_{L^q}\left(\ker_{L^p}(d^*_{n-k})\right)},$$
and the result now follows from the two following facts:

$$\mathrm{Ann}_{L^q}\left(\im_{L^p}(d^*_{n-k+1})\right)=\ker_{L^q}(d_{n-k}),$$
and

$$ \mathrm{Ann}_{L^q}\left(\ker_{L^p}(d^*_{n-k})\right)=\mathrm{Ann}_{L^q}\left(\mathrm{Ann}_{L^p}\left(\im_{L^q}(d_{n-k-1})\right)\right)=\im_{L^q}(d_{n-k-1}),$$
where for the last equality we have used the fact that 

$$\mathrm{Ann}_{L^q}\left(\mathrm{Ann}_{L^p}(V)\right)=V,$$
for a closed subspace $V\subset L^q$.

\end{proof}
Let 
\begin{align*}
Z:=\im_{L^p}(d_{0})\cap \im_{L^p}(d^*_{2})\subset \ker_{L^p}(\Delta_{1})	
\end{align*}
Furthermore, we let
\begin{itemize}
	\item[(i)] $X_1$ be a complement of $Z$ in $\im_{L^p}(d_{0})\cap \ker_{L^p}(\Delta_{1})$,
	\item[(ii)] $X_2$ be a complement of $Z\oplus X_1$ in $\im_{L^p}(d_{0})$,
	\item[(iii)] $Y_1$ be a complement of $Z$ in $\im_{L^p}(d^*_{2})\cap \ker_{L^p}(\Delta_{1})$ and
    \item[(iv)] $Y_2$ be a complement of $Z\oplus Y_1$ in $\im_{L^p}(d^*_{2})$. 
\end{itemize}
Here, by {\em complement} of closed subspace $V\subset E$ where $E$ is a Banach space, we mean a {\em closed} vector subspace $W\subset E$ such that $V\oplus W=E$. Recall that finite-dimensional subspaces of Banach spaces can always be complemented. Therefore, since $\ker_{L^p}(\Delta_{1})$ is finite-dimensional, all the above complements do actually exist.

\begin{proposition}\label{lem : de Rham decomp large p}
	Let $p\in [n,\infty)$ and the vector spaces $X_i,Y_i,Z\subset L^p(\Lambda^1M)$, $i=1,2$ defined as above. Then we have
	\begin{align}
		\label{eq : decomp1} L^p(\Lambda^1M)&=X_2\oplus X_1\oplus Z\oplus Y_1\oplus Y_2\oplus \ker_{-n}(\Delta_{1}),\\
		\label{eq : decomp2}\im_{L^p}(d_{0})&=X_2\oplus X_1\oplus Z,\\
		\label{eq : decomp3}\im_{L^p}(d^*_{2})&=Z\oplus Y_1\oplus Y_2,\\
		\label{eq : decomp4}\ker_{L^p}(d_1)&=X_2\oplus X_1\oplus Z\oplus Y_1\oplus \ker_{-n}(\Delta_{1}),\\
		\label{eq : decomp5}\ker_{L^p}(d^*_1)&= X_1\oplus Z\oplus Y_1\oplus Y_2\oplus \ker_{-n}(\Delta_{1}),\\
		\label{eq : decomp6}\ker_{1-n}(\Delta_{1})&=X_1\oplus Z\oplus Y_1\oplus \ker_{-n}(\Delta_{1}).
	\end{align}
\end{proposition}
\begin{proof}
	The equations \eqref{eq : decomp2} and \eqref{eq : decomp3} follow directly from construction.
	Then, \eqref{eq : decomp1} follows from the strong Hodge decompositions in Corollary \ref{cor:strong_hodge}. Now we are going to show \eqref{eq : decomp4}. First, notice that if $q=p'$ is the conjugate exponent, then $\ker_{L^q}(\Delta_1)=\ker_{-n}(\Delta_1)$, so
	
$$\ker_{L^p}(\Delta_1)\cap \ker_{L^q}(\Delta_1)=\ker_{-n}(\Delta_1).$$
It then follows from Lemma \ref{lem : de Rham intersection} that

$$Z\cap \ker_{-n}(\Delta_1)=\{0\}.$$
Next, note that
	\begin{align*}
		X_2\subset \im_{L^p}(d_{0})&\subset \ker_{L^p}(d_1),\\
		X_1\oplus Z\oplus Y_1\oplus\ker_{-n}(\Delta_{1})&\subset
		\ker_{L^p}(\Delta_{1})\subset \ker_{L^p}(d_1)
	\end{align*}
	(according to (a) in Proposition \ref{prop:decay_harmonic_forms}). Because $Y_2\subset  \im_{L^p}(d_{2}^*)\subset \ker_{L^p}(d^*_1)$, we have 
	\[
	Y_2\cap \ker_{L^p}(d_1)=
	Y_2\cap \ker_{L^p}(d_1)\cap\ker_{L^p}(d^*_1)=
	Y_2\cap \ker_{L^p}(\Delta_{1})=\{0\},
	\]
	because $Y_2$ complements $Z\oplus Y_1=\im_{L^p}(d_{0})\cap \ker_{L^p}(\Delta_{1})$ in $\im_{L^p}(d_{0})$. Therefore, we get \eqref{eq : decomp4} from intersecting $\ker_{L^p}(d_1)$ with \eqref{eq : decomp1}. The proof of \eqref{eq : decomp5} is completely analogous. It remains to show \eqref{eq : decomp6}. At first, we have
	\[
	X_1\oplus Z\oplus Y_1\oplus \ker_{-n}(\Delta_{1})\subset \ker_{1-n}(\Delta_{1})
	\]
	by construction and since $\ker_{1-n}(\Delta_{1})=\ker_{L^p}(\Delta_1)$. To show equality, it suffices by \eqref{eq : decomp1} to show
	\[
	(X_2\oplus Y_2)\cap \ker_{1-n}(\Delta_{1})=\{0\}.
	\]
	From Proposition \ref{prop:decay_harmonic_forms}, we obtain
	\begin{align*}
	(X_2\oplus Y_2)\cap \ker_{1-n}(\Delta_{1})&=
	(X_2\oplus Y_2)\cap \ker_{L^p}(\Delta_{1})\\&=
	(X_2\oplus Y_2)\cap \ker_{L^p}(d_1)\cap\ker_{L^p}(d^*_1)
	\end{align*}
	and \eqref{eq : decomp1}, \eqref{eq : decomp4} and \eqref{eq : decomp5} show that the intersection on the right hand side is the zero space. This finishes the proof. 
\end{proof}
This proposition allows us to make the remaining cohomology spaces more explicit:

\begin{cor}\label{cor:cohomology_decomposition}
	Let $p\in [n,\infty)$ and $q\in (1,\frac{n}{n-1}]$ such that $1=\frac{1}{p}+\frac{1}{q}$. Let $X_i$ and $Y_i$, $i=1,2$ be defined by Proposition \ref{lem : de Rham decomp large p}, then we have identifications
	\begin{align*}
		\overline{H}^1_p(M)\cong Y_1\oplus \ker_{-n}(\Delta_{1})\cong \overline{H}^{n-1}_q(M),\\
		\overline{H}^1_q(M)\cong X_1\oplus  \ker_{-n}(\Delta_{1})\cong \overline{H}^{n-1}_p(M).
	\end{align*}
\end{cor}
\begin{proof}
We treat only the case $k=1$, the case $k=n-1$ following from Proposition \ref{prop:duality}. The description of $\overline{H}^1_p(M)$ follows from \eqref{eq : decomp2} and \eqref{eq : decomp4}. For $\overline{H}^1_q(M)$, recall that
	\[
	\ker_{L^q}(d_1)=\mathrm{Ann}_{L^q}(\im_{L^p}(d^*_2)),\qquad
	\im_{L^q}(d_0)=\mathrm{Ann}_{L^q}(\ker_{L^p}(d^*_1)),
	\]
	which allows us to identify
	\[
	\overline{H}^1_q(M)=\frac{\ker_{L^q}(d_1)}{\im_{L^q}(d_0)}
	=\frac{\mathrm{Ann}_{L^q}(\im_{L^p}(d^*_2))}{\mathrm{Ann}_{L^q}(\ker_{L^p}(d^*_1))}\cong \left(\frac{\ker_{L^p}(d^*_1)}{\im_{L^p}(d^*_2)}\right)^*,
	\]
	where ${}^{*}$ denotes the dual space. From \eqref{eq : decomp3} and \eqref{eq : decomp5},
	 we get
	\[
	\frac{\ker_{L^p}(d^*_1)}{\im_{L^p}(d^*_2)} \cong X_1\oplus  \ker_{-n}(\Delta_{1})
	\]
According to Proposition \ref{pro:Hodge}, $X_1\oplus \ker_{-n}(\Delta_1)\subset \ker_{1-n}(\Delta_1)$ is finite dimensional, hence isomorphic to its dual space. This finishes the proof.
\end{proof}
Let us now analyse in more detail the spaces $X_i,Y_i,Z$ which appeared in Proposition \ref{lem : de Rham decomp large p}.

\begin{Pro}\label{pro:Lp2}

Let $p\in [n,+\infty)$. Then, the following hold:

\begin{enumerate}

\item[(i)] $Z=\{0\}$.

\item[(ii)] $Y_1=\{0\}$.

\item [(iii)] $X_1$ is given by the following equalities:

$$X_1=\im_{L^p}(d_0)\cap \ker_{1-n}(\Delta_1)=d(\ker_0(\Delta_0)).$$

\end{enumerate}
\end{Pro}

\begin{proof}
The fact that $Z=\{0\}$ follows immediately from point (iii) in Proposition \ref{pro:Lp1}. Concerning (ii), by intersecting \eqref{eq : decomp3} and \eqref{eq : decomp6}, we get
\begin{align*}
\ker_{L^p}(\Delta_1)\cap \im_{L^p}(d_2^*)=	\ker_{1-n}(\Delta_1)\cap \im_{L^p}(d_2^*)=Z\oplus Y_1=Y_1
	\end{align*}
However, point (ii) in Proposition \ref{pro:Lp1} states that 

$$\ker_{L^p}(\Delta_1) \cap \im_{L^p}(d_2^*)=\{0\},$$
hence $Y_1=\{0\}$.

Let us now prove (iii). The first equality in (iii) follows from intersecting \eqref{eq : decomp2} and \eqref{eq : decomp6} and using that $Y_1=\{0\}$ and $Z=\{0\}$. Additionally, \eqref{eq : decomp6} and the conditions $Y_1=\{0\}$ and $Z=\{0\}$ directly yield
\begin{align*}
\ker_{1-n}(\Delta_1)=X_1\oplus  \ker_{-n}(\Delta_1).	
\end{align*}
On the other hand, we know from Corollary \ref{cor:decomp_kernel_1} that 
\begin{align*}
	\ker_{1-n}(\Delta_1)=d(\ker_0(\Delta_0))\oplus  \ker_{-n}(\Delta_1).	
\end{align*}
Therefore, it suffices to show the inclusion
\begin{align*}
d(\ker_0(\Delta_0))\subset X_1= \im_{L^p}(d_0)\cap \ker_{1-n}(\Delta_1),
\end{align*}
and this follows from Proposition \ref{pro:im_p>n}.

\end{proof}
We are now ready to state and prove the result concerning the remaining $L^p$ cohomology spaces:

\begin{Thm}\label{mainthm:part_2}

Let $p\in [n,+\infty)$ and let $q\in (1,\frac{n}{n-1}]$ be the conjugate exponent. Then,
\begin{enumerate}
\item[(i)] $$\overline{H}_p^1(M)\cong \ker_{-n}(\Delta_1)\cong \overline{H}_q^{n-1}(M),$$
\item[(ii)] $$\overline{H}_q^1(M)\cong \ker_{1-n}(\Delta_1)\cong \overline{H}_p^{n-1}(M),$$
\end{enumerate}

\end{Thm}
\begin{proof}As before, we only prove the case $k=1$. By Corollary \ref{cor:cohomology_decomposition}, we have
	\begin{align*}
		\overline{H}^1_p(M)\cong Y_1\oplus \ker_{-n}(\Delta_{1})
		\end{align*}
	but $Y_1=\{0\}$ by Proposition \ref{pro:Lp1}. This proves (i). Again by Corollary \ref{cor:cohomology_decomposition}, we have
	\begin{align*}
		\overline{H}^1_q(M)\cong X_1\oplus \ker_{-n}(\Delta_{1}).	
	\end{align*}
Combining this with Proposition \ref{pro:Lp2} and Corollary \ref{cor:decomp_kernel_1} directly yields
\begin{align*}
X_1\oplus \ker_{-n}(\Delta_{1})=d(\ker_0(\Delta_0))\oplus\ker_{-n}(\Delta_{1})=\ker_{1-n}(\Delta_{1}),
\end{align*}
which shows (ii).

\end{proof}
This completes the study of $L^p$ cohomology spaces. It remains to look at $L^q$ Hodge decomposition for $q\in (1,\frac{n}{n-1}]$. In this respect, we can show:

\begin{proposition}\label{pro:small_exp}
Let $k\in\{1,n-1\}$ and $q\in (1,\frac{n}{n-1}]$ . Then we have direct sums
	\begin{align*}
	\im_{L^q}(d_2^*)\oplus \im_{L^q}(d_{0})\oplus \ker_{L^q}(\Delta_1)\subset L^q(\Lambda^1M)
\end{align*}
and 
\begin{align*}
	\im_{L^q}(d_n^*)\oplus \im_{L^q}(d_{n-2})\oplus \ker_{L^q}(\Delta_{n-1})\subset L^q(\Lambda^{n-1}M)
\end{align*}
which are closed subspaces of codimension $N-1$, $N$ being the number of ends.
\end{proposition}

\begin{proof}
As usual, it is enough to prove the result for $k=1$. We have $\ker_{-n}(\Delta_1)=\ker_{L^q}(\Delta_1)\subset\ker_{L^p}(\Delta_1)=\ker_{1-n}(\Delta_1)$, where  $p\in [n,\infty)$ is the conjugate exponent. Therefore, the sum on the left hand side is direct by Lemma \ref{lem : de Rham intersection}. According to Theorem \ref{thm:closed sum}, and Lemma \ref{lem:sum_finite_dim}, $\ker_{L^q}(\Delta_1)$ being finite dimensional, it is closed.
 Its codimension equals the dimension of its annihilator in $L^p$, which we are going to compute for the remainder of this proof. We have
\begin{align*}
	\mathrm{Ann}_{L^p}&(\im_{L^q}(d_2^*)\oplus \im_{L^q}(d_0)\oplus \ker_{-n}(\Delta_1))\\&=
	\ker_{L^p}(d_1^*)\cap \ker_{L^p}(d_1)\cap \mathrm{Ann}_{L^p}(\ker_{-n}(\Delta_1))\\
	&=\ker_{L^p}(\Delta_1)\cap \mathrm{Ann}_{L^p}(\ker_{-n}(\Delta_1))\\
	&=\ker_{1-n}(\Delta_1)\cap \mathrm{Ann}_{L^p}(\ker_{-n}(\Delta_1)).
\end{align*}
(where we have used Proposition \ref{prop:decay_harmonic_forms} to pass from the second to the third line).
Because $L^p$ and $L^q$ are dual one to another by the $L^2$-pairing, the $L^2$-orthogonal decomposition in Corollary \ref{cor:decomp_kernel_1} directly implies
 \begin{align*}
 \ker_{1-n}(\Delta_1)\cap \mathrm{Ann}_{L^p}(\ker_{-n}(\Delta_1))=d(\ker_0(\Delta_0)).
 \end{align*}

%
By Corollary \ref{cor:dimension_d_harmonic functions}, the space on the right hand side is $N-1$-dimensional, which concludes the proof.
\end{proof}
Finally, we can conclude this section with the proof of two main results of the present article, Theorem \ref{thm:Hodge} and Theorem \ref{thm:main1}:

\begin{proof}[Proof of Theorem \ref{thm:Hodge}:]

The cases $2\leq k\leq n-2$, or $k\in\{1,n-1\}$ and $\frac{n}{n-1}<p<n$, or $k\in \{1,n-1\}$, $p\in (1,\infty)$ and $M$ has only one end, or $k\in \{1,n-1\}$, $p\geq n$ and $M$ has more than two ends follow from Corollary \ref{cor:strong_hodge}. The case $k\in\{1,n-1\}$, $p\in (1,\frac{n}{n-1})$ and $M$ has more than two ends follows from Proposition \ref{pro:small_exp}. It remains to see the cases $k=0$ and $k=n$. One can obtain the $k=n$ from the $k=0$ case, by Hodge duality, so let us prove the latter. For $k=0$, one wishes to show that
\begin{align}\label{eq: Hodge decomp 0-form}
L^p(M)=\im_{L^p}(d_1^*),
\end{align}
since according to \cite[Theorem 3]{Y}, $\ker_{L^p}(\Delta_0)=\{0\}$ on any complete manifold with infinite volume. 
Let $q$ be the conjugate exponent to $p$. Then we have
$$\mathrm{Ann}_{L^q}( \im_{L^p}(d_1^*))=\ker_{L^q}(d_0).$$
However if $f\in \ker_{L^q}(d_0)$, then first $f$ is constant (since $M$ is connected), and because $f\in L^q(M)$, $f$ must vanish identically. Thus,
$$\mathrm{Ann}_{L^q}( \im_{L^p}(d_1^*))=\{0\},$$
and \eqref{eq: Hodge decomp 0-form} follows by duality.
This concludes the proof.
\end{proof}

\begin{proof}[ Proof of Theorem \ref{thm:main1}:]

According to Theorem \ref{mainthm:part_1}, if either $k\notin\{1,n-1\}$ or $p\in \left(\frac{n}{n-1},n\right)$, then $\overline{H}_p^k(M)\simeq \mathcal{H}_k(M)\simeq \overline{H}_2^k(M)$. Let us now assume that $k=1$; then, if $p\leq \frac{n}{n-1}$, part (ii) of Theorem \ref{mainthm:part_2} yields

$$\overline{H}_p^1(M)\simeq \ker_{1-n}(\Delta)=\ker_{L^2}(\Delta_1)=\mathcal{H}_1(M).$$
Now, if $p\geq n$, then according to Theorem \ref{mainthm:part_2},

$$\overline{H}_p^1(M)\simeq \ker_{-n}(\Delta_1),$$
which has codimension $N-1$ in $\ker_{1-n}(\Delta_1)$ by Corollaries \ref{cor:dimension_d_harmonic functions} and \ref{cor:decomp_kernel_1}. According to Theorem \ref{mainthm:part_2}, if $\frac{1}{p}+\frac{1}{q}=1$, then

$$\overline{H}_p^1(M)\cong \overline{H}_q^{n-1}(M), $$
and this easily implies the result for $\overline{H}^{n-1}_q(M)$.

\end{proof}

\section{Boundedness of Riesz transforms on forms}

In this section, we explain how the Hodge projectors can also be expressed by means of Riesz transforms on forms (see also \cite[Section 2]{AMR}), and how this fact allows us to recover some known results concerning the Riesz transforms on manifolds with Euclidean ends. It is convenient for explaining this to consider the vector bundle of forms of any degree:

$$\Lambda^* M=\oplus_{k=0}^n\Lambda^k M$$
Recall that we denote $\Delta=dd^*+d^*d$ the Hodge Laplacian on $\Lambda^*M$. We will also denote $\Pi_>=I-\Pi_0$, where $\Pi_0$ is the orthogonal projector onto $\ker_{L^2}(\Delta)$; and $\Pi_d=\oplus \Pi_{k,d}$, $\Pi_{d^*}=\oplus\Pi_{k,d^*}$ the Hodge projectors onto closed and co-closed forms of any degree respectively. It is well-known (and due to Gaffney) that if $M$ is complete, then the Hodge-Dirac operator $\mathcal{D}=d+d^*$, defined as a self-adjoint unbounded operator on $L^2(\Lambda T^*M)$ whose domain is given as the domain of the closed quadratic form

$$q(\omega,\omega)=\int_M|d\omega|^2+|d^*\omega|^2,$$
has $C_c^\infty (\Lambda^*M)$ as a core. In fact, if one assumes that $\omega$ is smooth, then $\mathcal{D}\omega$ is easily shown to be the limit in $L^2$ as $n\to\infty$ of $\mathcal{D}(\chi_n\omega)$, where $0\leq \chi_n\leq 1$, $\chi_n$ is $1$ in restriction to $B(o,n)$, $0$ in restriction to $M\setminus B(o,2n)$, and $||\nabla \chi_n||_\infty \lesssim 1$. The same result holds for general $\omega$ in the domain of $\mathcal{D}$ by first approximating $\omega$ by smooth forms on compact sets, using standard partition of unity and mollifiers arguments in local charts. As a consequence, the {\em range} of $(d+d^*)$, namely

$$\mathrm{Rg}(d+d^*):=\{(d+d^*)\omega\,;\,\omega\in \mathcal{D}om(d+d^*)\}$$
is closed for the graph norm (however it is {\em not} a closed subspace of $L^2$), and is furthermore equal to the closure of $(d+d^*)C_c^\infty( \Lambda^*M)$ for the graph norm. Moreover, since $\left(\mathrm{Rg}(d+d^*)\right)^\perp=\ker_{L^2}(d+d^*)$, one has that

$$L^2(\Lambda^*M)=\overline{\mathrm{Rg}(d+d^*)}\oplus_\perp \ker_{L^2}(d+d^*).$$
Define $\mathcal{D}om(d)=\{\omega\in L^2(\Lambda^*M)\,;d\omega\in L^2\}$, and analogously $\mathcal{D}om(d^*)$. Similar arguments to the ones presented above for $\mathcal{D}$ show that $C_c^\infty(\Lambda^*M)$ is a core, both for $d$ and for $d^*$, and this yields that

\begin{equation}\label{eq:Rg1}
\im_{L^2}(d_{k-1})=\overline{dC_c^\infty(\Lambda^{k-1}M)}^{L^2}=\overline{\mathrm{Rg}(d_{k-1})},
\end{equation}
and 

\begin{equation}\label{eq:Rg2}
\im_{L^2}(d^*_{k+1})=\overline{d^*C_c^\infty(\Lambda^{k+1} M)}^{L^2}=\overline{\mathrm{Rg}(d^*_{k+1})},
\end{equation}
where the first equalities are by definition, and the range $\mathrm{Rg}(A)$ of an unbounded operator $A$ on $L^2$ is defined in general as $\{A\omega\,;\,\omega\in \mathcal{D}om(A)\}$. The fact that $d^2=0$ and \eqref{eq:Rg1}, \eqref{eq:Rg2} yield that the vector spaces $\overline{\mathrm{Rg}(d_{k-1})}$ and $\overline{\mathrm{Rg}(d^*_{k+1})}$ are orthogonal one to another. Note moreover that if $\omega\in \mathrm{Rg}(d+d^*)$, then one can write 

$$\omega=(d+d^*)\alpha,$$
with $\alpha\in \mathcal{D}om(d+d^*)$ such that $\alpha\in (\ker_{L^2}(d+d^*))^\perp=\mathrm{Rg}(d+d^*)$. Define

$$(d+d^*)^{-1}\omega:=\alpha$$
The definition of the quadratic form $q$ implies immediately that 

$$\mathcal{D}om(d+d^*)\subset \mathcal{D}om(d)\cap \mathcal{D}om(d^*),$$
so 

$$\alpha\in \mathcal{D}om(d)\cap \mathcal{D}om(d^*),$$
and as a consequence,

$$d(d+d^*)^{-1}\omega=d\alpha\in \mathrm{Rg}(d),\quad d^*(d+d^*)^{-1}\omega=d^*\alpha\in \mathrm{Rg}(d^*).$$
Since $\omega=(d+d^*)(d+d^*)^{-1}\omega$ for $\omega\in \mathrm{Rg}(d+d^*)=\left(\ker_{L^2}(d+d^*)\right)^\perp$, one gets that

$$\mathrm{Rg}(d+d^*)\subset \mathrm{Rg}(d)\oplus_\perp\mathrm{Rg}(d^*).$$

Using this, it is not hard to see that the $L^2$ Hodge projectors $\Pi_d$, $\Pi_{d^*}$ on exact and coexact forms introduced previously are given by the following formulae:

$$\Pi_d\omega=d(d+d^*)^{-1}\omega,\quad \Pi_{d^*}\omega=d^*(d+d¨^*)^{-1}\omega,\quad \omega\in\mathrm{Rg}(d+d^*).$$
We want to obtain a similar formula, where $(d+d^*)^{-1}$ has been replaced by a function of the Laplacian. Noticing that $\Delta=(d+d^*)^2$ (at least on $C_c^\infty(\Lambda^*M)$), one can define a self-adjoint extension of $\Delta|_{C_c^\infty(\Lambda^*M)}$ as $(d+d^*)^2$ (this is the operator associated with the multiplier $|\lambda|^2$ in a spectral resolution of $(d+d^*)$). However, it is well-known that $\Delta$ is essentially self-adjoint on $C_c^\infty(\Lambda^*M)$ on any complete Riemannian manifold (see a proof for scalar Schr\"odinger operators in \cite[Theorem 3.13]{PRS}; this proof adapts without difficulty to the case of the Hodge Laplacian, upon noticing that thanks to the Bochner-Weitzenb\"ock formula, one can write $\Delta_k=\nabla^*\nabla+\mathscr{R}_k$ with a certain symmetric curvature tensor $\mathscr{R}_k$). Hence, the above self-adjoint operator $(d+d^*)^2$ is the self-adjoint Hodge Laplacian $\Delta$ that we have considered throughout this article. From now on we thus will use freely the formula $\Delta=(d+d^*)^2$ as self-adjoint operators. One can then define $(\Delta\Pi_>)^{-1/2}$ as $\varphi(d+d^*)$, for the spectral multiplier $\varphi(x)=|x|^{-1}\mathbf{1}_{\R^*}(x).$ We consider the {\em Riesz transform on forms}:

$$R=(d+d^*)(\Delta\Pi_>)^{-1/2},$$
as well as its exact and co-exact parts:

$$R_{d}=\Pi_ dR=d(\Delta\Pi_>)^{-1/2},\quad R_{d^*}=\Pi_{d^*}R=d^*(\Delta\Pi_>)^{-1/2}.$$
 Then, $R$ is $L^2$ bounded (it is the operator associated with the bounded spectral multiplier $\mathrm{sgn}(x)\mathbf{1}_{\R^*}(x)$ in a spectral resolution of $(d+d^*)$), and thus $R_{d}$, $R_{d^*}$ are $L^2$ bounded as well. The spectral theorem entails that $R^2=\Pi_>$, and since $R^2=(R_d+R_{d^*})^2$, developping the square and considering the degrees yield:
 
 $$(R_d)^2=0,\quad (R_{d^*})^2=0,$$
 and
 
 $$\Pi_>=R_dR_{d^*}+R_{d^*}R_{d}.$$
From the fact that $R^*=R$, it is also easily seen that $(R_d)^*=R_{d^*}$. Moreover, if $\eta\in L^2(\Lambda^*M)$ and $\varphi\in L^2(\Lambda^*M)$, self-adjointess of $R$ and the fact that $(d+d^*)$ and $(\Delta\Pi_>)^{-1/2}$ commute (by the functional calculus) yields that

$$((d+d^*)(\Delta\Pi_>)^{-1/2}\eta,\varphi)=(\eta,(\Delta\Pi_>)^{-1/2}(d+d^*)\varphi).$$
Taking $\eta$ of degree $k$ and $\varphi$ of degree $k+1$ then gives

$$(d(\Delta\Pi_>)^{-1/2}\eta,\varphi)=(\eta,(\Delta\Pi_>)^{-1/2}d^*\varphi),$$
hence $(R_d)^*=(\Delta\Pi_>)^{-1/2}d^*$; similarly, $(R_{d^*})^*=(\Delta\Pi_>)^{-1/2}d$.

If $\omega=d\eta+d^*\varphi\in L^2$ and $\eta,\varphi\in L^2$, then

$$R_d\omega=d(\Delta\Pi_>)^{-1/2}(d\eta+d^*\varphi)=d\left[(R_{d^*})^*\eta+(R_d)^*\varphi\right]\in \im_{L^2}(d),$$
hence $R_d$ sends $\im_{L^2}(d^*)\oplus \im_{L^2}(d)$ into $\im_{L^2}(d)$, and therefore $\im_{L^2}(R_d)\subset \im_{L^2}(d)$. Similarly, $\im_{L^2}R_{d^*}\subset \im_{L^2}(d^*)$.
As a consequence, for $\omega\in L^2(\Lambda^*M)$,

$$\Pi_>\omega=R_dR_{d^*}\omega+R_{d^*}R_{d}\omega$$
is necessarily the Hodge decomposition of $\Pi_>\omega$. It follows that 

$$\Pi_d=R_dR_{d^*},\quad \Pi_{d^*}=R_{d^*}R_{d}.$$
Hence, the boundedness on $L^p$ of the Hodge projectors is closely related to the boundedness on $L^p$ of the Riesz transforms on forms. In this direction, let us also mention that it is proved in \cite[Theorem 2.1]{AC} that for $k=1$, $R_{d^*}$ and $\Pi_d$ bounded on $L^p$ imply that $R_d$ is bounded on $L^p$; since $R_{d^*}$ is bounded on $L^p(\Lambda^1 M)$ for any $p\in [2,+\infty)$ on an ALE manifold (as a consequence of \cite[Theorem 1.1]{CD} and the fact that by duality, $R_{d^*}$ is the adjoint of the scalar Riesz transform), on such a manifold the boundedness on $L^p(M)$ of $R_d$ is equivalent to the boundedness on $L^p(\Lambda^1 M)$ of $\Pi_d$ for any $p\in [2,+\infty)$. Thus we get as a corollary:

\begin{Cor}\label{cor:Riesz-functions}

Let $M$ be an ALE manifold with dimension $n\geq 3$. Then, the Riesz transform on functions $R_d$ is bounded from $L^p(M)\to L^p(\Lambda^1 M)$ for every $p\in (1,p^*)$, where $p^*=+\infty$ if $M$ has only one end, $p^*=n$ otherwise.

\end{Cor}
Of course, this result is well-known, at least if $M$ is AE (asymptotically Euclidean): if $M$ has Euclidean ends, this is the famous result of \cite{CCH}; and the result for a merely AE manifold can be deduced from this by using the perturbation result \cite{CDun}. In the case where $M$ has Euclidean ends, there are several proofs of this result: apart from the original proof in \cite{CCH}, it is also a consequence of the combination of \cite{C5} and \cite{D2}; and in \cite{C6} there is yet another proof (which also applies to manifolds that are locally Euclidean outside a compact set). However, our proof in the present paper is arguably the shortest, and most elementary one of all these proofs.

We also mention that the $L^p$ boundedness of $\Pi_>$ is easily characterized:

\begin{Pro}\label{pro:Riesz-Hodge}

Let $k\in\{0,\cdots,n\}$, $p\in (1,\infty)$, and let $q=p'$ be the conjugate exponent. Assume that $M$ satisfies assumption (H). Then, the following are equivalent:

\begin{itemize}

\item[(a)] $\Pi_>$, defined on $L^2(\Lambda^k M)\cap L^p(\Lambda^k M)$, extends uniquely to a bounded operator on $L^p(\Lambda^k M)$.

\item[(b)] $\ker_{L^2}(\Delta_k)=\ker_{L^{\min(q,p)}}(\Delta_k)$.

\end{itemize}

\end{Pro}

\begin{proof}

Since $\Pi_>$ is self-adjoint, it is bounded on $L^p$ if and only if it is bounded on $L^q$, so without loss of generality we can assume that $q\leq 2\leq p$. Take $(\omega_i)_{i=1,\cdots,N}$ an orthonormal basis of $\ker_{L^2}(\Delta_k)$. Recall from Lemma \ref{lem:ker} that since $M$ satisfies assumption (H), then $\ker_{L^q}(\Delta_k)\subset \ker_{L^2}(\Delta_k)\subset \ker_{L^p}(\Delta_k)$. Next, $\Pi_0:=I-\Pi_>$ writes:

$$\Pi_0=\sum_{i=1}^N(\omega_i,\cdot)\omega_i.$$
From this formula, it is clear that (b)$\Rightarrow$(a). Let us show the converse. Since $C_c^\infty(\Lambda^k M)$ is dense in $L^2(\Lambda^k M)$ and $\ker_{L^2}(\Delta_k)$ is finite dimensional, it follows that $\Pi_0(C_c^\infty(\Lambda^k M))=\ker_{L^2}(\Delta_k)$. Thus, for any $i\in\{1,\cdots,N\}$, there exists $\varphi_i\in C_c^\infty(\Lambda^k M)\subset L^2\cap L^q(\Lambda^k M)$ such that

$$\Pi_0(\varphi_i)=\omega_i.$$
But since $\Pi_>$ is bounded on $L^q$, $\Pi_0(L^2\cap L^q(\Lambda^k M))\subset L^q(\Lambda^k M)$, hence $\omega_i\in L^q$, therefore $\ker_{L^2}(\Delta_k)\subset L^q$. Thus, we conclude that $\ker_{L^q}(\Delta_k)=\ker_{L^2}(\Delta_k)$.

\end{proof}
The boundedness on $L^p$ Riesz transform on forms on manifolds that are conical at infinity (and, in particular, on ALE manifolds) has also been studied in \cite{GS}, using the tools of pseudo-differential calculus on manifolds with corners. More precisely, in \cite[Corollary 9]{GS}, the set of $p\in (1,+\infty)$ such that $R$ is bounded on $L^p$ is characterized in terms of the decay at infinity of $L^2$ harmonic forms. However, in \cite{GS} the authors do not characterize the decay rates, in terms of the degree of the form and/or the topology at infinity of the manifold.

In contrast, Proposition \ref{prop:decay_harmonic_forms} and Corollary \ref{cor:decomp_kernel_1} yield the complete characterization of the set of $p$'s such that $R$ is bounded on $L^p$:

\begin{Cor}\label{cor:Riesz}

Let $M$ be an ALE manifold with dimension $n\geq 3$, and let $k\in\{0,\cdots,n\}$. Then the following holds:

\begin{itemize}

\item[(a)] If $k\notin \{0,1,n-1,n\}$ then $R$ is bounded on $L^p(\Lambda^k M)$ for every $p\in (1,+\infty)$.

\item[(b)] If $k=0$ or $k=n$, then $R$ is bounded on $L^p(\Lambda^k M)$ if and only if $p\in (1,p^*)$ where $p^*=+\infty$ if $M$ has only one end, $p^*=n$ otherwise.

\item[(c)] if $k=1$ or $k=n-1$, then $R$ is bounded on $L^p(\Lambda^k M)$ if and only if $p\in (p_*,+\infty)$ where $p_*=1$ if $M$ has only one end, $p_*=\frac{n}{n-1}$ otherwise.

\end{itemize}

\end{Cor}
Remembering that $R^2=\Pi_>$, Corollary \ref{cor:Riesz} can also be used to recover the fact, which is also a consequence of our results (Theorem \ref{thm:main1} and Corollary \ref{cor:extension_hodgeLp}), that $\Pi_>$ is bounded on $L^p(\Lambda^k M)$ if either $k\notin \{1,n-1\}$, or $k\in \{1,n-1\}$ and $M$ has only end, or $p\in \left(\frac{n}{n-1},n\right)$. The boundedness on $L^p$ of the Hodge projectors could be used to establish an $L^p$ Hodge decomposition, and thus this provides an alternative approach to some of our results; however, this is far less direct than the approach we took in the present paper, and moreover one cannot recover in this way the results of Theorem \ref{mainthm:part_2}.

\section{$L^p$ Hodge-Sobolev decompositions}\label{Section:Hodge-Sob}

In this section, we consider in some sense the strongest possible $L^p$ Hodge decompositions, that we call $L^p$ {\em Hodge-Sobolev} decompositions. For that, we denote $\dot{W}^{1,p}(\Lambda^kM)$ the homogeneous Sobolev space, that is the closure of $C_c^\infty(\Lambda^kM)$ under the norm $||\nabla \omega||_{L^p}$. Here, $\nabla$ is the natural Levi-Civita connection on tensors.

\begin{Def}
{\em 
We say that the $L^p$ {\em Hodge-Sobolev decomposition} holds for forms of degree $k$, if every $\omega\in L^p(\Lambda^kM)$ writes in a unique way

$$\omega=d\alpha+d^*\beta+\eta,$$
with $\alpha\in \dot{W}^{1,p}(\Lambda^{k-1}M)$, $\beta\in \dot{W}^{1,p}(\Lambda^{k-1}M)$ and $\eta\in\ker_{L^p}(\Delta_k)$, and moreover the following estimates hold:

$$||\alpha||_{\dot{W}^{1,p}}\lesssim ||\omega||_p,\quad ||\beta||_{\dot{W}^{1,p}}\lesssim ||\omega||_p.$$
}
\end{Def}
Analogously, we define a {\em modified $L^p$ Hodge-Sobolev decomposition}, by replacing in the above definition the condition $\eta\in \ker_{L^p}(\Delta_k)$ by $\eta\in\ker_{-n}(\Delta_k)$. T. Iwaniec and G. Martin have proved in \cite{IM} that for the Euclidean space itself, the $L^p$ Hodge-Sobolev decomposition holds in all degree and for every $p\in (1,\infty)$. We prove:

\begin{Thm}\label{thm:Hodge-Sob}

Let $M$ be a connected, oriented ALE manifold of dimension $n\geq 3$. Let $p\in (1,+\infty)$, $p\neq n$ and $k\in \{0,\cdots,n\}$. Then, in cases (a), (b), (c) of Theorem \ref{thm:Hodge}, the $L^p$ Hodge-Sobolev decomposition for forms of degree $k$ holds. And in case (d) of Theorem \ref{thm:Hodge}, the modified $L^p$ Hodge-Sobolev decomposition for forms of degree $k$ holds.

\end{Thm}

\begin{proof}

The proof follows from Corollary \ref{Cor:Hodge-Sob} and the $L^p$ (modified) Hodge decomposition of Theorem \ref{thm:Hodge}.

\end{proof}

\begin{Rem}
{\em 
It is possible that the interpolation arguments used in Section 4 can be adapted to show that the result of Theorem \ref{thm:Hodge-Sob} also holds for $p=n$. One difficulty is that the homogeneous Sobolev spaces are not known to interpolate by the real method on manifolds that do not support Poincar\'e inequalities (e.g. ALE manifolds with $N\geq 2$ ends); see \cite{Badr}. We leave this question for future work.
}
\end{Rem}

%
%
%
%
%
%
%
%
%

\appendix

\section{A uniqueness lemma for Hodge projectors}

Let $p\in (1,\infty)$, $q=p'$ its conjugate exponent, and let $\mathscr{E}=L^2(\Lambda T^*M)\cap L^p(\Lambda T^*M)$, $\mathscr{F}=L^2(\Lambda T^*M)+ L^q(\Lambda T^*M)$, endowed with the norms

$$||\omega||_\mathscr{E}:=\max(||\omega||_2,||\omega||_p),$$
and

$$||\omega||_\mathscr{F}:=\inf\{||\varphi||_2+||\eta||_q\,;\,\omega=\varphi+\eta\}.$$
Then, (see \cite[Section 3, Exercise 6]{BCS}), $\mathscr{E}$ and $\mathscr{F}$ are Banach spaces, and $\mathscr{F}$ is the dual of $\mathscr{E}$. We will denote $\mathscr{E}_k$ and $\mathscr{F}_k$ the subsets of $\mathscr{E}$ and $\mathscr{F}$ consisting of forms of degree $k$. We let $\mathscr{G}_k=dC_c^\infty(\Lambda^{k-1}T^*M)\oplus d^*C_c^\infty(\Lambda^{k+1}T^*M)\oplus \ker_{L^2}(\Delta_k)$. If $\ker_{L^2}(\Delta_k)\subset \ker_{L^p}(\Delta_k)$ then it follows that $\mathscr{G}_k\subset \mathscr{E}_k$. We have the following lemma:

\begin{Lem}\label{lem:dense}

Let $M$ be a complete manifold satisfying assumption ($H_k$), and assume that $\ker_{L^q}(\Delta_k)\subset \ker_{L^2}(\Delta_k)$. Then, $\mathscr{G}_k$ is a dense subspace of $\mathscr{E}_k$.

\end{Lem}

\begin{proof}

Let $f\in \mathscr{E}_k^*$ such that $f|_{\mathscr{G}_k}\equiv0$. It is enough to show that $f\equiv0$. Since $\mathscr{E}_k^*=\mathscr{F}_k$, there exists $\omega\in \mathscr{F}_k$ such that

$$f(\eta)=(\eta,\omega),\quad \forall \eta\in \mathscr{E}_k.$$
Hence, $d\omega=d^*\omega=0$ in the distribution sense. Thus, $\omega\in \ker_{L^2+L^q}(\Delta_k)$. But, since the curvature term $\mathscr{R}_k$ in the Bochner formula for $k$-forms is bounded, there exists a constant $C>0$ such that one as the domination:

$$|e^{-t\Delta_k}\eta|\leq e^{Ct}e^{-t\Delta_0}|\eta|.$$
As a consequence of assumption ($H_k$), the scalar heat semi-group is ultracontractive, therefore, for all $s\in [1,\infty]$ and all $t>0$,

$$e^{-t\Delta_0} : L^s\to L^s\cap L^\infty.$$
Hence, $e^{-t\Delta_k}: L^s(\Lambda T^*M)\to L^s(\Lambda T^*M)\cap L^\infty(\Lambda T^*M)$ on $M$. In particular,

$$e^{-\Delta_k}: L^2(\Lambda T^*M)+L^q(\Lambda T^*M)\to L^{\max(2,q)}(\Lambda T^*M).$$
But since $\omega$ is harmonic, we have $e^{-t\Delta_k}\omega=\omega$, for all $t\geq0$: indeed, for all $\varphi\in \mathscr{E}_k=\mathscr{F}_k^*$, one has 

$$\frac{d}{dt}(e^{-t\Delta_k}\omega,\varphi)=(e^{-t\Delta}\Delta\omega,\varphi)=0.$$
Thus, we conclude that 

$$\ker_{L^2+L^q}(\Delta_k)\subset \ker_{L^{\max(2,q)}}(\Delta_k).$$
But by assumption,

$$\ker_{L^{\max(2,q)}}(\Delta_k)\subset \ker_{L^2}(\Delta_k).$$
Thus, we conclude that $\omega\in \ker_{L^2}(\Delta_k)$. However, $\ker_{L^2}(\Delta_k)\subset \mathscr{G}_k$, and $f$ is assumed to be orthogonal to $\mathscr{G}_k$, therefore we finally conclude that $\omega=0$, and so $f\equiv 0$. This achieves the proof.

\end{proof}

\begin{Cor}\label{cor:extension}

Let $M$ be a complete manifold satisfying assumption (H), and let $k\in\{0,\cdots, n\}$ and $p\in (1,\infty)$, $q=p'$ be such that $\ker_{L^q}(\Delta_k)=\ker_{L^p}(\Delta_k)=\ker_{L^2}(\Delta_k)$. Let $\mathscr{L}$ be a bounded operator on $L^2(\Lambda^k M)$. Assume that $\mathscr{L}|_{\mathscr{G}_k}$, the restriction to $\mathscr{G}_k$ of $\mathscr{L}$, extends uniquely to a bounded operator $\tilde{\mathscr{L}}$ on $L^p(\Lambda^k M)$. Then, the restrictions to $L^2(\Lambda^k M)\cap L^p(\Lambda^k M)$ of $\mathscr{L}$ and $\tilde{\mathscr{L}}$ coincide.

\end{Cor}
This applies in particular to the Hodge projectors, under the assumptions of Proposition \ref{pro:Hodge}.

\begin{proof}

Let $\omega\in L^2(\Lambda^kM)\cap L^p(\Lambda^kM)$. According to Lemma \ref{lem:dense}, there is a sequence $\{\omega_n\}_{n\in\mathbb{N}}$ of elements of $\mathscr{G}_k$ which converges to $\omega$ in the $||\cdot||_{\mathscr{E}}$ norm; in particular, by definition of $||\cdot||_{\mathscr{E}}$, $\omega_n\to\omega$ both in $L^2$ and in $L^p$ norm. Since $\mathscr{L}|_{\mathscr{G}_k}=\tilde{\mathscr{L}}|_{\mathscr{G}_k}$, we have for every $n\in\N$ that

$$\mathscr{L}(\omega_n)=\tilde{\mathscr{L}}(\omega_n).$$
Using the convergence of $\{\omega_n\}_{n\in\mathbb{N}}$ to $\omega$ in $L^2$ (resp. in $L^p$) and the continuity of $\mathscr{L}$ in $L^2$ (resp. of $\tilde{\mathscr{L}}$ in $L^p$), we get by passing to the limit $n\to\infty$ in the above equality:

$$\mathscr{L}(\omega)=\tilde{\mathscr{L}}(\omega).$$
This concludes the proof.

\end{proof}

\section{Weighted Sobolev spaces and decay of harmonic forms}

In this appendix, we briefly present the theory of weighted Sobolev spaces on ALE manifolds, and use them to prove a few results that are used in the present paper. In particular, we will give a proof of Propositions \ref{prop:decay_harmonic_forms} and Proposition \ref{prop:closed_subspace}. For more details on weigthed Sobolev spaces, we refer to the presentation in the paper \cite{Bar86}. In what follows, we denote $\R^n_*=\R^n\setminus \{0\}$ and $r=|x|$, $\sigma=(1+|x|^2)^{1/2}$.

\begin{definition}

Let $\delta\in \R$ and $p\in (1,+\infty)$, then the weighted Lebesgue space $L^p_\delta$ (resp. $L'^p_\delta$) is defined as the set of functions $u$ in $L^p_{loc}(\R^n)$ (resp. $L^p_{loc}(\R^n_*)$), whose norms $||u||_{p,\delta}$ (resp. $||u||_{p,\delta}'$) are finite. Here,

$$||u||_{p,\delta}=\left(\int_{\R^n} |u|^p\sigma^{-\delta p-n}\,dx\right)^{1/p},$$
and

$$||u||'_{p,\delta}=\left(\int_{\R^n} |u|^p r^{-\delta p-n}\,dx\right)^{1/p}.$$
Then, for $k\in \N$, $k\geq 1$, the Sobolev spaces $W_\delta^{k,p}(\R^n)$ and $W'^{k,p}_\delta(\R^n_*)$ are defined as using the norms

$$||u||_{k,p,\delta}=\sum_{j=0}^k\sum_{|\alpha|=j}||\partial^\alpha u||_{p,\delta-j},$$
and

$$||u||'_{k,p,\delta}=\sum_{j=0}^k\sum_{|\alpha|=j}||\partial^\alpha u||'_{p,\delta-j}.$$

\end{definition}
If $M$ is ALE, then using coordinate systems at infinity and a smooth weight function which agrees with $\sigma$ in the coordinate neighbourhoods, one can define weighted Sobolev spaces $W^{k,p}_\delta$ on $M$. Because we require any asymptotic chart $\phi$ to fulfill the condition $\phi^*g_{eucl}-g\in \mathcal{O}_k(r^{-\tau-k})$ for all $k\in \N$,
 the definition of the weighted Sobolev spaces is independent of the chosen coordinates and the chosen weight function. One can also define weighted Sobolev spaces $W^{k,p}_\delta(\Lambda^*M)$ since the vector bundle $\Lambda^*M$ can be trivialized at infinity and one can thus look at the regularity componentwise. These weigthed Sobolev spaces satisfy properties that are analogous to classical properties of the usual Sobolev spaces on $\R^n$, for instance Sobolev embeddings, Rellich compactness theorem, etc. See \cite{Bar86} for more details. For the moment, we limit ourselves to point out the following regularity result, which is a consequence of the Sobolev embeddings (see \cite[Equation (1.10)]{Bar86}):

\begin{proposition}\label{prop:Sob}

Let $M$ be ALE, $\delta\in \R$ and $p\in (1,\infty)$. Let $\omega \in W^{k,p}_\delta(\Lambda^*M)$ for all $k\in \N$, then $\omega=o_\infty(r^\delta)$.

\end{proposition}
Recall from \cite[Def. 1.5]{Bar86} and \cite[Def. 1.5]{KP} the notion of operator {\em asymptotic to the Euclidean Laplacian or to the Euclidean Dirac operator}. We won't write down explictly the definition (it is quite intuitive) but the main examples are as follows: if $M$ is ALE then the Hodge-De Rham Laplacian is asymptotic to the Euclidean Hodge Laplacian $\Delta_{\R^n}$, while $\mathcal{D}:=d+d^*$ acting on $\Lambda^*M$ is asymptotic to the Euclidean Dirac operator $\mathcal{D}_{\R^n}=d_{\R^n}+d_{\R^n}^*$. If $M$ is ALE of order $\tau$, then this implies that for any $k\geq 2$, $\delta\in \R$ and $p\in (1,\infty)$,

\begin{equation}\label{eq:extra_decay_L1}
\Delta-\Delta_{\R^n}:W^{k,p}_\delta(\Lambda^*M)\rightarrow W^{k-2}_{\delta-2-\tau}(\Lambda^*M),
\end{equation}
and for $k\geq 1$,

\begin{equation}\label{eq:extra_decay_D1}
\mathcal{D}-\mathcal{D}_{\R^n}:W^{k,p}_\delta(\Lambda^*M)\rightarrow W^{k-1}_{\delta-1-\tau}(\Lambda^*M).
\end{equation}
Also, similarly,

\begin{equation}\label{eq:extra_decay_L2}
\Delta-\Delta_{\R^n}: \mathcal{O}_\infty(r^\delta)\to \mathcal{O}_\infty(r^{\delta-2-\tau}),
\end{equation}
and

\begin{equation}\label{eq:extra_decay_D2}
\mathcal{D}-\mathcal{D}_{\R^n}:\mathcal{O}_\infty(r^\delta)\to \mathcal{O}_\infty(r^{\delta-1-\tau}).
\end{equation}
In fact, all this relies on

\begin{equation}\label{eq:extra_decay_d*1}
d^*-d^*_{\R^n}:W^{k,p}_\delta(\Lambda^*M)\rightarrow W^{k-1}_{\delta-1-\tau}(\Lambda^*M).
\end{equation}
and
\begin{equation}\label{eq:extra_decay_d*2}
d^*-d^*_{\R^n}:\mathcal{O}_\infty(r^\delta)\to \mathcal{O}_\infty(r^{\delta-1-\tau}).
\end{equation}
We will use the following elliptic regularity result (see \cite[Prop. 2.10]{KP}):

\begin{proposition}\label{prop:ell-reg}

Suppose that $M$ is an ALE manifold, $p\in (1,\infty)$ and $k\geq 2$. Then, for every $\omega\in C^\infty(\Lambda^*M)$,

\begin{equation}\label{eq:ell-reg-lapl}
||\omega ||_{k,p,\delta} \leq C(||\Delta \omega||_{k-2,p,\delta-2}+||u||_{k-2,p,\delta}),
\end{equation}
and

\begin{equation}\label{eq:ell-reg-dir}
||\omega ||_{k,p,\delta} \leq C(||\mathcal{D}\omega||_{k-1,p,\delta-1}+||u||_{k-1,p,\delta}).
\end{equation}

\end{proposition}
These estimates together with the Sobolev embeddings (\ref{prop:Sob}) imply easily that:
\begin{corollary}\label{cor:ell_reg}

Assume that $M$ is ALE, and $\omega\in L^p_\delta$, $p\in (1,\infty)$, $\delta\in\R$, such that $\Delta\omega=0$ or $\mathcal{D}\omega=0$. Then,

$$\omega=o_\infty(r^\delta).$$

\end{corollary}
The elliptic regularity estimates of Proposition \ref{prop:ell-reg} can further be refined if the weight $\delta$ is {\em non-exceptional}. More precisely, when considering the Hodge Laplacian, we say that $\delta$ is {\em exceptional} if $\delta\in \{2-n,1-n,-n,\cdots\}\cup \N$, while considering the Hodge-Dirac operator, $\delta$ is exceptional if $\delta\in \{1-n,-n,\cdots\}\cup \N$. The result is as follows: if $\delta$ is non-exceptional for the Hodge Laplacian (resp. for the Hodge-Dirac operator), then there exists $R>0$ such that \eqref{eq:ell-reg-lapl} (resp. \eqref{eq:ell-reg-dir}) can be refined into

\begin{equation}\label{eq:ell-reg-lapl2}
||\omega ||_{k,p,\delta} \leq C(||\Delta \omega||_{k-2,p,\delta-2}+||u||_{L^p(B_R)}),
\end{equation}
resp.

\begin{equation}\label{eq:ell-reg-dir2}
||\omega ||_{k,p,\delta} \leq C(||\mathcal{D}\omega||_{k-1,p,\delta-1}+||u||_{L^p(B_R)}).
\end{equation}
See \cite[Prop. 2.7]{KP}.

Another corollary of these estimates concerns the Fredholmness of $\Delta$ and $\mathcal{D}$:

\begin{corollary}\label{cor:fredholm}

Assume that $M$ is ALE, $p\in (1,\infty)$ and $\delta$ is non exceptional for the Hodge Laplacian (resp. for the Hodge-Dirac operator). Then, 

$$\Delta:W^{2,p}_\delta(\Lambda^*M)\rightarrow W_{\delta-2}^{0,p}(\Lambda^*M),$$
resp.

$$\mathcal{D}:W^{1,p}_\delta(\Lambda^*M)\rightarrow W_{\delta-1}^{0,p}(\Lambda^*M),$$
is a Fredholm operator.

\end{corollary}

\begin{proof}[Sketch of the proof]
The estimates \eqref{eq:ell-reg-lapl2} and \eqref{eq:ell-reg-dir2} together with Rellich compactness imply that $\mathcal{D}$ and $\Delta$ are semi-Fredholm, i.e. have finite dimensional kernel and closed range. Since these two operators are also self-adjoint, using duality of the weighted Sobolev spaces we see that their adjoint is also semi-Fredholm, hence they also have finite dimensional cokernel. Therefore, they are Fredholm. See the proof of \cite[Prop. 1.14]{Bar86} for more details.

\end{proof}

A key property for investigating the decay of harmonic forms is that the behaviour as $r\to\infty$ of functions on $\R^n$ which are harmonic outside a compact set is known; for $\delta\in \R$, we denote by $k_-(\delta)$ the largest exceptional weight $k$ for the Laplacian such that $k\leq \delta$. Then, one has:

\begin{proposition}\label{prop:decay-Rn}

Let $n\geq 3$, and $f$ be a harmonic function on $\R^n\setminus B_R$, $R>0$. Assume that $f=\mathcal{O}_\infty(r^\delta)$. Then, actually $f= \mathcal{O}_\infty(r^{k_-(\delta)})$.

\end{proposition}
Indeed, this follows by decomposing the function $f$ into spherical harmonics, and using the fact that the exceptional weights are precisely the powers of $r$ appearing in this decomposition. Note that the (scalar) components of a harmonic differential form on $\R^n$ are harmonic functions, so Proposition \ref{prop:decay-Rn} also applies to differential forms that are harmonic on $\R^n\setminus B_R$. 

Now, let us prepare for the proof of Proposition \ref{prop:decay_harmonic_forms}; we will need the following two lemmas. The first one is obtained by a minor variation on the proof of \cite[Lemma 4.2]{KP}:

\begin{lemma}\label{lem:decay-harmo-Rn}

Let $n\geq 3$, $\delta<0$ and $\omega\in \Lambda^k(\R^n\setminus B_R)$, $R>0$, $k\in\{0,\cdots,n\}$ be such that $\omega=\mathcal{O}_\infty(r^\delta)$ and $\Delta_{\R^n}\omega=0$. Then,

\begin{enumerate}

\item[(i)] $\omega\in \mathcal{O}_\infty(r^{2-n})$.

\item[(ii)] if $k\in \{1,\cdots,n-1\}$ and if moreover $d_{\R^n}^*\omega\in \mathcal{O}_\infty(r^{1-n-\epsilon})$ for some $\epsilon>0$, then $\omega\in \mathcal{O}_\infty(r^{1-n})$.

\item[(iii)] if $k\in\{2,\cdots,n-2\}$, and if moreover $d\omega,\,d_{\R^n}^*\omega\in \mathcal{O}_\infty(r^{-n-\epsilon})$ for some $\epsilon>0$, then $\omega\in \mathcal{O}_\infty(r^{-n})$.

\end{enumerate}

\end{lemma}

\begin{proof}[Sketch of the proof]

Since $\delta<0$, one has $k_-(\delta)\leq 2-n$, hence by Proposition \ref{prop:decay-Rn}, $\omega\in \mathcal{O}_\infty(r^{2-n})$, which yields (i). If now $k\in\{1,\cdots,n-1\}$, then the proof of (ii) and (iii) follows \cite[Lemma 4.2]{KP}; indeed, the only difference with \cite[Lemma 4.2]{KP} is that $\omega$ is not supposed to be closed and co-closed. However, one easily sees that the same proof applies, given the assumed decay rate of $d\omega,d^*\omega$. Details are left to the interested reader.

\end{proof}
The second lemma, which will be used repeatedly in the proof of Proposition \ref{prop:decay_harmonic_forms}, is the following:

\begin{lemma}\label{lem:iterate}

Let $M$ be ALE to order $\tau>0$, $\delta\in \R$ and $p\in (1,\infty)$. Let $\omega\in \mathrm{ker}_\delta(\Delta)$, then at each end $E_i$, there exists $R>0$, $\epsilon>0$ and $\bar{\omega}\in \Lambda^*((\R^n\setminus B_R)/\Gamma_i)$, $\Delta_{\R^n}\bar{\omega}=0$ on $(\R^n\setminus B_R)/\Gamma_i$, such that

$$\omega=\bar{\omega}+\mathcal{O}_\infty(r^{k_-(\delta)-\epsilon}).$$
Moreover, $\bar{\omega}=\mathcal{O}_\infty(r^{k_-(\delta)})$, and 

$$\omega=\mathcal{O}_\infty(r^{k_-(\delta)}).$$

\end{lemma}

\begin{proof}

The hypothesis that $M$ is ALE to order $\tau>0$ implies by \eqref{eq:extra_decay_L2} that $\Delta_{\R^n}\omega= \mathcal{O}_\infty(r^{\delta-2-\tau})$. Hence, in particular $\Delta_{\R^n}\omega= \mathcal{O}_\infty(r^{\delta-2-\epsilon})$ for any $\varepsilon\in (0,\tau]$. Choose $\epsilon\in [\frac{\tau}{2},\tau]$ such that $\delta-\epsilon$ is non-exceptional; this is possible since the exceptional set is discrete. Since by \cite[Theorem 1.7]{Bar86}, $\Delta_{\R^n_*} : W'^{2,p}_{\delta-\epsilon}\rightarrow W'^{0,p}_{\delta-2-\epsilon}$ is an isomorphism , one can find $\omega_0\in W_{\delta-\epsilon}^{2,p}((\R^n\setminus B_R)/\Gamma_i)$ such that 

$$\Delta_{\R^n}\omega_0=\Delta_{\R^n}\omega$$
in restriction to $(\R^n\setminus B_R)/\Gamma_i$. Elliptic regularity (Proposition \ref{prop:ell-reg}) and Sobolev embeddings (Proposition \ref{prop:Sob}) imply that $\omega_0= \mathcal{O}_\infty(r^{\delta-\epsilon})$, hence $\omega_0=\mathcal{O}_\infty(r^{\delta-\frac{\tau}{2}})$ since $\epsilon\geq \frac{\tau}{2}$. Define $\bar{\omega}=\omega-\omega_0$, then $\Delta_{\R^n}\bar{\omega}=0$ outside $B_R/\Gamma_i$. Clearly, one has $\bar{\omega}\in \mathcal{O}_\infty(r^\delta)$. According to Proposition \ref{prop:decay-Rn}, in fact $\bar{\omega}\in \mathcal{O}_\infty(r^{k_-(\delta)})$. Thus, starting from the assumption that $\omega\in\ker_\delta(\Delta)$, we have obtained the existence of $\bar{\omega}\in \Lambda^*((\R^n\setminus B_R)/\Gamma_i)$, $\Delta_{\R^n}\bar{\omega}=0$ on $(\R^n\setminus B_R)/\Gamma_i$, such that

$$\omega=\bar{\omega}+\mathcal{O}_\infty(r^{\delta-\frac{\tau}{2}}).$$
And furthermore, $\bar{\omega}=\mathcal{O}_\infty(r^{k_-(\delta)})$. This implies that 

$$\omega=\mathcal{O}_\infty(r^\lambda),\quad \lambda=\max(k_-(\delta),\delta-\frac{\tau}{2}).$$
One can now iterate this argument, starting with the new decay rate for $\omega$: for any $\ell\in \N$, this gives a decomposition

$$\omega=\bar{\omega}+\mathcal{O}_\infty(r^{\lambda-\frac{\tau}{2}}),\quad \lambda=\max(k_-(\delta),\delta-(\ell-1)\frac{\tau}{2}),$$
where $\bar{\omega}=\mathcal{O}_\infty(r^{k_-(\delta)})$ is harmonic for the Euclidean Laplacian outside a compact set. Now choose $\ell$ large enough, so that 

$$\delta-\ell\frac{\tau}{2}<k_-(\delta)\leq \delta-(\ell-1)\frac{\tau}{2},$$
then we get for $\epsilon=k_-(\delta)-\delta-\ell\frac{\tau}{2}>0$ that

$$\omega=\bar{\omega}+\mathcal{O}_\infty(r^{k_-(\delta)-\epsilon}),$$
and this concludes the proof of the lemma.

\end{proof}
Now we are ready to give the proof of Proposition \ref{prop:decay_harmonic_forms}.

\begin{proof}[Proof of Proposition \ref{prop:decay_harmonic_forms}]

Let $p\in \ker_{L^p}(\Delta_k)$. Note that $L^p=L^p_\delta$ with $\delta=-\frac{n}{p}$, so $\omega\in \mathrm{ker}_\delta(\Delta)$. According to Lemma \ref{lem:iterate}, we get

$$\omega= \mathcal{O}_\infty(r^{k_-(\delta)}).$$
Now, $\delta<0$, so $k_-(\delta)\leq 2-n$, hence 

$$\omega= \mathcal{O}_\infty(r^{2-n}).$$
This implies that $d\omega,d^*\omega=\mathcal{O}_\infty(r^{1-n})$. Consider for $R>0$ large enough, relatively compact open sets $B_R$ in $M$, which have boundary that identifies using the coordinate systems at infinity in each end to the spheres $\partial B_{\R^n}(0,R)$ in $\R^n$ quotiented by the action of $\Gamma_i$. Integration by parts implies that

\begin{eqnarray*}
0 &=& (\Delta\omega,\omega)_{L^2}\\
&=& \lim_{R\to\infty}\left\{\int_{B_R} [|d\omega|^2+|d^*\omega|^2]\dv+\int_{\partial B_R} [(d^*\omega,\iota_\nu \omega)+(\omega,\iota_\nu d\omega)]\dS\right\}.
\end{eqnarray*}
Given the asymptotics of $\omega,d^*\omega,d\omega$, the boundary integral is $O(R^{2-n})$, hence tends to zero as $R\to\infty$. Therefore, we conclude that

$$0=\lim_{R\to \infty} \int_{B_R}[ |d\omega|^2+|d^*\omega|^2]\dv,$$
so $d\omega=0$ and $d^*\omega=0$. We have thus proved that $\omega$ is closed and co-closed, hence point (a) of the proposition is proved. Now, let us apply once more Lemma \ref{lem:iterate}: we obtain a form $\bar{\omega}$ such that at each end $E_i$, $\Delta_{\R^n}\bar{\omega}=0$ on $(\R^n\setminus B_R)/\Gamma_i$, $\bar{\omega}=\mathcal{O}_\infty(r^{2-n})$, and for some $\epsilon>0$,

\begin{equation}\label{eq:decomp_harmo_decay}
\omega=\bar{\omega}+\mathcal{O}_\infty(r^{2-n-\epsilon}).
\end{equation}
Since $\omega$ is closed, the above equation implies that

$$d\bar{\omega}=\mathcal{O}_\infty(r^{1-n-\epsilon}).$$
Moreover, since $\omega$ is co-closed, one has $d^*\bar{\omega}=\mathcal{O}_\infty(r^{1-n-\epsilon})$. Since $M$ is ALE of order $\tau>0$, equation \eqref{eq:extra_decay_d*2} entails that, up to lowering the value of $\epsilon>0$,

$$d^*_{\R^n}\bar{\omega}=\mathcal{O}_\infty(r^{1-n-\epsilon}).$$
All in all, we have obtained that 

$$d\bar{\omega},\,d^*_{\R^n}\bar{\omega}=\mathcal{O}_\infty(r^{1-n-\epsilon}).$$
According to Lemma \ref{lem:decay-harmo-Rn}, we then obtain that 

$$\bar{\omega}\in \mathcal{O}_\infty(r^{1-n}).$$
Coming back to \eqref{eq:decomp_harmo_decay}, and lowering the value of $\epsilon>0$ if necessary, we obtain that

$$\omega=\mathcal{O}_\infty(r^{2-n-\epsilon}).$$
Applying Lemma \ref{lem:iterate} with this new decay rate, we obtain

$$\omega=\mathcal{O}_\infty(r^{k_-(2-n-\epsilon)}).$$
But $k_-(2-n-\epsilon)\leq 1-n$, therefore

$$\omega=\mathcal{O}_\infty(r^{1-n}).$$
This concludes the proof of point (b) of the proposition. 

We now assume that $k\in \{2,\cdots,n-2\}$; we play the same game as before: by Lemma \ref{lem:iterate}, starting with the decay rate $\omega=\mathcal{O}_\infty(r^{1-n})$, one obtains, instead of \eqref{eq:decomp_harmo_decay},

\begin{equation}\label{eq:decomp_harmo_decay2}
\omega=\bar{\omega}+\mathcal{O}_\infty(r^{1-n-\epsilon}).
\end{equation}
with $\bar{\omega}=\mathcal{O}_\infty(r^{1-n})$ harmonic for the Euclidean Laplacian outside a compact set. One then shows in a similar way as before that

$$d\bar{\omega},\,d^*_{\R^n}\bar{\omega}=\mathcal{O}_\infty(r^{-n-\epsilon}).$$
Now, Lemma \ref{lem:decay-harmo-Rn} yields that

$$\bar{\omega}\in \mathcal{O}_\infty(r^{-n}),$$
and coming back to \eqref{eq:decomp_harmo_decay2}, we arrive to

$$\omega=\mathcal{O}_\infty(r^{1-n-\epsilon}).$$
Applying once more Lemma \ref{lem:iterate} with this new decay rate, we get

$$\omega=\mathcal{O}_\infty(r^{k_-(1-n-\epsilon)}),$$
and noticing finally that $k_-(1-n-\epsilon)\leq -n$, we arrive to

$$\omega=\mathcal{O}_\infty(r^{-n}).$$
This concludes the proof of point (c) of the proposition in the case $k\in\{2,\cdots,n-2\}$.

Finally, one assumes that $k\in \{1,\cdots,n-1\}$ and $p\leq \frac{n}{n-1}$. In this case, $\delta \leq 1-n$. According to Corollary \ref{cor:ell_reg}, one has

$$\omega\in o_\infty(r^{1-n}).$$
Therefore, applying Lemma \ref{lem:iterate}, one finds a form $\bar{\omega}$ such that $\Delta_{\R^n}\bar{\omega}=0$ on $\R^n\setminus B_R$, and

\begin{equation}
\omega=\bar{\omega}+\mathcal{O}_\infty(r^{1-n-\epsilon}).
\end{equation}
Moreover, the decay of $\omega$ implies that

$$\bar{\omega}=o_\infty(r^{1-n}).$$
A variation on Proposition \ref{prop:decay-Rn}, which is left to the reader, implies that
$$\bar{\omega}=\mathcal{O}_\infty(r^{-n})$$
(every harmonic function on $(\R^n\setminus B_R)/\Gamma_i$ which decays faster than $r^{1-n}$ has to decay at least to order $r^{-n}$). The end of the proof is now the same as in the case $k\in \{2,\cdots,n-2\}$. The proof of point (c) of the proposition is now complete.

Let us now prove point (d). Notice that if $p>\frac{n}{n-1}$, then

$$\ker_{1-n}(\Delta_k)\subset \ker_{L^p}(\Delta_k).$$
So, using (b),

$$\ker_{L^p}(\Delta_k)\subset \ker_{1-n}(\Delta_k)\subset \ker_{L^p}(\Delta_k).$$
Hence,

$$\ker_{L^p}(\Delta_k)=\ker_{1-n}(\Delta_k).$$
Finally, concerning point (e), since $\ker_{-n}(\Delta_k)\subset \ker_{1-n}(\Delta_1)$ it is enough to prove that the latter is finite dimensional. But according to point (d), one has

$$\ker_{1-n}(\Delta_k)=\ker_{L^2}(\Delta_k)=\mathcal{H}^k(M),$$
and the dimension of the latter space is equal to the $k$th reduced $L^2$ Betti number of $M$, which is known to be finite for ALE manifolds (see for instance \cite{C1}).

\end{proof}

\begin{proof}[Proof of Lemma \ref{lem:expansion_fn}]

Lemma \ref{lem:iterate} yields a decomposition $u=u_0+u_1$, where $\Delta^{\R^n}u_0=0$, $u_0$ is bounded, and $u_1\in \mathcal{O}_{\infty}(r^{-\epsilon})$, $\epsilon>0$. It is well-known that a bounded function $v$ which is harmonic outside a compact set of $\R^n$ has a limit at infinity. Let us recall briefly a proof of this fact; let $w:=v-\Delta^{-1}(\Delta v)$, then $w$ is harmonic and bounded on $\R^n$. Let us consider $\tilde{w}:=w+c$, where the constant $c$ is chosen so that $\tilde{w}$ is non-negative and the infimum of $\tilde{w}$ (which is attained at infinity) if zero. We claim that $\tilde{w}$ is zero everywhere, which implies the following representation formula for $v$:

$$v=-c+\Delta^{-1}(\Delta v),$$
and using the expression of the Green operator in $\R^n$ and the fact that $\Delta v$ has compact support, it follows that $v$ tends to $c$ at infinity. Thus, let us come back to $\tilde{w}$ and prove the claim. Denote $A_R$ the annulus $B(0,2R)\setminus B(0,R)$. Then, $A_R$ can be covered by a number $C(n)$ of balls of radius $R$. The Harnack inequality for these balls implies that the annuli $A_R$ satisfy too a Harnack inequality, with a constant independant of $R$. Hence, there is a constant $C>0$ such that for every $R>0$,

$$\sup_{A_R}\tilde{w}\leq C \inf_{A_R}\tilde{w}.$$
But the right-hand side tends to $0$ as $R\to\infty$, hence the left-hand side as well. However, thanks to the maximum principle,

$$\max_{\R^n}\tilde{w}=\lim_{R\to \infty} \max_{A_R}\tilde{w},$$
and we conclude that the maximum of $\tilde{w}$ is zero. Since $\tilde{w}$ is non-negative, it follows that $\tilde{w}$ is identically zero, which achieves the proof of the claim.

Coming back to $u_0$, and applying the result of \cite[Lemma 4.1]{KP}, one thus obtains the existence of constants $c_i$ and $A_i$, $i=1,\cdots,N$ such that as $r\to \infty$ in the end $E_i$,

$$u_0=c_i+A_ir^{2-n}+\mathcal{O}_\infty(r^{1-n}),\quad r\to\infty.$$
Applying iteratively Lemma \ref{lem:iterate} to $u-c_i$ in each end $E_i$, one can in fact obtain a decomposition $u=u_0+u_1$ with $u_0$ satisfying the above asymptotics in each end (with constants $A_i$ that may be different), and $u_1\in \mathcal{O}_\infty(r^{-\lambda})$, $\lambda=\max(2-n,-\ell\frac{\tau}{2})$, $\ell\in \N$. Taking $\ell$ large enough so that $\ell\frac{\tau}{2}>n-2$, we obtain $u-c_i=\mathcal{O}_\infty(r^{2-n})$ in the end $E_i$. Applying one last time Lemma \ref{lem:iterate} gives $u=u_0+u_1$ with $u_0$ as above, and $u_1=\mathcal{O}_\infty(r^{2-n-\frac{\tau}{2}})$. The lemma is thus proved, with the choice $\epsilon=\frac{\tau}{2}>0$.

\end{proof}

\begin{proof}[Proof of Lemma \ref{lem:expansion_forms}]
Because $\omega\in\ker_{-\alpha}(\Delta_1)$, we have $\omega\in L^p$ for each $p> \frac{\alpha}{n}$. By Proposition 
\ref{prop:decay_harmonic_forms},
$\omega\in\ker_{1-n}(\Delta_1)$ 
 and $\omega$ is closed and co-closed.
By Lemma \ref{lem:iterate},
one finds a form $\bar{\omega}\in\mathcal{O}_\infty(r^{1-n})$ such that at each end $E_i$, we have $\Delta_{\R^n}\bar{\omega}=0$ on $(\R^n\setminus B_R)/\Gamma_i$, and
\begin{equation}
\omega=\bar{\omega}+\mathcal{O}_\infty(r^{1-n-\epsilon}).
\end{equation}
Moreover, $d\bar{\omega}=0$ and according to \eqref{eq:extra_decay_d*2}, up to lowering the value of $\epsilon>0$, $d_{\R^n}^*\bar{\omega}= \mathcal{O}_\infty(r^{-n-\epsilon})$.
Because $\bar{\omega}$ is a one-form, we have 
following the proof of 
\cite[Lemma 4.2]{KP} that
\begin{equation*}\bar{\omega}=A_id(r^{2-n})+\mathcal{O}(r^{-n})=(n-2)A_i r^{1-n}dr+\mathcal{O}(r^{-n}),
\end{equation*}
 for some constant $A_i\in\R$. Letting $B_i=(n-2)A_i$, we get that in each end $E_i$, as $r\to\infty$,
 
 $$\omega=B_ir^{1-n}dr+\mathcal{O}_\infty(r^{1-n-\epsilon}),$$
which concludes the proof of the lemma.
\end{proof}
The rest of this section consists of {\em new} material, which for reasons of clarity of the exposition we choose to present here. It is devoted to the proof of Proposition \ref{prop:closed_subspace}. We start with the following new result, which is a variation on point (a) of Proposition \ref{prop:decay_harmonic_forms}:

\begin{Lem}\label{lem:bounded_kernel}

Let $\delta<1$ and $\omega\in \mathrm{ker}_\delta(\mathcal{D})$; then, $\omega$ is $\mathcal{O}_\infty(1)$, and is closed and co-closed.

\end{Lem}
Note that if $-n<\delta<0$, one can write $\delta=-\frac{n}{p}$ for some $p\in (1,+\infty)$; since $\Delta\omega=\mathcal{D}^2\omega=0$, one has thus that $\ker_\delta(\mathcal{D})\subset \ker_{L^p}(\Delta)$, and the result in this case for forms of degree $\neq 0$ or $n$ is then a consequence of Proposition \ref{prop:decay_harmonic_forms}, point (a); the result for any $\delta<0$ and under the same restriction on the degree follows. The real improvement here is that one can allow $\delta$ to be (slightly) positive, at the expense of the stronger assumption $\mathcal{D}\omega=0$ instead of just $\Delta\omega=0$.

\begin{proof}

Let us first prove that $\omega\in \ker_0(\mathcal{D})$, which follows from the iterative argument already used in the proof of Proposition \ref{prop:decay_harmonic_forms}. First, since $\mathcal{D}^2=\Delta$, we have $\omega\in \ker_\delta(\Delta)$. Lemma \ref{lem:iterate} implies that

$$\omega = \mathcal{O}_\infty(r^{k_-(\delta)}).$$
Since $\delta<1$, one has $k_-(\delta)\leq 0$, therefore

$$\omega= \mathcal{O}_\infty(1).$$
Since $d$ and $d^*$ commute with $\Delta$, and $\Delta\omega=0$, we conclude that

$$d\omega,\,d^*\omega\in \ker_{-1}(\Delta).$$
Write

$$d^*\omega=\sum_{j=0}^{n-1}\theta_j,\quad \theta_j\in \Lambda^jM,$$
and

$$d\omega=\sum_{i=1}^{n} \eta_i,\quad \eta_i\in \Lambda^iM.$$
Obviously, since $\Delta$ preserves the degree of differential forms, we have for any $i$ and $j$,

$$\eta_i,\,\theta_j\in \ker_{-1}(\Delta).$$
By the maximum principle for the scalar Laplacian, a decaying harmonic form of degree $0$ or $n$ is necessarily identically zero. Therefore, $\eta_n=0$ and $\theta_0=0$. Thus, since $\ker_{-1}(\Delta)\subset \ker_{L^p}(\Delta)$ for any $p>n$, Proposition \ref{prop:decay_harmonic_forms} point (b) implies that $d\omega,\,d^*\omega\in \ker_{1-n}(\Delta)$. We are going to prove, using some results from Section 4, that actually, for any $i\neq 1$ and $j\neq n-1$, $\eta_i,\,\theta_j\in \ker_{-n}(\Delta)$. First, notice that Proposition \ref{prop:decay_harmonic_forms} point (c) implies that any $\theta_j$, $\eta_i$ with $i,j\in\{2,\cdots,n-2\}$ belongs to $\ker_{-n}(\Delta)$. It thus remains to prove that $\theta_{1}$ and $\eta_{n-1}$ belong to $\ker_{-n}(\Delta)$. As will be apparent, the proofs of these two facts are completely similar one to another, so we do it only for $\theta_1$. Decompose

$$\omega=\sum_{k=0}^n\omega_k,$$
then it follows from degree considerations that

\begin{equation}\label{eq:theta_n-1}
\theta_{n-1}=d^*\omega_n.
\end{equation}
According to Corollary \ref{cor:decomp_kernel_1}, there exists $h\in \ker_0(\Delta_0)$ such that

$$\theta_1=dh+\mathcal{O}_\infty(r^{-n}).$$
Let 

$$\bar{\omega}:=\omega-\omega_n,$$
and define $\bar{\theta}_j$, $\bar{\eta}_i$ analogously to $\theta_j$, $\eta_i$ by decomposing $d^*\bar{\omega}$ and $d\bar{\omega}$ according to the degrees. Then, by \eqref{eq:theta_n-1},

$$\bar{\theta}_j=\theta_j,\quad j\neq n-1,$$
and

$$\bar{\theta}_{n-1}=0.$$
Therefore,

$$d^*\bar{\omega}=dh+\mathcal{O}_\infty(r^{-n}).$$
Taking the Hodge star of this identity, we get

\begin{equation}\label{eq:exact_dh}
*d^*\bar{\omega}=*dh+\mathcal{O}_\infty(r^{-n}).
\end{equation}
But using the well-known facts that $d^*=\pm *d*$ and $*^2=\pm id$ in restriction to forms of a given degree (where the signs depend on the particular degree), we see that $*d^*\bar{\omega}$ is a sum of exact forms (of various degrees). Integrate the degree $n-1$ component of \eqref{eq:exact_dh} over each Euclidean sphere $S_i(R)$ of radius $R$ and center $0$ in the quotient space $\R^n/\Gamma_i$, at each end $E_i$ of $M$. By Stokes formula and the fact that the left-hand side of \eqref{eq:exact_dh} is exact, we obtain for each sphere that
$$\int_{S_i(R)} *dh=O(\frac{1}{R}),$$
as $R\to\infty$. Furthermore, Lemma \ref{lem:expansion_fn} gives the following expansion of $h$ in each end $E_i$:

$$h=c_i+A_ir^{2-n}+\mathcal{O}_\infty(r^{2-n-\epsilon}),\quad c_i\in \R,\,\epsilon>0,$$
so

$$dh=(2-n)A_ir^{1-n}dr+\mathcal{O}_\infty(r^{1-n-\epsilon}).$$
From this expansion, it is easy to see that

$$\int_{S_i(R)} *dh=(2-n)A_i\mathrm{Vol}(S^{n-1}/\Gamma_i)+o(1),$$
as $R\to \infty$. Therefore, we conclude that for any end $E_i$, one has $A_i=0$. By (ii) in Lemma \ref{lem:Ai}, this implies that $h$ is constant, and so $dh=0$. Therefore, since $\theta_1=dh+\mathcal{O}_\infty(r^{-n})$, we conclude that

$$\theta_1\in\ker_{-n}(\Delta).$$
As indicated above, the proof that $\eta_{n-1}\in\ker_{-n}(\Delta)$ is completely similar and will be skipped. Define now

$$\tilde{\omega}:=\omega-\omega_0-\omega_n,$$
and $\tilde{\theta}_j$, $\tilde{\eta}_i$ the associated forms in the decomposition of $d^*\tilde{\omega}$, $d\tilde{\omega}$ according to the degrees. We have

$$d^*\tilde{\omega}=d^*\omega-d^*\omega_n,$$
and

$$d\tilde{\omega}=d\omega-d\omega_0,$$
which implies that

$$\tilde{\theta}_j=\theta_j,\quad j\neq n-1,$$

$$\tilde{\theta}_{n-1}=0,$$

$$\tilde{\eta}_i=\eta_i,\quad i\neq 1,$$

$$\tilde{\eta}_1=0.$$
It follows that $d\tilde{\omega},\,d^*\tilde{\omega}\in \ker_{-n}(\Delta)$. Moreover, one also has $\tilde{\omega}\in \ker_0(\Delta)$. Let us see that this implies $d\tilde{\omega}=0$, $d^*\tilde{\omega}=0$. The argument for this is similar to the one already used in the proof of Proposition \ref{prop:decay_harmonic_forms}: consider for $R>0$ large enough, relatively compact open sets $B_R$ in $M$, which have boundary that identifies using the coordinate systems at infinity to the spheres $\partial B_{\R^n}(0,R)$ in $\R^n$ quotiented by the action of $\Gamma_i$. Integration by parts imply that

\begin{eqnarray*}
0 &=& (\Delta\tilde{\omega},\tilde{\omega})_{L^2}\\
&=&\lim_{R\to\infty}\left\{ \int_{B_R}[ |d\tilde{\omega}|^2+|d^*\tilde{\omega}|^2]\dv+\int_{\partial B_R} [(d^*\tilde{\omega},\iota_\nu \tilde{\omega})+(\tilde{\omega},\iota_\nu d\tilde{\omega})]\dS\right\}.
\end{eqnarray*}
Given the asymptotics of $\tilde{\omega},d^*\tilde{\omega},d\tilde{\omega}$, the boundary integral is $O(\frac{1}{R})$, hence tends to zero, as $R\to\infty$. Therefore, we conclude that

$$0=\lim_{R\to \infty} \int_{B_R}[ |d\tilde{\omega}|^2+|d^*\tilde{\omega}|^2]\dv,$$
so $d\tilde{\omega}=0$ and $d^*\tilde{\omega}=0$. Let us now come back to $\omega$. We obtain

$$d\omega=d\omega_0,\quad d^*\omega=d^*\omega_n,$$
However, by assumption one also has $0=\mathcal{D}\omega=(d+d^*)\omega=d\omega_0+d^*\omega_n$. But $d\omega_0$ and $d^*\omega_n$ are of degree respectively $1$ and $n-1$, and since $n\geq 3$ one has $1\neq n-1$. Therefore, one concludes that $d\omega_0=0$ and $d^*\omega_n=0$. Hence, $d\omega=0$ and $d^*\omega=0$, and this concludes the proof.

\end{proof}
We are now ready for the proof of Proposition \ref{prop:closed_subspace}.

\begin{proof}[Proof of Proposition \ref{prop:closed_subspace}]

Let $\omega\in  \im_{L^p}(d_{k-1})+ \im_{L^p}(d^*_{k+1})$,  and consider sequences  $\alpha_i\in C^{\infty}_{c}(\Lambda^{k-1}M)$, $\beta_i\in C^{\infty}_{c}(\Lambda^{k+1}M)$, $i\in\N$  such that we have $d\alpha_i+d^*\beta_i\to\eta$ in $L^p$.
 Let $\eta \in \ker_{L^q}(\Delta_{k})$.
  By integration by parts and Proposition \ref{prop:decay_harmonic_forms},
	\begin{align*}
	 (d\alpha_i,\eta)_{L^2}=(\alpha_i,d^*\eta)_{L^2}=0,\qquad (d^*\beta_i,\eta)=(\beta_i,d\eta)=0.
	 \end{align*}
	 Therefore, by taking $i\to\infty$, we have
	 $(\omega,\eta)_{L^2}=0$, so
	  \begin{align*} \im_{L^p}(d_{k-1})+ \im_{L^p}(d^*_{k+1})\subset \mathrm{Ann}_{L^p}(\ker_{L^q}(\Delta_k)).
	  \end{align*}
	  To prove the converse inclusion, we consider the Hodge-Dirac operator $\mathcal{D}=d+d^*: C^{\infty}(\Lambda^*M)\to C^{\infty}(\Lambda^*M)$ on the whole exterior algebra, and denote
	  
$$\mathrm{im}_{L^p}(\mathcal{D})=\overline{\mathcal{D}(C_0^\infty(\Lambda^*M))}^{L^p}.$$	  
Because the operator $\mathcal{D}$ is self-adjoint, 

$$\ker_{L^q}(\mathcal{D})=\mathrm{Ann}_{L^q}(\im_{L^p}(\mathcal{D})),$$
so by reflexivity of $L^p$,

$$\im_{L^p}(\mathcal{D})=\mathrm{Ann}_{L^p}\left(\mathrm{Ann}_{L^q}(\im_{L^p}(\mathcal{D}))\right)=\mathrm{Ann}_{L^p}(\ker_{L^q}(\mathcal{D})).$$
But according to Proposition \ref{prop:decay_harmonic_forms}, 

$$\ker_{L^q}(\mathcal{D})=\ker_{L^q}(\Delta),$$
therefore

$$\im_{L^p}(\mathcal{D})=\mathrm{Ann}_{L^p}(\ker_{L^q}(\Delta)).$$
We now regard $\mathcal{D}$ as an operator between weighted Sobolev spaces

\begin{align*}
\mathcal{D}: W^{1,p}_{\delta}(\Lambda^*M)\to L^p_{\delta-1}(\Lambda^*M) 
\end{align*}
and by Corollary \ref{cor:fredholm}, this operator is Fredholm if $\delta$ is non-exceptional, that is, if $\delta\notin\left\{0,1,2,\ldots\right\}\cup\left\{1-n,-n,-n-1,\ldots\right\}$. Choose now $\delta-1=-\frac{n}{p}$, so that
\begin{align*}
L^p_{\delta-1}(\Lambda^*M)=L^p(\Lambda^*M).
\end{align*}
Observe also that $\delta=1-\frac{n}{p}$ is non-exceptional as long as $p\in (1,\infty)$, $p\neq n$. 
Therefore, 
\begin{align*}
\mathcal{D}(W^{1,p}_{1-\frac{n}{p}}(\Lambda^*M))\subset L^p(\Lambda^*M)
\end{align*}
is a closed subspace. But by definition,

$$\im_{L^p}(\mathcal{D})\subset \overline{\mathcal{D}(W^{1,p}_{1-\frac{n}{p}}(\Lambda^*M))}^{L^p},$$
so the closeness of the image implies that we have the inclusion
\begin{align*}
\im_{L^p}(\mathcal{D})\subset \mathcal{D}(W^{1,p}_{1-\frac{n}{p}}(\Lambda^*M)).
\end{align*}
Furthermore, because
\begin{align*}
\ker_{W^{1,p}_{1-\frac{n}{p}}}(\mathcal{D})
\end{align*}
is finite-dimensional, one can find a closed space $V\subset W^{1,p}_{1-\frac{n}{p}}(\Lambda^*M)$ such that
\begin{align*}
W^{1,p}_{1-\frac{n}{p}}(\Lambda^*M)=V\oplus \ker_{W^{1,p}_{1-\frac{n}{p}}}(D),
\end{align*}
and furthermore, $\mathcal{D}$ restricts to an isomorphism
\begin{align}\label{eq:Hodge_iso}
\mathcal{D}|_V:V\to \mathcal{D}(W^{1,p}_{1-\frac{n}{p}}(\Lambda^*M))\subset L^p(\Lambda^*M).
\end{align}
Now let
\begin{align*}
\omega\in \mathrm{Ann}_{L^p}(\ker_{L^q}(\Delta_k))\subset 
L^p(\Lambda^kM)
\end{align*}
If we consider $\omega$ as sitting in the whole exterior algebra, we have
\begin{align*}
\omega\in \mathrm{Ann}_{L^p}(\ker_{L^q}(\mathcal{D}))= \im_{L^p}(\mathcal{D}).
\end{align*}
Let $\{\omega_i\}_{i\in\N}\in C^{\infty}_0(\Lambda^*M)$ be a sequence such that $\mathcal{D}\omega_i\to\omega$ in $L^p$. Because 
\begin{align*}
\mathcal{D}\omega_i\in \im_{L^p}(\mathcal{D})\subset \mathcal{D}(W^{1,p}_{1-\frac{n}{p}}(\Lambda^*M)),
 \end{align*}we have by \eqref{eq:Hodge_iso} unique forms $\tilde{\omega}_i\in V$ such that $\mathcal{D}\tilde{\omega}_i=\mathcal{D}\omega_i$. Because $\tilde{\omega}_i-\omega_i$ belongs to $\mathrm{ker}_\delta(\mathcal{D})$ and $\delta=1-\frac{n}{p}<1$, Lemma \ref{lem:bounded_kernel} implies that it is in fact bounded, and is closed and co-closed. So, we get 
\begin{align*}
d\omega_i=d\tilde{\omega_i},\qquad d^*\omega_i=d^*\tilde{\omega}_i.
\end{align*}
Now, because
 $\mathcal{D}\omega_i=\mathcal{D}\tilde{\omega}_i\to\omega$ in $L^p$, $\mathcal{D}\tilde{\omega}_i$
  is a Cauchy sequence, hence converges, in $L^p=L^p_{-\frac{n}{p}}$. By \eqref{eq:Hodge_iso}, $\tilde{\omega}_i$ is thus a Cauchy sequence in $W^{1,p}_{1-\frac{n}{p}}$. Denote by $\omega_\infty$ its limit. We claim that the following estimate holds: for every $\alpha\in C^\infty(\Lambda^*M)$,
  
\begin{equation}\label{eq:semi-norm}
||\nabla \alpha||_p\lesssim ||\alpha||_{W^{1,p}_{1-\frac{n}{p}}}.
\end{equation}
Indeed, the estimate is trivial locally, and thus it is enough to prove it in charts at infinity. We can assume that $\alpha$ is a $k$-form for some $k\in \{1,\cdots,n\}$. Note then that the fact that the metric is ALE to order $\tau$ implies that the Christoffel symbols in a chart at infinity satisfy

$$\Gamma_{ij}^k=\mathcal{O}_\infty (\sigma^{-1-\tau}),\quad i,j,k\in\{1,\cdots,n\}$$
(recall that $\sigma=(1+r^2)^{1/2}$).
 Thus, writing

$$\nabla_{\partial_i}\alpha=\partial_k\alpha+\sum_{1\leq i_1,\cdots,i_k\leq n}\alpha_{i_1\cdots i_k} \sum_{\ell=1}^n\sum_{s=1}^k(-1)^\ell\Gamma_{si}^\ell dx_\ell\wedge dx_{i_1}\wedge\cdots \wedge  \widehat{dx_{i_s}}\wedge \cdots \wedge dx_{i_k},$$
we conclude that

$$||\nabla \alpha||_p\lesssim \sum_{|\alpha|=1}||\partial^\alpha \alpha||_p+||\sigma^{-\tau-1}\alpha||_p.$$
It follows by definition of the weighted Sobolev norm that

$$\sum_{|\gamma|=1}||\partial^\gamma \alpha||_p+||\sigma^{-\tau-1}\alpha||_p\leq \sum_{|\gamma|=1}||\partial^\gamma \alpha||_p+||\sigma^{-1}\alpha||_p\equiv ||\alpha||_{W^{1,p}_{1-\frac{n}{p}}},$$
and \eqref{eq:semi-norm} follows.
 
The estimate \eqref{eq:semi-norm} implies $\nabla{\tilde{\omega}}_i \to \nabla \omega_\infty$ in $L^p$, hence the following sequences converge strongly in $L^p$:
\begin{align*}
d\omega_i=d\tilde{\omega_i} \to d\omega_\infty:= \eta,\qquad d^*\omega_i=d^*\tilde{\omega}_i\to d^*\omega_\infty:=\xi,
\end{align*}
in particular $\eta,\xi\in L^p(\Lambda^*M)$. Since $\mathcal{D}\omega_i\to \omega$ and $\mathcal{D}\omega_i\to \eta+\xi$ in $L^p$, we get by uniqueness of the limit that $\eta+\xi=\omega$ a.e. Decompose $\omega_i=\sum_{k=0}^n\omega_i^k$
where $\omega_i^k\in L^p(\Lambda^kM)$. Then we also have convergence for the $L^p$-sequences
\begin{align*}
d_k\omega_i^k\to \eta^{k+1}\in L^p(\Lambda^{k+1}M) 
\qquad 
d^*_k\omega_i^k\to \xi^{k-1}\in L^p(\Lambda^{k-1}M)
\end{align*}
and in particular,
\begin{align*}
d_{k-1}\omega_i^{k-1}+d^*_{k+1}\omega_i^{k+1}\to \eta^k+\xi^k=\omega\in L^p(\Lambda^{k}M)
\end{align*}
so that
\begin{align*}
\omega\in \im_{L^p}(d_{k-1})+ \im_{L^p}(d^*_{k+1}).
\end{align*}
This finishes the proof of the proposition.
\end{proof}
Actually, Proposition \ref{prop:closed_subspace} and its proof imply the following Corollary which is used in Section \ref{Section:Hodge-Sob}:

\begin{Cor}\label{Cor:Hodge-Sob}

Assume that $p\neq n$, then every $\omega\in \im_{L^p}(d)+\im_{L^p}(d^*)$ writes uniquely 

$$\omega=d\alpha+d^*\beta,$$
with $\alpha\in \dot{W}^{1,p}(\Lambda^{k-1}M)$, $\beta\in \dot{W}^{1,p}(\Lambda^{k-1}M)$, and moreover

$$||\alpha||_{\dot{W}^{1,p}}\lesssim ||\omega||_p,\quad ||\beta||_{\dot{W}^{1,p}}\lesssim ||\omega||_p.$$

\end{Cor}

\end{document}